\documentclass[11 pt]{amsart}
\usepackage[utf8]{inputenc}

\usepackage[normalem]{ulem}

\usepackage{mystyle}
\hypersetup{
  pdftitle={Counting hypergraphs without linear cycles of fixed length},
  pdfauthor={Jozsef Balogh, Ramon I. Garcia, and Abhishek Methuku}
}
\calclayout

\usepackage[backend=biber, style=trad-abbrv, sorting=nyt, maxnames=100,backref=true, giveninits=true, sortcites=true, url=false, doi=false]{biblatex}
\usepackage{amsmath}
\usepackage{amsthm}
\usepackage{amssymb}
\usepackage{mathtools}
\usepackage[dvipsnames]{xcolor}
\usepackage{enumitem}
\usepackage{cleveref}
\usepackage{tikz}
\usetikzlibrary {graphs} 
\usepackage{caption}

\def\HH{\mathcal{H}}
\def\KK{\mathcal{K}}

\title{Counting hypergraphs without linear cycles of fixed length}

\author{
J\'ozsef Balogh}
\email{jobal@illinois.edu}

\author{
Ramon I. Garcia
}
\email{rgarcia2@iastate.edu}

\author{Abhishek Methuku}
\email{abhishekmethuku@gmail.com}

\thanks{J\'ozsef Balogh received funding from NSF grants RTG DMS-1937241,  FRG DMS-2152488, 
and DMS-2554179, the UIUC Campus Research Board Award RB26026, a Simons Collaboration grant, and the Institute for Basic Science (IBS-R029-C4). Ramon I. Garcia received funding from NSF grants RTG DMS-1937241, FRG DMS-2152488, and the UIUC R.H. Schark Fellowship. Abhishek Methuku received funding from Simons award SFI-MPS-TSM-00025583 and the UIUC Campus Research Board Award RB25050.}
\date{\today}

\begin{document}

\begin{abstract} Let $C_{k}^{(r)}$ be the $r$-uniform linear cycle on $k$ hyperedges. An $r$-graph is $C_{k}^{(r)}$-\emph{free} if it contains no copy of $C_{k}^{(r)}$. Let $\operatorname{ex}_{r}(n,C_{k}^{(r)})$ denote the maximum number of hyperedges in an $n$-vertex $C_{k}^{(r)}$-free $r$-graph. Balogh, Narayanan and Skokan asked whether, for every pair of integers $r,k\ge 3$, the number of $C_k^{(r)}$-free $r$-graphs on $n$ labelled vertices is
\[
2^{(1+o(1))\operatorname{ex}_{r}(n,C_{k}^{(r)})}.
\]
While the analogous statement is known to fail for graphs ($r=2$), by
a construction of Morris and Saxton, the general question remained
open for hypergraphs.  Very recently, Jiang and Longbrake answered the
question affirmatively when $r\geq5$.

In this paper, we completely resolve the problem by proving that the
answer is affirmative for every pair of integers $r,k\geq3$.  When
$(r,k)\neq(3,3)$, we establish balanced
supersaturation results for linear cycles and combine them with the hypergraph container
method.  For $(r,k)=(3,3)$, we instead decompose $3$-graphs according
to their pair-codegrees, encode the subhypergraph formed by the hyperedges
that contain a large codegree pair as a directed graph, and apply a
multicolour entropy theorem.
\end{abstract}
\maketitle

\section{Introduction}

\subsection{Background} For a positive integer $n$, we denote by $[n]$ the set $\{1,2,\ldots, n\}$. An $r$-\emph{graph} is a collection of sets of size $r$.
Given $r$-graphs $F$ and $H$, we say $H$ is $F$-\emph{free} if it does not contain a copy of $F$, regardless of whether the copy is induced or not. For an $r$-graph $F$ and an integer $n$, we denote by $\textup{ex}_{r}(n,F)$ the maximum number of hyperedges in an $n$-vertex $F$-free $r$-graph. 

For integers $r\geq2$ and $k\geq3$, the \emph{linear cycle}
$C_k^{(r)}$ is an $r$-graph consisting of hyperedges
$e_1,\ldots,e_k$ and distinct vertices $v_1,\ldots,v_k$ such that,
with indices taken modulo $k$,
\[
 e_i\cap e_{i+1}=\{v_{i+1}\}
 \quad\text{and}\quad
 e_i\cap e_j=\emptyset
 \quad\text{whenever }j\notin\{i-1,i,i+1\}.
\]
The requirement that $v_1,\ldots,v_k$ are distinct is essential when
$k=3$, since it excludes three hyperedges having a common vertex.  The
\emph{linear path} $P_k^{(r)}$ is defined similarly: it consists of
hyperedges $e_1,\ldots,e_k$ such that
$|e_i\cap e_{i+1}|=1$ for every $i\in[k-1]$, while
$e_i\cap e_j=\emptyset$ whenever $|i-j|>1$.

Let $L\subseteq[n]$ have size $\lfloor(k-1)/2\rfloor$, and let
$S_L^{(r)}(n)$ consist of all
$e\in\binom{[n]}r$ satisfying $e\cap L\neq\emptyset$.  This $r$-graph
is $C_k^{(r)}$-free.  Indeed, if it contained a linear cycle, then
each of the $k$ hyperedges of the cycle would intersect $L$, whereas
every vertex of $L$ could belong to at most two hyperedges of the cycle.
This would give
$k\leq2|L|=2\lfloor(k-1)/2\rfloor<k$, a contradiction.  Moreover,
\[
 |S_L^{(r)}(n)|
 =\binom nr-\binom{n-|L|}r
 =\bigl(|L|+o(1)\bigr)\binom{n}{r-1}.
\]
Frankl and F\"uredi~\cite{FranklFuredi1987} proved that, for $k=3$ and
every $r\geq3$, the construction $S_L^{(r)}(n)$ is extremal for every
sufficiently large $n$.  F\"uredi and
Jiang~\cite{Furedi2014LinearCycles} subsequently determined the exact
Tur\'an numbers for every $k\geq3$ and $r\geq5$, and Kostochka, Mubayi
and Verstra\"ete~\cite{Kostochka2015} determined them for every
$k\geq4$ and $r\geq3$.  For odd
$k\geq5$, the construction $S_L^{(r)}(n)$ is extremal.  For even
$k\geq4$, the extremal constructions are obtained by adjoining to
$S_L^{(r)}(n)$ an explicitly characterized $r$-graph $F$ supported on
$[n]\setminus L$, where $|F|=\Theta(n^{r-2})$; the case
$(r,k)=(3,4)$ requires a different choice of $F$.
Consequently, for every fixed $r,k\geq3$,
\[
 \textup{ex}_r(n,C_k^{(r)})
 =
 \left(\left\lfloor\frac{k-1}{2}\right\rfloor+o(1)\right)
 \binom{n}{r-1}.
\]

More generally, once an asymptotic estimate for
$\textup{ex}_r(n,F)$ has been determined, it is natural to ask how many
copies of $F$ must occur in an $r$-graph with at least
$(1+\varepsilon)\textup{ex}_r(n,F)$ hyperedges.  Results of this kind
are known as \emph{supersaturation} results and form a classical theme
in extremal combinatorics.  For instance, Erd\H{o}s and
Simonovits~\cite{Erdos1983Supersaturated} established general
supersaturation results for graphs and hypergraphs, while
Mubayi~\cite{Mubayi2013Counting} obtained sharp estimates for several
specific hypergraphs.  Supersaturation has also been studied for
linear cycles.  Jiang and Yepremyan~\cite{Jiang2020Supersaturation}
considered linear cycles of even length when the host hypergraph is
itself linear, Mubayi and Solymosi~\cite{Mubayi2026Pentagons} counted
pentagons in linear triple systems, and Deng, Han, Nie and
Spiro~\cite{deng2025supersaturation} obtained supersaturation results
for odd linear cycles in $r$-graphs with $n^{r-\delta}$ hyperedges,
where $0<\delta<1$ is a fixed constant.

For applications of the hypergraph container method, one often needs
a \emph{balanced supersaturation} result in which the copies are
distributed so that no prescribed collection of hyperedges belongs to
too many of them.  Such results have been established for several
degenerate graphs and hypergraphs; see, for example, the work of
Corsten and Tran~\cite{Corsten2021Balanced}.  When combined with the
container method, balanced supersaturation can be used both to count
$F$-free hypergraphs and to prove random analogues of extremal
theorems.  The container method has also been used to count
hypergraphs of large girth.  Balogh and
Li~\cite{Balogh2020Linear} studied the number of linear uniform
hypergraphs of large girth, while Spiro and
Verstra\"ete~\cite{Spiro2022} considered general uniform hypergraphs
of large Berge girth.  In the setting of linear cycles, Mubayi and
Yepremyan~\cite{mubayi2023random}, and independently
Nie~\cite{Nie2024Cycles}, established random variants of the extremal
result for $C_k^{(r)}$ when $k$ is even.  Nie and
Spiro~\cite{nie2024random} proved random versions of extremal results
for expansions of hypergraphs.

We now turn to the counting problem studied in this paper.  For an
$r$-graph $F$, let $\textup{forb}_r(n,F)$ denote the number of
$F$-free $r$-graphs with vertex set $[n]$.  Every subhypergraph of an
extremal $C_k^{(r)}$-free $r$-graph is also $C_k^{(r)}$-free, and
hence
\[
 \textup{forb}_r(n,C_k^{(r)})
 \geq 2^{\textup{ex}_r(n,C_k^{(r)})}.
\]
Conversely, every $C_k^{(r)}$-free $r$-graph has at most
$\textup{ex}_r(n,C_k^{(r)})$ hyperedges.  Therefore,
\[
 \textup{forb}_r(n,C_k^{(r)})
 \leq \sum_{i=0}^{\textup{ex}_r(n,C_k^{(r)})}
          \binom{\binom nr}{i}
 =2^{O(n^{r-1}\log n)}.
\]

In the graph case ($r=2$), Erd\H{o}s, Frankl and
R\"odl~\cite{ErdosFranklRodl1986} proved that, for every fixed
non-bipartite graph $F$,
$
 \textup{forb}_2(n,F)
 =2^{(1+o(1))\textup{ex}_2(n,F)}
$
as $n\to\infty$. They also suggested that the same formula
should hold for every graph $F$ containing a cycle; this conjecture is
often attributed to Erd\H{o}s (see, for example, Balogh, Bollob\'as
and Simonovits~\cite{BaloghBollobasSimonovits2011}).  Morris and
Saxton~\cite{MorrisCycles2016} proved that
$\textup{forb}_2(n,C_{2\ell})\leq2^{O(n^{1+1/\ell})}$ for every
$\ell\geq2$, but disproved Erd\H{o}s's conjecture by showing that there
exists a constant $c>0$ such that, for infinitely many $n$,
\begin{equation}
\label{eq:Erdosdisproof}
 \textup{forb}_2(n,C_6)
 \geq 2^{(1+c)\textup{ex}_2(n,C_6)}.
\end{equation}
This shows that the leading
constant in the exponent need not be $1$ in the graph case.

For hypergraphs, Nagle, R\"odl and
Schacht~\cite{NagleRodlSchacht2006} proved that, for every fixed
$r$-graph $F$,
\begin{equation}
\label{Nagle-Rodl-Schacht}
 \textup{forb}_r(n,F)
 \leq 2^{\textup{ex}_r(n,F)+o(n^r)}.
\end{equation}
Recall that an $r$-graph is $r$-partite if its vertex set can be
partitioned into $r$ classes so that every hyperedge contains exactly
one vertex from each class.  If $F$ is not $r$-partite, then
$\textup{ex}_r(n,F)=\Theta(n^r)$, and the upper bound
in~\eqref{Nagle-Rodl-Schacht}, together with the general lower bound
$2^{\textup{ex}_r(n,F)}$, determines $\textup{forb}_r(n,F)$ up to a
factor $1+o(1)$ in the exponent.  When $F$ is $r$-partite, however,
the error term $o(n^r)$ may be larger than $\textup{ex}_r(n,F)$, so
\eqref{Nagle-Rodl-Schacht} need not determine $\textup{forb}_r(n,F)$
up to a factor $1+o(1)$ in the exponent.

In particular, recall that $C_k^{(r)}$ is
$r$-partite and, by the asymptotic formula above,
$\textup{ex}_r(n,C_k^{(r)})=\Theta(n^{r-1})$.  Consequently,
\eqref{Nagle-Rodl-Schacht} alone does not imply an upper bound of the
form $\textup{forb}_r(n,C_k^{(r)})\leq2^{O(n^{r-1})}$. Mubayi and
Wang~\cite{Mubayi2019Cycles} proved that
$\textup{forb}_3(n,C_k^{(3)})\leq
2^{C\textup{ex}_3(n,C_k^{(3)})}$ when $k$ is even.  Balogh,
Narayanan and Skokan~\cite{Balogh2019NumberCyclefree} established the
corresponding bound for every $r,k\geq3$, and Ferber, McKinley and
Samotij~\cite[Corollary~8]{Ferber2020Supersaturated} recovered this
result from their general counting theorem.  Together with the general
lower bound, these results show that
$\textup{forb}_r(n,C_k^{(r)})=2^{\Theta(n^{r-1})}$, but they do not
determine whether
$\textup{forb}_r(n,C_k^{(r)})=
2^{(1+o(1))\textup{ex}_r(n,C_k^{(r)})}$.  

Balogh, Narayanan and Skokan
therefore asked whether the following hypergraph analogue of
Erd\H{o}s's conjecture holds.

\begin{quest}[Balogh, Narayanan and Skokan~\cite{Balogh2019NumberCyclefree}]
\label{conj:enumeration result}
For every pair of integers $r,k\geq 3$, does the following hold as $n\to\infty$?
\[
\textup{forb}_r(n,C_k^{(r)})
=2^{(1+o(1))\textup{ex}_r(n,C_k^{(r)})}.
\]
\end{quest}

Recently, Jiang and Longbrake used the delta-system method to answer
Question~\ref{conj:enumeration result} for $r\geq5$ and $k\geq3$.

\begin{theorem}[Jiang and Longbrake~\cite{JiangLongbrake2026Hfree}]
\label{thm:Jiang-Longbrake}
    Let $r\geq5$ and $k\geq3$ be integers. Then, as $n\to\infty$,
    \[
    \mathrm{forb}_r(n,C_k^{(r)})
    =2^{(1+o(1))\mathrm{ex}_r(n,C_k^{(r)})}.
    \]
\end{theorem}

\subsection{Main result} In this paper, we completely answer Question~\ref{conj:enumeration result}
as follows.
\begin{theorem}\label{thm:numberOfCycleFree}
    Let $r,k\geq3$ be integers. Then
    \[
    \mathrm{forb}_r(n, C_{k}^{(r)})=2^{(1+o(1))\textup{ex}_{r}(n,C_{k}^{(r)})}.
    \]
\end{theorem}
Our proof of Theorem~\ref{thm:numberOfCycleFree} gives a new proof of the result of Jiang and
Longbrake stated in Theorem~\ref{thm:Jiang-Longbrake}, without using
the delta-system method.  It also stands in sharp contrast to the graph
case: as discussed above, Morris and Saxton's
result~\eqref{eq:Erdosdisproof} shows that the asymptotic formula fails
for $C_6$ when $r=2$, whereas Theorem~\ref{thm:numberOfCycleFree}
shows that it holds for every cycle $C_k^{(r)}$ with
$r,k\geq3$.

The proof of Theorem~\ref{thm:numberOfCycleFree} has two main
components.  The first, which covers every $(r,k)\neq(3,3)$, is based on the four balanced supersaturation lemmas for linear cycles $C_k^{(r)}$ established in Section~\ref{sec:Supersaturation}.  Together with the hypergraph container method, these lemmas prove Theorem~\ref{thm:numberOfCycleFree} in all these cases.
The second component resolves the case $(r,k)=(3,3)$.
In Section~\ref{sec:concluding-remarks}, we describe a construction
that demonstrates an obstruction to obtaining the balanced supersaturation
required for our container argument.  We therefore use a different
counting method in this case.  Given a
$C_3^{(3)}$-free $3$-graph, we partition its hyperedges according to
whether they contain a pair of large codegree.  The subhypergraph
consisting of the hyperedges that contain no such pair and the set of
large-codegree pairs each have only $2^{o(n^2)}$ possibilities.  We
encode the remaining subhypergraph by a digraph and apply a multicolour
entropy theorem to count the possible encodings.  This proves
Theorem~\ref{thm:numberOfCycleFree} in the only case not covered by
balanced supersaturation.

We next describe our balanced supersaturation arguments in more
detail. The proof of the supersaturation lemmas depends on the structure of the
hypergraph.  In particular, we consider three main cases, according to
whether the hypergraph contains a large $3$-uniform subhypergraph with
bounded pair-codegrees, a large subhypergraph with large codegrees, or
a subhypergraph whose shadow is large and whose shadow hyperedges have
sufficiently many extensions.  In each case, the balanced
supersaturation argument must meet two competing requirements: it must
construct sufficiently many cycles while ensuring that no prescribed
set of hyperedges is contained in too many of the constructed cycles.

In the bounded pair-codegree case, we first show that the hypergraph must contain a large linear subhypergraph and use the graph regularity method
to find many cycles in the shadow, discard those whose extensions do
not form linear cycles, and select a balanced subcollection. In the large-codegree case, we build on the skeleton construction of
Balogh, Narayanan and Skokan~\cite{Balogh2019NumberCyclefree}, which
constructs many copies of $C_k^{(r)}$ in hypergraphs with large
codegrees.
Our main tool is a collection of auxiliary graphs that
guides the choices made during the construction of cycles.  Writing
$H$ for the large-codegree subhypergraph, for every
$f\in\partial H$ we define such a graph on the set
$
 N_H(f)\coloneqq\{z:f\cup\{z\}\in H\}
$
of extension vertices of $f$.  An adjacency in this graph permits one
extension vertex of $f$ to be replaced by another while $f$ remains
fixed.  The auxiliary graphs are chosen to be almost regular: the lower
bound on their vertex degrees gives many choices at each step of the
construction, while the upper bound controls the number of constructions
compatible with any prescribed collection of cycle hyperedges.
The key distinction from the supersaturation result of
Balogh, Narayanan and Skokan~\cite{Balogh2019NumberCyclefree} is that our supersaturation result
(Lemma~\ref{lem: Supersat Large codegrees}) applies even to host
hypergraphs with only $\varepsilon n^{r-1}$ hyperedges, for an
arbitrarily small fixed $\varepsilon>0$, provided that every member
of the shadow is contained in at least $C_1$ hyperedges, where $C_1$
is sufficiently large in terms of $\varepsilon,r,k$.

When the shadow is large, we build on the method of Kostochka, Mubayi
and Verstra\"ete~\cite{Kostochka2015}, which constructs linear cycles
in a hypergraph from complete multipartite subhypergraphs of its shadow
whose hyperedges have sufficiently many extensions.  Here, an
extension of $f\in\partial H$ is a hyperedge of the form
$f\cup\{z\}$.  For balanced supersaturation, we find many such complete
multipartite subhypergraphs and select the resulting cycles so that no
prescribed collection of hyperedges occurs in too many of them.  The
extension lemma of Kostochka, Mubayi and Verstra\"ete applies to all
the large-shadow cases with $k\geq4$, except when $r=3$ and $k$ is
odd.  Our main contribution in these remaining cases is to develop new
ways to extend configurations in the shadow to linear cycles while
controlling how often any prescribed hyperedge occurs.  For $k=3$ and
$r\geq4$, we establish a new extension lemma for complete multipartite
configurations.
For $r=3$ and odd $k\geq5$,
the proof of our extension lemma (Lemma~\ref{lem:a copy of Ck from Ktt})
splits into two cases.  In the first, there is a vertex that is an
available extension vertex for many pairwise disjoint edges of the
shadow.  We extend a suitable path in the shadow to a copy of
$P_{k-2}^{(3)}$ using only extension vertices different from this
vertex, and then complete the path using two hyperedges containing the
vertex.  In the second case, we extend an even cycle in the shadow to
a copy of $C_{k-1}^{(3)}$ and then replace one of its hyperedges by two
hyperedges.  A pair of auxiliary
bipartite graphs provides the choices
needed for this replacement and controls how often a prescribed
hyperedge occurs in the resulting cycles.  A detailed discussion is
given in Section~\ref{proofoverview}.

\subsection{Organization of the paper}
In the next subsection we introduce notation used throughout the paper,
and in Section~\ref{proofoverview} we give an overview of the proof.
Section~\ref{sec:tools} collects the external and preliminary results,
organized into subsections on structural and elementary tools, graph
regularity, containers, hypergraphs with large shadow, and linear
triangles in uniformity $3$.  In Section~\ref{sec:Cycle from shadow}
we prove the lemmas used to construct linear cycles from complete
multipartite subhypergraphs of the shadow.  In
Section~\ref{sec:Supersaturation} we prove the four balanced
supersaturation lemmas for all pairs $(r,k)\neq(3,3)$, and in
Section~\ref{sec:EnumerationProof} we combine them
with the hypergraph container lemma to prove our container theorem for
cycle-free hypergraphs, Theorem~\ref{thm:MainContainer}.
Section~\ref{sec:C33-enumeration} gives the separate entropy argument
for linear triangles in uniformity three and proves
Theorem~\ref{thm:C33-enumeration}.  In
Section~\ref{sec:proof-main-theorem}, we combine this result with
Theorem~\ref{thm:MainContainer} to complete the proof of
Theorem~\ref{thm:numberOfCycleFree}.  We conclude with some remarks in
Section~\ref{sec:concluding-remarks}.

\subsection{Notation}
We call the members of a hypergraph \emph{hyperedges}, the members of
an ordinary graph \emph{edges}, and the ordered pairs belonging to a
digraph \emph{arcs}.  We write $H$ for the set $E(H)$ of hyperedges of
an $r$-graph $H$ and $v(H)$ for $|V(H)|$.
For a graph $G$ and a set $U\subseteq V(G)$, we write $G[U]$ for the
subgraph of $G$ induced by $U$; thus $V(G[U])=U$ and
$E(G[U])=E(G)\cap\binom{U}{2}$.

In an $r$-graph $H$, a set $f \subseteq V(H)$ is called a \emph{subedge} if there exists $e \in H$ with $f \subseteq e$. The \emph{shadow} of $H$, denoted by $\partial H$, is defined as the collection of all subedges $f$ of $H$ with $|f|=r-1$.  For an $r$-graph $H$, an $(r-1)$-graph $G\subseteq \partial H$ and $f\in G$, let 
\[
L_{G}(f)\coloneqq \{v\in V(H)\setminus V(G): (\{v\}\cup f)\in H\}.
\]
We also let $\widehat G\coloneqq\{f\cup\{x\}:f\in G,\ x\in L_G(f)\}$.

For $f\in\partial H$, a hyperedge $e\in H$ containing $f$ is called an
\emph{extension} of $f$ in $H$, and the unique vertex in $e\setminus f$
is its \emph{extension vertex}.  If $Q\subseteq\partial H$, an
\emph{extension of $Q$ in $H$} is a collection $\{e_f:f\in Q\}$,
where $e_f$ is one chosen extension of $f$ for every $f\in Q$.

For an $r$-graph $H$ and a vertex $v \in V(H)$, the number of hyperedges of $H$ containing $v$ is called the \emph{degree} of $v$, and is denoted by $d_H(v)$. Moreover, for a set $f\subseteq V(H)$ of size at least $2$, the \emph{codegree} of $f$ is the number of hyperedges of $H$ that contain $f$, and is denoted by $d_H(f)$. For convenience, we sometimes suppress the set notation and write, for example, $d_H(u,v)$ instead of $d_H(\{u,v\})$. For an $(r-1)$-set $f$, we define the \emph{neighborhood}
$N_{H}(f)\coloneqq \{v \in V(H) : f \cup \{v\} \in H\}$. For an $r$-graph $H$ and $\ell\in [r]$, the \emph{maximum $\ell$-codegree} of $H$ is defined by
\[
\Delta_\ell(H)\coloneqq \max\{d_H(f): f\subseteq V(H),\ |f|=\ell\}.
\]
For a collection $\KK$ of copies of $C_k^{(r)}$ in an $r$-graph $H$, sometimes we regard $\KK$ as a $k$-graph with vertex set $V(\KK)\coloneqq E(H)$, where each copy of $C_k^{(r)}$ gives rise to a hyperedge consisting of its hyperedges. Thus, for $j\in [k]$, $\Delta_j(\KK)$ denotes the maximum number of copies of $C_k^{(r)}$ in $\KK$ containing any given set of $j$ hyperedges of $H$. 

An $r$-graph $H$ is called \emph{$d$-full} if every $(r-1)$-subedge of $H$ has codegree at least $d$. An $r$-graph $H$ is \emph{$(t,C)$-sparse} if every set of $t$ vertices is contained in at most $C$ hyperedges, implying that $\Delta_t(H) \le C$. Note that a $(2,1)$-sparse $r$-graph is simply a linear $r$-graph.

We use the notation $f\ll g$ in two standard ways. When $f$ and $g$ are functions of $n$, it means that $f=o(g)$ as $n\to\infty$. When it is used between fixed parameters, it denotes a hierarchy of constants: for example, $a\ll b,c$ means that $a$ is chosen sufficiently small in terms of $b$ and $c$. The intended meaning will always be clear from context.
We write $\log$ for the natural logarithm and $\log_2$ for the
base-two logarithm.
\section{Proof overview}
\label{proofoverview}

Our proof of Theorem~\ref{thm:numberOfCycleFree} separates the case
$(r,k)=(3,3)$ from all other pairs $(r,k)$.  For $(r,k)=(3,3)$ we count
$C_3^{(3)}$-free $3$-graphs directly, using a digraph encoding and a
multicolour counting theorem.  For every other pair we construct a
balanced family of copies of $C_k^{(r)}$ and apply the hypergraph
container method. We describe these two arguments in that order.

\subsection{Linear triangles in uniformity three}
Our argument here combines the Ruzsa--Szemer\'edi theorem~\cite{RuzsaSzemeredi1978}, the Brown--Harary theorem~\cite{BrownHarary1970} for
digraphs, and the multicolour counting theorem of Falgas--Ravry,
O'Connell and Uzzell~\cite{FalgasRavryOConnellUzzell2019}.
Write $T\coloneqq C_3^{(3)}$, and let $\mathcal L_n$ be the family of
linear $T$-free $3$-graphs on $[n]$.  The Ruzsa--Szemer\'edi theorem implies that
$|\mathcal L_n|=2^{o(n^2)}$, so we may choose an integer $q=q(n)$ such
that $q\to\infty$ and $q\log_2|\mathcal L_n|=o(n^2)$.
Fix a $T$-free $3$-graph $H$ on $[n]$, let
$P\coloneqq\{xy\in\tbinom{[n]}2:d_H(xy)\geq q\}$, and let $H_0$
consist of the hyperedges of $H$ containing no pair from $P$.  The graph
whose vertices are the hyperedges of $H_0$, with two vertices adjacent when
the corresponding hyperedges share a pair of vertices, has maximum degree less than $3q$.
A proper vertex $3q$-coloring of the graph therefore partitions $H_0$ into at most  $3q$ linear
$T$-free $3$-graphs, giving $|\mathcal L_n|^{3q}=2^{o(n^2)}$
possibilities for $H_0$.  Moreover, a theorem of Frankl and F\"uredi (Theorem~\ref{thm:C33-FF-extremal}) gives $q|P|\leq3|H|=O(n^2)$.  Hence $|P|=o(n^2)$, and there
are $2^{o(n^2)}$ possibilities for $P$.

Next we count the choices for $H\setminus H_0$, which consists precisely of the hyperedges of $H$ that contain a pair from $P$.  We encode these hyperedges using a digraph.  To define this encoding, for each $z\in[n]$, let $F_z$ be the graph whose edges are the pairs $xy\in P$ for which $xyz\in H$, and let $V_z$ be the set of vertices incident to an edge of $F_z$.
Lemma~\ref{lem:C33-induced-link-solution}
shows that $P$ is triangle-free and that $F_z=P[V_z]$ for every
$z\in[n]$.  Define a digraph $D$ on $[n]$ by putting $x\to z$ in $D$
exactly when $x\in V_z$.  If $xy\in P$, then
$xyz\in H$ if and only if both $x\to z$ and $y\to z$ belong to $D$.
Thus $H_0$, $P$, and $D$ determine every hyperedge of $H$.

An unordered pair of vertices is a \emph{digon} if both possible arcs
on that pair are present.  For each $xy\in P$, delete from $D$ whichever of the arcs $x\to y$ and
$y\to x$ are present, and denote the resulting digraph by $D_0$.
Lemma~\ref{lem:C33-mixed-triangle-high-pair} shows that every directed
triangle in $D_0$ has all three underlying pairs as digons.  Let $\Gamma_n$ be the family of digraphs on $[n]$ with this property. To count the possibilities for $\Gamma_n$, regard the four possible states on
an unordered pair---no arc, either one of the two single arcs, or a
digon---as four colours.  A four-colour template $\Phi$ assigns to
each pair $xy$ a nonempty palette $\Phi_{xy}$ of permitted states; it
is $\Gamma_n$-valid if every choice of one state from each palette
produces a member of $\Gamma_n$.  For distinct vertices $x,y\in[n]$, call the orientation $x\to y$
\emph{loose in $\Phi$} if $\Phi_{xy}$ contains the state consisting
of the single arc $x\to y$.  Define a digraph $J=J(\Phi)$ on $[n]$ by putting $x\to y$ in $J$
if some state in $\Phi_{xy}$ contains the arc $x\to y$.  For each
$xy\in\binom{[n]}2$, let $\ell_{xy}\in\{0,1,2\}$ be the number
of loose arcs on $xy$, and let $B$ consist of the pairs $xy$ on which
$J$ contains both arcs and $\ell_{xy}\leq1$.  In the proof of
Lemma~\ref{lem:C33-mixed-digraph-count} we construct a new digraph $K$ by
including all loose arcs, the non-loose arc on each pair in $B$ with
$\ell_{xy}=1$, and the arc directed from the smaller endpoint to the
larger endpoint on each pair in $B$ with $\ell_{xy}=0$.  In particular, this construction ensures that $K$ has no directed triangles and that
\[
 \sum_{xy\in\binom{[n]}2}\log_2|\Phi_{xy}|
 \leq\sum_{xy}\ell_{xy}+|B|=|A(K)|.
\]
Now, the Brown--Harary
theorem gives $|A(K)|\leq\lfloor n^2/2\rfloor$.  The multicolour counting theorem
(Theorem~\ref{thm:C33-multicolour-counting}) then yields
$|\Gamma_n|\leq2^{\binom n2+o(n^2)}$.  As shown above, there are $2^{o(n^2)}$ possibilities
for each of $H_0$ and $P$, and $|P|=o(n^2)$.  Once these have been fixed, $D$ is determined by $D_0\in\Gamma_n$
and, for each $xy\in P$, by which of the arcs $x\to y$ and $y\to x$
are present in $D$. The latter
have at most $4^{|P|}=2^{o(n^2)}$ possibilities.  Since $H_0$, $P$,
and $D$ determine $H$, we obtain
$\mathrm{forb}_3(n,T)\leq2^{\binom n2+o(n^2)}$.  Conversely, a full
$3$-uniform star is $T$-free and has $\binom{n-1}{2}$ hyperedges, so
its subhypergraphs give
$\mathrm{forb}_3(n,T)\geq2^{\binom{n-1}{2}}$.

\subsection{Balanced supersaturation for the remaining cases}
The main task for the remaining pairs $(r,k)$ is not merely to find
many linear cycles, but to find them in a sufficiently even
distribution: no small collection of hyperedges should occur in too
many of the chosen cycles.  The construction used to obtain this
balance depends on properties of a suitable subhypergraph of the dense
host hypergraph.
We state these properties precisely below.  First, we describe the
quantitative bounds meant by balanced supersaturation.
Given a collection $\KK$ of copies of $C_k^{(r)}$ in an $r$-graph
$H$, we regard $\KK$ as a $k$-graph with vertex set $E(H)$.  For
$\ell\in[k]$, the quantity $\Delta_\ell(\KK)$ is the maximum number
of copies in $\KK$ containing a fixed collection of $\ell$ hyperedges of
$H$.  Each of the four supersaturation lemmas specifies parameters
$D,K,s$, where $s$ is $n^2$, $n^{r-1}$, or $|H_2|$, as stated in the
relevant lemma, and produces a collection $\KK$ satisfying
$|\KK|\geq D^{k-1}s/K$ and
$\Delta_\ell(\KK)\leq D^{k-\ell}$ for every $\ell\in[k]$.  These
bounds allow us to apply
the hypergraph container lemma.

We next describe the subhypergraphs to which the four supersaturation
lemmas are applied.  Recall that the shadow $\partial H$ of an $r$-graph $H$ is the $(r-1)$-graph
consisting of the $(r-1)$-sets contained in hyperedges of $H$.  An
$r$-graph is $d$-full if every member of its shadow is contained in at
least $d$ of its hyperedges.  A $3$-graph is $(2,C_2)$-sparse if every
pair of vertices is contained in at most $C_2$ hyperedges.

Propositions~\ref{prop:kOdd} and~\ref{prop:kEven} give the following
structural alternatives for sufficiently dense hypergraphs.  If $r=3$ and $k\geq5$
is odd, then every $3$-graph $H$ with at
least $\big((k-1)/2+\varepsilon/2\big)\binom{n}{2}$
hyperedges satisfies at least one of:
\smallskip
\begin{enumerate}[start=0,label=(A\arabic*),font=\upshape,itemsep=4pt]
\item
$H$ contains a $(2,C_2)$-sparse $3$-graph $H_0$ with
$\Omega(\varepsilon n^2)$ hyperedges;
\item
$H$ contains a $C_1$-full $3$-graph $H_2$ with
$\Omega(\varepsilon n^2)$ hyperedges;
\item
$H$ contains a $\frac{k+1}{2}$-full
$3$-graph $H_1$ with $\Omega(\varepsilon n^2)$ hyperedges and
$|\partial H_1|=\Omega(\varepsilon n^2/C_1)$ such that every hyperedge of
$H_1$ contains a pair $xy$ satisfying
$d_H(xy)\geq C_2$ in the original $3$-graph $H$.
\end{enumerate}
\smallskip
If $r\geq4$ and $k\geq3$, or if $r=3$ and $k\geq4$ is even, then every
$r$-graph $H$ with at least
$\big(\lfloor(k-1)/2\rfloor+\varepsilon/2\big)\binom{n}{r-1}$ hyperedges
satisfies at least one of:

\smallskip
\begin{enumerate}[label=(B\arabic*),font=\upshape,itemsep=4pt]
\item
$H$ contains a $C_1$-full $r$-graph $H_2$ with
$\Omega(\varepsilon n^{r-1})$ hyperedges;
\item
$H$ contains a $\left\lfloor\frac{k+1}{2}\right\rfloor$-full
$r$-graph $H_1$ with $\Omega(\varepsilon n^{r-1})$ hyperedges and
$|\partial H_1|=\Omega(\varepsilon n^{r-1}/C_1)$.
\end{enumerate}
\smallskip

In cases~\ref{itm: Odd Case linear},
\ref{itm: Odd Case Shadow}, and~\ref{itm: Even Case Shadow}, we first
construct a collection $\KK^*$ and then select a subcollection
$\KK\subseteq\KK^*$ by starting with the empty collection and adding
cycles one by one.
During this selection, for $1\leq\ell\leq k-1$, an unordered $\ell$-tuple
$\sigma=\{e_1,\ldots,e_\ell\}$ of distinct hyperedges is called
\emph{saturated} when it is already
contained in $\lfloor D^{k-\ell}\rfloor$ members of $\KK$.  We add a
cycle from $\KK^*\setminus\KK$ only when it contains no saturated
tuple.  Counting the pairs $(\sigma,C)$ with $\sigma$ saturated,
$C\in\KK$, and $\sigma\subseteq C$ shows that there are at most
$\binom{k}{\ell}|\KK|/\lfloor D^{k-\ell}\rfloor$ saturated
$\ell$-tuples: each such tuple belongs to
$\lfloor D^{k-\ell}\rfloor$ members of $\KK$, and each member contains
only $\binom{k}{\ell}$ sets of $\ell$ hyperedges.  For every
$\ell\in[k-1]$, the proof of the applicable lemma bounds the number of
cycles in $\KK^*$ containing a saturated $\ell$-tuple.  The sum of these
bounds is smaller than the number of
unused cycles in $\KK^*$ until the asserted lower bound for $|\KK|$
is reached.

In cases~\ref{itm: Odd Case Codegree} and
\ref{itm: Even Case Codegrees}, no selection of a subcollection is
needed.  We count directly both the cycles produced by the
construction and those that contain a fixed set of hyperedges.
These two counts give, respectively, the lower bound for $|\KK|$
and the upper bounds for $\Delta_\ell(\KK)$.

\subsubsection{Case~\ref{itm: Odd Case linear}: a large
$(2,C_2)$-sparse $3$-graph}
Although the $3$-graph $H_0$ in
condition~\ref{itm: Odd Case linear} need not be linear, its bounded
pair-codegrees allow us to extract a large linear subhypergraph.  We define
an auxiliary graph with vertex set $E(H_0)$, joining two vertices
whenever the corresponding hyperedges
intersect in two vertices.  Its maximum degree is $O(C_2)$, so a
greedy proper coloring gives a linear $3$-graph
$H_0^*\subseteq H_0$ with $\Omega(|H_0|/C_2)=\Omega(\varepsilon n^2/C_2)$ hyperedges.

For every $e\in H_0^*$, the three pairs contained in $e$ form a
triangle in the shadow $\partial H_0^*$; call this the triangle
associated with $e$.  Because $H_0^*$ is linear, every shadow edge
belongs to exactly one associated triangle.  We apply Szemer\'edi's
regularity lemma~\cite{Szemeredi1978Regular} with parameters depending only on $\varepsilon,k$
and $C_2$.  After deleting the edges inside clusters, incident to the
exceptional set, between irregular pairs, or between sufficiently
sparse regular pairs, many associated triangles remain.  The reduced
graph, whose vertices are the regularity clusters and whose edges
represent dense regular pairs, therefore contains a triangle.  Since
$k$ is odd, the graph counting lemma then gives
$\Omega_{\varepsilon,k,C_2}(n^k)$ copies of $C_k$ in
$\partial H_0^*$.

Every edge of a graph cycle in $\partial H_0^*$ has a unique extension
in $H_0^*$.  The collection of these extensions can fail to be a copy
of $C_k^{(3)}$ only if two shadow edges have the same extension vertex,
or if the extension vertex of one shadow edge lies on another edge of
the graph cycle.  Claim~\ref{claim2} shows that after discarding the
graph cycles with either defect,
$\Omega_{\varepsilon,k,C_2}(n^k)$ copies of $C_k^{(3)}$ remain.
Let $\KK^*$ denote this collection of copies.

For a copy with hyperedges $e_1,\ldots,e_k$ in cyclic order, let $v_i$
be the unique vertex in $e_{i-1}\cap e_i$, where indices are taken
modulo $k$.
Thus $v_1,\ldots,v_k$, in this order, form a copy of $C_k$ in
$\partial H_0^*$, and the third vertex of $e_i$ belongs to no other
hyperedge of this copy.  Fix a saturated $\ell$-tuple $\sigma$, where
$1\leq\ell\leq k-1$.  There
are $O_k(1)$ ways to choose the positions occupied by its hyperedges in
the cyclic order and, for each of these hyperedges, to decide which two
of its vertices play the roles of $v_i$ and $v_{i+1}$.  After these
choices are fixed, the hyperedges in $\sigma$ determine at least
$\ell+1$ of $v_1,\ldots,v_k$, because a proper set of $\ell$ edges of
the graph cycle is incident with at least $\ell+1$ vertices.  There are
therefore at most $n^{k-\ell-1}$ choices for the remaining vertices
$v_i$.  Since $H_0^*$ is linear, each pair $v_iv_{i+1}$ is contained
in at most one hyperedge of $H_0^*$.  Consequently, $\sigma$ is
contained in $O_k(n^{k-\ell-1})$ members of $\KK^*$.

Here $D=n^{(k-2)/(k-1)}$ and $K=\log n$.  For sufficiently large $n$, while
$|\KK|<D^{k-1}n^2/K$, the preceding double count gives at most
$O_k(D^{\ell-1}n^2/K)$ saturated $\ell$-tuples.  Combining this with
the bound $O_k(n^{k-\ell-1})$ for each tuple and using $D\leq n$
shows that only $O_k(n^k/\log n)$ members of $\KK^*$ contain any
saturated tuple.  Since $|\KK^*|=\Omega_{\varepsilon,k,C_2}(n^k)$
and $|\KK|<n^k/\log n$ until the required size is reached, there is
always an unused cycle containing no saturated tuple.  The selection
can therefore continue until $|\KK|\geq D^{k-1}n^2/K$, while ensuring
$\Delta_\ell(\KK)\leq D^{k-\ell}$ for every $\ell\in[k]$.
Szemer\'edi's regularity lemma and the graph counting lemma provide the
graph cycles from which Claim~\ref{claim2} constructs $\KK^*$, and the
rule of adding only cycles that contain no saturated tuple produces the
required subcollection $\KK$.  This proves our first supersaturation
lemma, Lemma~\ref{lem: Supersat Large Linear subgraph}.

\subsubsection{Cases~\ref{itm: Odd Case Codegree} and
\ref{itm: Even Case Codegrees}: a large $C_1$-full subhypergraph}
For the $C_1$-full subhypergraph $H_2$, we construct the balanced collection
by adapting a cycle construction of Balogh, Narayanan and
Skokan~\cite{Balogh2019NumberCyclefree}.  For each $f\in\partial H_2$,
we fix an auxiliary graph $\Gamma_f$ on $N_{H_2}(f)$, with every
degree equal to $C_1-1$ or $C_1$.
If $f\cup\{x\}$ is a hyperedge obtained at some stage of the
construction and $x'$ is adjacent to $x$ in $\Gamma_f$,
then $f\cup\{x'\}$ is also a hyperedge, so we may replace $x$ by $x'$
while keeping the other $r-1$ vertices fixed.  We call this operation
a \emph{switch over $f$}.

Here is the system that specifies the order of the switches.  With
indices modulo $k$, we seek distinct vertices
$v_1,\ldots,v_k,u_1,\ldots,u_k$ and put
$y_i\coloneqq\{v_i,v_{i+1},u_i\}$ for $i\in[k]$.  We call these sets
the \emph{boundary triples}.  Starting with a hyperedge $e_1$, we use
an ordering fixed in advance for the vertices of every hyperedge to
write $e_1=y_1\cup R$, where
$R\coloneqq(z_1,\ldots,z_{r-3})$ is an ordered list of its remaining
vertices; a union with $R$ means a union with its underlying set.
For $2\leq i\leq k-1$, put
$q_i\coloneqq\{v_1,v_i,v_{i+1}\}$.  Form a graph on the triples
$y_1,\ldots,y_k,q_2,\ldots,q_{k-1}$ by joining two triples when they
share two vertices; see Figure~\ref{fig:triangulation-tree}.
This graph is a tree.  Starting with $y_1\cup R$ and moving away from
$y_1$ in the tree, we obtain $a\cup R$ for each new triple $a$ by
one switch from $b\cup R$, where $b$ is the unique neighbor of $a$
one step closer to $y_1$.  At each switch, the new vertex is chosen
to be distinct from every vertex chosen earlier.  All vertices of
$R$ remain fixed during these switches.  The
$2k-3$ switches construct all the boundary triples.  For each
$i\in\{2,\ldots,k\}$, we then replace the vertices of $R$ in the
hyperedge $y_i\cup R$, one at a time, by vertices outside
$\{v_1,\ldots,v_k,u_1,\ldots,u_k\}\cup R$ and distinct from all earlier
replacement vertices.
Such a new vertex is \emph{private} to the
resulting hyperedge.  The resulting $k$ hyperedges form a copy of
$C_k^{(r)}$.

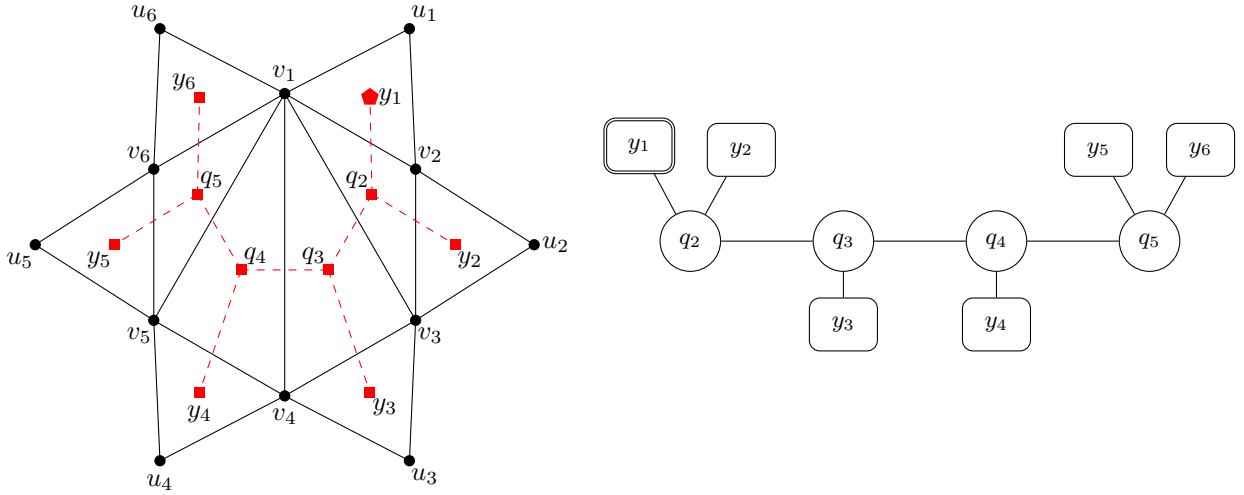
\begin{figure}[ht]
    \centering
\begin{tikzpicture}[scale=1,
    every node/.style={circle, fill=black, inner sep=1.5pt}]

\foreach \i/\angle in {1/90,2/30,3/-30,4/-90,5/-150,6/150}
    \coordinate (v\i) at (\angle:2);

\foreach \i/\angle in {1/60,2/0,3/-60,4/-120,5/180,6/120}
    \coordinate (u\i) at (\angle:3.3);

\coordinate (y1) at (barycentric cs:v1=1,u1=1,v2=1);
\coordinate (y2) at (barycentric cs:v2=1,u2=1,v3=1);
\coordinate (y3) at (barycentric cs:v3=1,u3=1,v4=1);
\coordinate (y4) at (barycentric cs:v4=1,u4=1,v5=1);
\coordinate (y5) at (barycentric cs:v5=1,u5=1,v6=1);
\coordinate (y6) at (barycentric cs:v6=1,u6=1,v1=1);

\coordinate (q2) at (barycentric cs:v1=1,v2=1,v3=1);
\coordinate (q3) at (barycentric cs:v1=1,v3=1,v4=1);
\coordinate (q4) at (barycentric cs:v1=1,v4=1,v5=1);
\coordinate (q5) at (barycentric cs:v1=1,v5=1,v6=1);

\draw
    (v1)--(v2)--(v3)--(v4)--(v5)--(v6)--cycle;

\draw
    (v1)--(u1)--(v2)
    (v2)--(u2)--(v3)
    (v3)--(u3)--(v4)
    (v4)--(u4)--(v5)
    (v5)--(u5)--(v6)
    (v6)--(u6)--(v1)
    (v1)--(v3)
    (v1)--(v4)
    (v1)--(v5);
\draw[red,dashed]
 (q2)--(y1)
 (q2)--(q3)
 (q3)--(q4)
 (q4)--(q5) 
 (q2)--(y2)
 (q3)--(y3)
 (q4)--(y4)
 (q5)--(y5)
 (q5)--(y6);

\foreach \i in {1,...,6} {
    \node at (v\i) {};
    \node at (u\i) {};
}
\foreach \i in {2,...,5} {
    \node[rectangle, fill=red, inner sep=2pt] at (q\i) {};
}

\node[regular polygon,regular polygon sides=5, fill=red, inner sep=2pt] at (y1) {};

\foreach \i in {2,...,6} {
    \node[rectangle, fill=red, inner sep=2pt] at (y\i) {};
}

\node[fill=none, above]  at (v1) {$v_1$};
\node[fill=none, above right] at (v2) {$v_2$};
\node[fill=none, below right]       at (v3) {$v_3$};
\node[fill=none, below] at (v4) {$v_4$};
\node[fill=none, below left]  at (v5) {$v_5$};
\node[fill=none, above left]        at (v6) {$v_6$};

\node[fill=none, above right]       at (u1) {$u_1$};
\node[fill=none, right] at (u2) {$u_2$};
\node[fill=none, below right] at (u3) {$u_3$};
\node[fill=none, below]       at (u4) {$u_4$};
\node[fill=none, below left]  at (u5) {$u_5$};
\node[fill=none, above left]  at (u6) {$u_6$};

\node[fill=none,  right]       at (y1) {$y_1$};
\node[fill=none, below right] at (y2) {$y_2$};
\node[fill=none, below right] at (y3) {$y_3$};
\node[fill=none, below]       at (y4) {$y_4$};
\node[fill=none, below left]  at (y5) {$y_5$};
\node[fill=none, above left]  at (y6) {$y_6$};

\node[fill=none, above left] at (q2) {$q_2$};
\node[fill=none, above left] at (q3) {$q_3$};
\node[fill=none, above right]  at (q4) {$q_4$};
\node[fill=none, above right ]  at (q5) {$q_5$};

\end{tikzpicture}\quad \raisebox{.65\height}{\begin{tikzpicture}[
  x=1.35cm,
  y=1.05cm,
  every node/.style={font=\small},
  diagonal/.style={circle,draw,minimum size=8mm,inner sep=1pt},
  boundary/.style={rectangle,rounded corners,draw,minimum width=9mm,
                   minimum height=7mm,inner sep=1pt}]
\node[diagonal] (q2) at (-2.25,0) {$q_2$};
\node[diagonal] (q3) at (-0.75,0) {$q_3$};
\node[diagonal] (q4) at (0.75,0) {$q_4$};
\node[diagonal] (q5) at (2.25,0) {$q_5$};
\node[boundary,double] (y1) at (-2.75,1.20) {$y_1$};
\node[boundary] (y2) at (-1.75,1.15) {$y_2$};
\node[boundary] (y3) at (-0.75,-1.05) {$y_3$};
\node[boundary] (y4) at (0.75,-1.05) {$y_4$};
\node[boundary] (y5) at (1.75,1.15) {$y_5$};
\node[boundary] (y6) at (2.75,1.15) {$y_6$};
\draw (q2)--(q3)--(q4)--(q5);
\draw (q2)--(y1) (q2)--(y2) (q3)--(y3) (q4)--(y4)
      (q5)--(y5) (q5)--(y6);
\end{tikzpicture}}
\caption{The triangulation tree for $k=6$, rooted at $y_1$.  The figure on the left shows the vertices $v_i$ and $u_i$, the boundary triples $y_i$,
and the triples $q_i$; the dashed edges join triples that share
two vertices.  The right figure shows the resulting tree: the rounded
boxes are the boundary triples, the circles are the triples $q_i$,
and the double box marks the root $y_1$.}
\label{fig:triangulation-tree}
\end{figure}

A \emph{construction record} is the initial hyperedge $e_1$ together
with the new vertex chosen at each switch.  Each record contains
$M\coloneqq2k-3+(r-3)(k-1)=(r-1)(k-1)-1$ switch choices.
Before each switch, fewer than $k(r-1)$ vertices have been chosen.
Excluding these vertices therefore leaves at least $C_1/2$ choices,
since every degree of $\Gamma_f$ is at least $C_1-1$ and $C_1$ is
sufficiently large.  Conversely, when $f$ and the new vertex $x'$ are
fixed, the upper degree of $\Gamma_f$ gives at most $C_1$ choices
for the vertex $x$ that it replaced.

For each initial hyperedge $e_1$, the construction gives at least
$\Omega_{r,k}((C_1)^{(r-1)(k-1)-1})$ records, and any fixed cycle is
produced by $O_{r,k}(1)$ records.  If $1\leq j\leq k-1$ hyperedges are prescribed, then
there are $O_{r,k}(1)$ ways to place them in the cyclic order and to
identify which of their vertices play the roles of the intersection
vertices, the vertices $u_i$, and the ordered private vertices.  After
these choices, there are at most $(C_1)^{(r-1)(k-j)-1}$ records.  With
$D=(C_1)^{r-1}$ and $K=\Theta_{r,k}(C_1)$, these estimates give
$|\KK|\geq D^{k-1}|H_2|/K$ and
$\Delta_j(\KK)\leq D^{k-j}$ for $1\leq j\leq k-1$, without a
separate selection step.  Also, $\Delta_k(\KK)\leq1$ because a set
of $k$ hyperedges determines at most one copy.  In particular,
$K\log D/D=O_{r,k}(\log C_1/(C_1)^{r-2})$, which can be made arbitrarily
small for every fixed $r\geq3$.  This is the balanced collection in
Lemma~\ref{lem: Supersat Large codegrees}.

\subsubsection{Cases~\ref{itm: Odd Case Shadow} and
\ref{itm: Even Case Shadow}: a full subhypergraph with large shadow}
For the $\lfloor(k+1)/2\rfloor$-full subhypergraph $H_1$ with large shadow,
we combine supersaturation for complete multipartite hypergraphs with
an extension procedure of
Kostochka, Mubayi and Verstra\"ete~\cite{Kostochka2015}.  In
case~\ref{itm: Even Case Shadow}, Lemma~\ref{lem:Large Shadow} is the
starting point: if $|\partial H_1|>\delta n^{r-1}$, it gives a
set $T\subseteq V(H)$ and an $(r-1)$-graph
$F\subseteq\partial H_1$ such that $V(F)\cap T=\varnothing$,
$|F|=\Omega_{r,k}(\delta n^{r-1})$, and every $f\in F$ has at least
$\lfloor(k+1)/2\rfloor$ extension vertices in $T$.  For each
$f\in F$, fix a set $\mathcal L(f)\subseteq N_{H_1}(f)\cap T$ of
size $\lfloor(k+1)/2\rfloor$.

In case~\ref{itm: Odd Case Shadow},
Lemma~\ref{lem:Large Shadow odd} gives the corresponding sets $T$ and
$F$ and, for every $xy\in F$, an extension vertex with the additional
codegree properties needed below.

Put $t\coloneqq(r2^k)^{4r}$.  Proposition~\ref{prop:many-complete-partite} gives
$\Omega_{r,k,\delta}(n^{(r-1)t})$ complete $(r-1)$-partite
$(r-1)$-graphs in $F$ with $t$ vertices in each part.

In case~\ref{itm: Even Case Shadow}, for each complete multipartite
subhypergraph $G$ we choose an extended cycle supplied by
Lemma~\ref{lem:KMV-cycle-from-shadow} when $k\geq4$ and by
Lemma~\ref{lem:triangle-from-shadow} when $k=3$, and let $a(G)$ be the number of vertices of
that cycle that lie in $T$.  The value of $a(G)$ belongs to
$\{\lceil k/2\rceil,\ldots,k\}$, so by the pigeonhole principle there
is a value $a$ that occurs for at least a $1/k$ fraction of the
multipartite subhypergraphs.  We restrict to those subhypergraphs.  If $a=k$, deleting
the vertices in $T$ from a cycle leaves a copy of $C_k^{(r-1)}$ in
$F$.  If $a<k$, it leaves $k-a$ vertex-disjoint paths in $F$, each
containing at least two hyperedges.  When $a=k$, a fixed set of $\ell$ hyperedges belongs to
$O_{r,k}(n^{(r-2)(k-\ell)-1})$ chosen cycles; when $a<k$, the bound
is $O_{r,k}(n^{(r-2)(k-\ell)+(k-a)-1})$.
In either case, double counting shows that saturated tuples exclude
only a fixed fraction of the chosen cycles before the required lower
bound is reached.  The selection therefore gives
$|\KK|\geq D^{k-1}n^{r-1}/K$ and
$\Delta_\ell(\KK)\leq D^{k-\ell}$ with
$D=n^{1/(k-1)}$ and $K=\log n$.  This is the balanced collection
asserted by Lemma~\ref{lem: Supersat Large shadow even}.

Case~\ref{itm: Odd Case Shadow} requires an additional idea.
For each complete bipartite subgraph of $F$ used in the construction,
Lemma~\ref{lem:a copy of Ck from Ktt} gives one of two alternatives.
In the first alternative, the lemma directly extends edges in the
shadow to a copy of $C_k^{(3)}$.  If this alternative holds for at
least half of the complete bipartite subgraphs, we group the resulting
cycles according to their number of vertices in $T$ and apply the same
multiplicity and saturated-tuple estimates as in the general
large-shadow argument above.

In the second alternative, we find a copy of $C_{k-1}$ in $F$, extend
its edges to a copy of $C_{k-1}^{(3)}$ in $H_1$, and replace one
hyperedge of this cycle by two hyperedges of $H$ to obtain a copy of
$C_k^{(3)}$.  For the shadow edge of the hyperedge that is replaced,
Lemma~\ref{lem:Large Shadow odd} supplies an extension vertex
$\alpha$ with the codegree properties used in
Definition~\ref{defi: Auxiliary graph large shadow}.  The two
auxiliary bipartite graphs defined there provide the two replacement
hyperedges and control the number of resulting cycles containing any
prescribed hyperedge.

If the second alternative holds for more than half of the complete
bipartite subgraphs, we use these cycles in the selection.  Depending
on whether a saturated $\ell$-tuple contains two, one, or no replacement
hyperedges, it belongs to at most $O_k(n^{k-1-\ell})$,
$O_k(C_2n^{\max\{k-2-\ell,0\}})$, or
$O_k(C_2n^{k-2-\ell})$ constructed cycles, respectively.  Summing
over all saturated tuples in either alternative shows that
the selection reaches $|\KK|\geq D^{k-1}n^2/K$ with $D=C_2$ and with
$K$ depending only on $C_1$ and $k$, while maintaining
$\Delta_\ell(\KK)\leq D^{k-\ell}$ for every $\ell\in[k]$.  Together,
these estimates establish Lemma~\ref{lem: Supersat Large shadow odd}.

\subsubsection{From supersaturation to enumeration using containers}
The final step combines these balanced collections with the hypergraph
container method of Balogh--Morris--Samotij and
Saxton--Thomason~\cite{balogh2015independent,saxton2015hypergraph}.
To prove Theorem~\ref{thm:MainContainer},
Proposition~\ref{prop:Dense regime} first gives at most
$2^{\frac{\varepsilon}{4}\binom{n}{r-1}}$ containers, each having
$O(n^{r-1})$ hyperedges, such that every
$C_k^{(r)}$-free $r$-graph is a subhypergraph of one of them.  Whenever a
container has at least
$\big(\lfloor(k-1)/2\rfloor+\varepsilon/2\big)\binom{n}{r-1}$ hyperedges,
one of the cases~\ref{itm: Odd Case linear}--\ref{itm: Odd Case Shadow}
or~\ref{itm: Even Case Codegrees}--\ref{itm: Even Case Shadow}
holds.  The corresponding construction supplies a balanced collection
of cycles $\mathcal K$ and parameters $D,K,s$ satisfying
\[
 |\mathcal K|\geq \frac{D^{k-1}s}{K}
 \qquad\text{and}\qquad
 \Delta_\ell(\mathcal K)\leq D^{k-\ell}
 \quad\text{for every }\ell\in[k].
\]
In every case,
$s\geq(\varepsilon/8)\binom{n}{r-1}$.  We then apply
Corollary~\ref{cor:containers-balanced} to $\mathcal K$.
Corollary~\ref{cor:containers-balanced}
replaces the container $H$ by
$2^{O_{r,k}((\log D/D)|H|)}$ subcontainers, each with at most
$|H|-\Omega_{r,k}(s/K)$ hyperedges; the subcontainers still contain every
$C_k^{(r)}$-free subhypergraph of $H$.  We call this replacement one
\emph{refinement}.

For each of the finitely many choices of $(D,K,s)$ supplied by the four
supersaturation lemmas, every refinement using those parameters
decreases the container size by $\Omega_{r,k}(s/K)$.  Hence those
parameters can be used only $O_{r,k}(Kn^{r-1}/s)$ times in any chain
of successive refinements, regardless of the order.  Multiplying the
bounds on the number of subcontainers introduced by these refinements and using
$K\log D/D\ll1$ gives a final collection of at most
$2^{\frac{\varepsilon}{2}\binom{n}{r-1}}$ containers, each with at
most $\big(\lfloor(k-1)/2\rfloor+\varepsilon/2\big)
\binom{n}{r-1}$ hyperedges.  Counting their subhypergraphs proves the upper bound
in Theorem~\ref{thm:numberOfCycleFree}; the lower bound follows from
the subhypergraphs of an extremal $C_k^{(r)}$-free $r$-graph.

The following table summarizes the proof methods we use for the different values of $r$ and $k$.

\begin{table}[ht]
    \centering
    \small
    \renewcommand{\arraystretch}{1.4}
    \begin{tabular}{|p{0.31\textwidth}| p{0.62\textwidth}|}
        \hline
        \textbf{Values of $r,k$} & \textbf{Proof idea} \\
        \hline

        $r\geq 4$ and $k\geq 4$, or
        $r=3$ and $k$ even
        &
        We combine balanced supersaturation with the hypergraph
        container method.  In the large-shadow case,
        Lemma~\ref{lem:KMV-cycle-from-shadow} supplies the required
        extension.
        \\[3pt]
        \hline

        $r=3$ and odd $k\geq 5$
        &
        We combine balanced supersaturation with the hypergraph
        container method.  The bounded pair-codegree case is handled
        using graph regularity, while the large-shadow case requires
        the extension argument in
        Lemma~\ref{lem:a copy of Ck from Ktt}.
        \\[3pt]
        \hline

        $r\geq 4$ and $k=3$
        &
        We combine balanced supersaturation with the hypergraph
        container method.  In the large-shadow case, we use the
        extension lemma for triangles,
        Lemma~\ref{lem:triangle-from-shadow}.
        \\[3pt]
        \hline

        $r=3$ and $k=3$
        &
        We count the hypergraphs directly using a digraph encoding
        and the multicolour counting theorem, rather than obtaining balanced
        supersaturation and using the hypergraph container method.
        \\

        \hline
    \end{tabular}
    \caption{Summary of proof ideas depending on the values
    of $r$ and $k$.}
    \label{tab:proof-summary}
\end{table}
\section{Tools}\label{sec:tools}

\subsection{Structural and elementary tools}

We first state two elementary propositions.

\begin{proposition}\label{prop:Entropy}
    For every $\alpha\in [0,1/2]$ and $n\in \mathbb{Z}^{+}$,
    \[
    \sum_{i\leq \alpha n}\binom{n}{i}\leq 2^{h(\alpha) n},
    \]
    where
    $h(p)=-p\log_2(p)-(1-p)\log_2(1-p)$ is the binary entropy
    function. 
\end{proposition}

\begin{proposition}\label{prop:Graph with fixed degrees}
    For every pair of positive integers $n$, $k$ with $n\geq k+1$, there exists an $n$-vertex graph $G$ such that $d(v)\in \{k,k+1\}$ for every $v\in V(G)$. 
\end{proposition}

We also need a form of Harris's inequality, which is used in
Lemma~\ref{lem:Large Shadow odd}.  Let $X$ be a random subset of a finite
set, obtained by including its elements independently.  An event
$\mathcal A$ determined by $X$ is \emph{increasing} if
$X\in\mathcal A$ and $X\subseteq Y$ imply $Y\in\mathcal A$.

\begin{proposition}[Harris's inequality]
\label{prop:Harris}
If $\mathcal A_1,\ldots,\mathcal A_m$ are increasing events determined
by independent inclusion trials, then
\[
 \mathbb P\left(\bigcap_{i=1}^m\mathcal A_i\right)
 \geq\prod_{i=1}^m\mathbb P(\mathcal A_i).
\]
\end{proposition}

Recall that an $r$-graph $H$ is called $d$\emph{-full} if every $(r-1)$-subedge of $H$ is contained in at least $d$ hyperedges of $H$.
We use the following folklore lemma; see~\cite{Kostochka2015}.

\begin{lemma}\label{lem:full}
    For every pair of integers $r\geq 3$, $d\geq 1$, every $r$-graph $H$ has a $(d+1)$-full subhypergraph $H'$ with
    \[
    |H'|\geq |H|-d|\partial H|.
    \]
\end{lemma}

Recall that an $r$-graph $H$ is $(t,C)$-\emph{sparse} if every set of $t$ vertices of $H$ is contained in at most $C$ hyperedges. 

Propositions~\ref{prop:kOdd} and~\ref{prop:kEven} provide the
structural alternatives used in the supersaturation proof.

\begin{proposition}\label{prop:kOdd}
    Let $k\geq3$ be an odd integer, $\varepsilon>0$, $C_2, C_1\geq 1$ and let $H$ be a $3$-graph with $|H|\geq \left(\left\lfloor\frac{k-1}{2}\right\rfloor+\frac{\varepsilon}{2}\right)\binom{n}{2}$. Then one of the following three conditions holds:
    \vspace{4pt}
    \begin{enumerate}[start=0,label=(A\arabic*), font=\upshape, itemsep=4 pt]
    \item\label{itm: Odd Case linear} There is a $(2,C_2)$-sparse $3$-graph $H_0\subseteq H$ such that $|H_0|\geq \frac{\varepsilon}{4}\binom{n}{2}$. 
    \item\label{itm: Odd Case Codegree} There is a $3$-graph $H_2\subseteq H$ such that $H_2$ is $C_1$-full and $|H_2|\geq \frac{\varepsilon}{8}\binom{n}{2}$.  
    \item\label{itm: Odd Case Shadow} 
    There is a $\left\lfloor\frac{k+1}{2}\right\rfloor$-full $3$-graph
$H_1\subseteq H$ such that
$|H_1|\geq\frac{\varepsilon}{4}\binom{n}{2}$,
$|\partial H_1|\geq\frac{\varepsilon}{8C_1}\binom{n}{2}$,
and every hyperedge $e\in H_1$ contains a pair
$xy\subseteq e$ satisfying $d_H(xy)\geq C_2$.
\end{enumerate}
\end{proposition}

We refer to \ref{itm: Odd Case linear} as the case when $H$ has a large linear subhypergraph, to \ref{itm: Odd Case Codegree} as the case when $H$ has a large subhypergraph with large codegrees, and to \ref{itm: Odd Case Shadow} as the case when $H$ has large shadow.

\begin{proof}[Proof of Proposition~\ref{prop:kOdd}]
Let
\[
 H_0\coloneqq
 \bigl\{xyz\in H:
 \max\{d_H(xy),d_H(yz),d_H(xz)\}\leq C_2\bigr\}.
\]
As $H_0$ is $(2,C_2)$-sparse, we may assume that
$|H_0|<\frac{\varepsilon}{4}\binom{n}{2}$, as otherwise
condition~\ref{itm: Odd Case linear} holds.  Thus
\[
 |H\setminus H_0|\geq
 \left(\left\lfloor\frac{k-1}{2}\right\rfloor
       +\frac{\varepsilon}{4}\right)\binom{n}{2}.
\]
Every hyperedge $e\in H\setminus H_0$ contains a pair
$xy\subseteq e$ with $d_H(xy)>C_2$.

We apply Lemma~\ref{lem:full} with $d=\left\lfloor\frac{k-1}{2}\right\rfloor$ to $H\setminus H_0$ to obtain a $\left\lfloor\frac{k+1}{2}\right\rfloor$-full $3$-graph $H_1\subseteq H\setminus H_0$ such that 
\[
|H_1|\geq |H\setminus H_0|-\left\lfloor\frac{k-1}{2}\right\rfloor\cdot|\partial H|
\geq\frac{\varepsilon}{4}\binom{n}{2},
\]
where the second inequality uses $|\partial H|\leq\binom n2$.

By applying Lemma~\ref{lem:full} again with $d=C_1$ to $H_1$, we obtain a $(C_1+1)$-full $3$-graph $H_2\subseteq H_1$ such that
\[
|H_2|\geq |H_1|-C_1\cdot|\partial H_1|.
\]

The inequalities
$|H_2|\geq|H_1|-C_1|\partial H_1|$ and
$|H_1|\geq\frac{\varepsilon}{4}\binom{n}{2}$ imply that
we either have
$|H_2|\geq\frac{\varepsilon}{8}\binom{n}{2}$,
in which case condition~\ref{itm: Odd Case Codegree} holds, or
$|\partial H_1|\geq\frac{\varepsilon}{8C_1}\binom{n}{2}$.
If the shadow bound in the second alternative holds, then, since
$H_1\subseteq H\setminus H_0$, every hyperedge
$e\in H_1$ contains a pair $xy\subseteq e$ satisfying
$d_H(xy)\geq C_2$, and hence condition~\ref{itm: Odd Case Shadow}
holds.
\end{proof}

\begin{proposition}\label{prop:kEven}
    Let $r,k\geq 3$ be integers, $\varepsilon>0$ and $C_1\geq 1$.
    Assume either that $r\geq4$, or that $r=3$ and $k\geq4$ is even.
    Let $H$ be an $r$-graph on $n$ vertices with $|H|\geq \left(\left\lfloor\frac{k-1}{2}\right\rfloor+\frac{\varepsilon}{2}\right)\binom{n}{r-1}$. Then one of the following two conditions holds:
        \vspace{4pt}
    \begin{enumerate}[label=(B\arabic*), font=\upshape,itemsep=4pt]
    \item\label{itm: Even Case Codegrees} There is an $r$-graph $H_2\subseteq H$ such that $H_2$ is $C_1$-full and $|H_2|\geq \frac{\varepsilon}{4}\binom{n}{r-1}$.  
    \item\label{itm: Even Case Shadow} There is a $\left\lfloor\frac{k+1}{2}\right\rfloor$-full $r$-graph $H_1\subseteq H$ such that $|H_1|\geq \frac{\varepsilon}{2} \binom{n}{r-1}$ and $|\partial H_1|\geq \frac{\varepsilon}{4C_1}\binom{n}{r-1}$.  
    \end{enumerate}
\end{proposition}
We refer to \ref{itm: Even Case Codegrees} as the case when $H$ has a large subhypergraph with large codegrees and to \ref{itm: Even Case Shadow} as the case when $H$ has large shadow.

\begin{proof}[Proof of Proposition~\ref{prop:kEven}]

Applying Lemma~\ref{lem:full} to $H$ with $d=\left\lfloor\frac{k-1}{2}\right\rfloor$, we obtain a $\left\lfloor\frac{k+1}{2}\right\rfloor$-full $r$-graph $H_1\subseteq H$ such that
\[
 |H_1|\geq |H|-\left\lfloor\frac{k-1}{2}\right\rfloor|\partial H|
 \geq\frac{\varepsilon}{2}\binom{n}{r-1},
\]
where the second inequality uses $|\partial H|\leq\binom n{r-1}$.
Now applying Lemma~\ref{lem:full} with $d=C_1$ to $H_1$, we obtain a $(C_1+1)$-full $r$-graph $H_2\subseteq H_1$ such that $|H_2|\geq |H_1|-C_1\cdot|\partial H_1|$. The inequalities $|H_2|\geq |H_1|-C_1|\partial H_1|$ and $|H_1|\geq \frac{\varepsilon}{2}\binom{n}{r-1}$ imply that either $|H_2|\geq \frac{\varepsilon}{4}\binom{n}{r-1}$, which gives~\ref{itm: Even Case Codegrees}, or $|\partial H_1|\geq \frac{\varepsilon
}{4C_1} \binom{n}{r-1}$ and~\ref{itm: Even Case Shadow} holds,  as desired.
\end{proof}

\subsection{Graph regularity tools}

Let $G$ be a graph and let $A,B\subseteq V(G)$ be disjoint non-empty
sets. We write
\[
 d_G(A,B)\coloneqq\frac{e_G(A,B)}{|A||B|},
\]
where $e_G(A,B)$ is the number of edges of $G$ with one endpoint in
$A$ and the other in $B$. For $0<\eta<1$, the pair $(A,B)$ is
\emph{$\eta$-regular} if
\[
 |d_G(A',B')-d_G(A,B)|\leq\eta,
\]
whenever $A'\subseteq A$ and $B'\subseteq B$ satisfy
$|A'|\geq\eta|A|$ and $|B'|\geq\eta|B|$.

We use the following standard form of Szemer\'edi's regularity lemma;
see~\cite{Szemeredi1978Regular} and
\cite[Lemma~40]{kuhn2009embedding}.

\begin{lemma}[Szemer\'edi Graph Regularity Lemma]
\label{lem:graph-regularity}
For every $0<\eta<1$ and every positive integer $m_0$, there are
integers $M=M(\eta,m_0)$ and $n_0=n_0(\eta,m_0)$ such that every graph
$G$ on $n\geq n_0$ vertices has a partition
\[
 V(G)=V_0\cup V_1\cup\cdots\cup V_m
\]
with the following properties:
\begin{enumerate}[label=\textnormal{(R\arabic*)},itemsep=4pt]
\item $m_0\leq m\leq M$;
\item $|V_0|\leq\eta n$ and $|V_1|=\cdots=|V_m|$;
\item all but at most $\eta m^2$ of the pairs $(V_i,V_j)$ with
$1\leq i<j\leq m$ are $\eta$-regular.
\end{enumerate}
\end{lemma}

The next lemma is a graph counting lemma in the form used in
Subsection~\ref{subsec:linear subgraph}.

\begin{lemma}[Graph counting lemma
{\cite[Lemma~2.4]{conlon2013graph}}]
\label{lem:graph-counting}
Let $J$ be a graph with vertex set $[h]$ and $m$ edges, and let
$0<\rho,d<1$ satisfy
\[
 \rho\leq\frac{d^h}{4h}.
\]
Let $G$ be a graph, and let $W_1,\ldots,W_h\subseteq V(G)$ satisfy
$|W_i|\geq\rho^{-1}$ for every $i\in[h]$. Suppose that, for every
$ij\in E(J)$, the sets $W_i$ and $W_j$ are disjoint and the pair
$(W_i,W_j)$ is $\rho$-regular with density greater than $d$.
Then there are at least
\[
 2^{-h}d^m\prod_{i=1}^h|W_i|
\]
injective maps $\phi:[h]\to V(G)$ such that $\phi(i)\in W_i$ for
every $i\in[h]$ and $\phi(i)\phi(j)\in E(G)$ for every
$ij\in E(J)$.
\end{lemma}

\subsection{Container tools}
For a hypergraph
$\mathcal{H}$, a set $I \subseteq V(\mathcal{H})$ is called \emph{independent} if there is no hyperedge $e \in \mathcal{H}$ such that $e \subseteq I$. We write $\mathcal{I}(\mathcal{H})$ for the collection of independent sets of $\mathcal{H}$.

We will use the hypergraph container method developed independently in ~\cite{balogh2015independent} and~\cite{saxton2015hypergraph}. In par\-ticular, we need the following variant from~\cite{Balogh2018HypergraphContainers}. Recall that the maximum $\ell$-codegree of a hypergraph $\HH$ is defined as $\Delta_{\ell}(\HH)\coloneqq \max \{d_{\HH}(f) : f \subset V(\HH), |f|=\ell\}$.

\begin{lemma}\label{lem:containers}
    Let $k\in \mathbb{N}$ and set $\delta\coloneqq 2^{-k(k+1)}$. Let $\HH$ be a $k$-graph and suppose that
    \[
    \Delta_{\ell}(\HH)\leq \left(\frac{b}{\lvert V(\HH)\rvert}\right)^{\ell-1}\frac{|\HH|}{m}
    \]
    for some $b,m\in \mathbb{N}$ and every $\ell\in [k]$. Then there exists a collection $\mathcal{C}(\HH)$ of subsets of $V(\HH)$ and a function $f:\mathcal{P}(V(\HH))\rightarrow \mathcal{C}(\HH)$ such that:
        \vspace{4pt}
    \begin{enumerate}[label=(C\arabic*), font=\upshape,itemsep=4pt]
        \item\label{itm: BMS property i} for every $I\in \mathcal{I}(\HH)$, there exists $S\subseteq I$ with $|S|\leq (k-1)b$ and $I\subseteq f(S)$;
        \item\label{itm: BMS property ii} $|\HH'|\leq v(\HH) -\delta m$ for every $\HH'\in\mathcal{C}(\HH)$.
    \end{enumerate}
\end{lemma}

We state it in 
 the form in which we apply
Lemma~\ref{lem:containers}.  The definitions using ceilings and floors
ensure that the parameters $b$ and $m$ in Lemma~\ref{lem:containers}
are integers.

\begin{corollary}\label{cor:containers-balanced}
Let $H$ be an $r$-graph and let $\KK$ be a collection of copies of
$C_k^{(r)}$ in $H$.  Suppose that $|H|\geq D\geq4(k-1)$,
$s\geq2K$, and
\[
 |\KK|\geq\frac{D^{k-1}s}{K},
 \qquad
 \Delta_\ell(\KK)\leq D^{k-\ell}
 \quad\quad\text{for every }\quad\ell\in[k].
\]
Then there is a collection $\mathcal C(H)$ of subhypergraphs of $H$ such
that every $C_k^{(r)}$-free subhypergraph of $H$ is contained in a member of
$\mathcal C(H)$,
\[
 |\mathcal C(H)|
 \leq 2^{h(2(k-1)/D)|H|},
\]
and every $H'\in\mathcal C(H)$ satisfies
\[
 |H'|\leq |H|-\frac{\delta s}{2K},
 \qquad \delta=2^{-k(k+1)}.
\]
\end{corollary}

\begin{proof}
Set
\[
 b=\left\lceil\frac{|H|}{D}\right\rceil,
 \qquad
 m=\left\lfloor\frac{s}{K}\right\rfloor.
\]
The hypotheses imply $|H|\geq D$ and hence
$b\leq2|H|/D$, while $m\geq s/(2K)$.  Also $m\leq s/K$, so, for every
$\ell\in[k]$,
\[
 \left(\frac b{|H|}\right)^{\ell-1}\frac{|\KK|}{m}
 \geq D^{1-\ell}\frac{K}{s}|\KK|
 \geq D^{k-\ell}
 \geq\Delta_\ell(\KK).
\]
Thus Lemma~\ref{lem:containers} can be applied. Let $f$ be the function supplied by Lemma~\ref{lem:containers}, and
define
\[
 \mathcal C(H)\coloneqq
 \{f(S):S\subseteq H,\ |S|\leq(k-1)b\}.
\]
By condition~\ref{itm: BMS property i}, every independent set in the
$k$-graph $\KK$ is contained in a member of $\mathcal C(H)$.  Moreover,
$|\mathcal C(H)|$ is at most the number of sets of at most
$(k-1)b\leq2(k-1)|H|/D$ hyperedges of $H$.
Proposition~\ref{prop:Entropy} therefore gives the asserted bound on
$|\mathcal C(H)|$.  Finally, condition~\ref{itm: BMS property ii} and
$m\geq s/(2K)$ imply that every $H'\in\mathcal C(H)$ satisfies
\[
 |H'|\leq |H|-\delta m
 \leq |H|-\frac{\delta s}{2K},
\]
as required.
\end{proof}

We also use the following container result from
\cite[Proposition~4.1]{Balogh2019NumberCyclefree} to obtain the
initial collection of dense containers.

\begin{theorem}\label{thm:BNS}
For every pair of integers $k,r\geq 3$, there exists a constant $L=L(r,k)>2$ such that the following holds for all $n\in \mathbb{N}$. Let $H$ be an $n$-vertex $r$-graph with $|H|=tn^{r-1}$ for some $t\geq L$. Then there exists a collection $\mathcal{C}(H)$ of at most
\begin{equation}
\label{eq:countainerbound2.7}
\exp\left(\frac{Ln^{r-1}}{t^{1/4}}\right)
\end{equation}
subhypergraphs of $H$ such that
    \vspace{4pt}
\begin{enumerate}[label=\textnormal{(D\arabic*)}, itemsep=4pt]
    \item\label{itm: BNS cond i}
    every $C_k^{(r)}$-free subhypergraph $H'\subseteq H$ is contained in some $G\in \mathcal{C}(H)$, and
    \item\label{itm: BNS cond ii}
    $|G|\leq \left(1-\frac{1}{L}\right)|H|$ for every $G\in \mathcal{C}(H)$.
\end{enumerate}
\end{theorem}
The following proposition is obtained by iteratively applying
Theorem~\ref{thm:BNS}. 

\begin{proposition}\label{prop:Dense regime}
For every pair of integers $k,r\geq 3$ and every $\varepsilon>0$, there exists a constant $C_0=C_0(r,k,\varepsilon)$ such that for every $n\in\mathbb N$ there is a collection $\mathcal{C}$ of $r$-graphs such that
    \vspace{4pt}
\begin{enumerate}[label=\textnormal{(E\arabic*)}, itemsep=4pt]
    \item\label{itm: Dense regime cond 1}
    for every $C_k^{(r)}$-free $r$-graph $H'$ with vertex set $[n]$, there exists an $r$-graph $H\in \mathcal{C}$ such that $H'\subseteq H$,
    \item\label{itm: Dense regime cond 2}
    $|\mathcal{C}|\leq 2^{\frac{\varepsilon}{4}\binom{n}{r-1}}$, and \item\label{itm: Dense regime cond 3}
    $|H|\leq C_0\binom{n}{r-1}$ for every $H\in \mathcal{C}$.
\end{enumerate}
\end{proposition}
\begin{proof}
Let $L=L(k,r)$ be the constant from Theorem~\ref{thm:BNS}, and set
$q\coloneqq1-1/L$.  Choose a constant $A\geq L/q$ so large
that
\begin{equation}\label{eq:dense-choice-A}
 \frac{8L^2(r-1)^{r-1}}{A^{1/4}}
 \leq\frac{\varepsilon\log 2}{4}.
\end{equation}
Since $\binom nr/n^{r-1}\to\infty$ as $n\to\infty$, we may choose an
integer $N\geq r$ such that
\begin{equation}\label{eq:dense-choice-N}
 \frac{\binom nr}{n^{r-1}}>A
 \qquad\text{for every }n\geq N,
\end{equation}
and set
\[
 C_0\coloneqq\max\{A(r-1)^{r-1},N\}.
\]

Fix $n\in\mathbb N$.  If $n<N$, let $G_0$ be the complete $r$-graph
on $[n]$ and set $\mathcal C\coloneqq\{G_0\}$.  This collection
satisfies condition~\ref{itm: Dense regime cond 1}, and it satisfies
condition~\ref{itm: Dense regime cond 2} because
$|\mathcal C|=1\leq2^{(\varepsilon/4)\binom{n}{r-1}}$.  If $n<r$,
then $G_0$ has no
hyperedges.  If $r\leq n<N$, then
\[
 |G_0|=\binom nr
 =\frac{n-r+1}{r}\binom{n}{r-1}
 \leq N\binom{n}{r-1}
 \leq C_0\binom{n}{r-1}.
\]
Thus condition~\ref{itm: Dense regime cond 3} also holds.  We may
therefore assume for the remainder of the proof that $n\geq N$.

Let $G_0$ be the complete $r$-graph on $[n]$, and define
\[
 t_0\coloneqq\frac{|G_0|}{n^{r-1}}
 \qquad\text{and}\qquad
 t_i\coloneqq q^it_0\quad\text{for every }i\geq1.
\]
By~\eqref{eq:dense-choice-N}, we have $t_0>A$.  Let $m$ be the
smallest positive integer for which $t_m\leq A$.  The minimality of
$m$ gives
\begin{equation}\label{eq:dense-last-threshold}
 t_m=qt_{m-1}>qA\geq L.
\end{equation}

Set $\mathcal C_0\coloneqq\{G_0\}$.  For every
$i\in\{0,\ldots,m-1\}$, having defined $\mathcal C_i$, define
\[
 \mathcal C_{i+1}
 \coloneqq
 \{G\in\mathcal C_i:|G|<t_{i+1}n^{r-1}\}
 \cup
 \bigcup_{\substack{G\in\mathcal C_i\\
                      |G|\geq t_{i+1}n^{r-1}}}
 \mathcal C(G),
\]
where $\mathcal C(G)$ is the collection supplied by
Theorem~\ref{thm:BNS}.  We claim that
\begin{equation}\label{eq:dense-level-properties}
 |G|\leq t_i n^{r-1}
 \quad\text{for every }G\in\mathcal C_i.
\end{equation}
This holds for $i=0$.  If it holds for $i$, then every member retained
from $\mathcal C_i$ has fewer than $t_{i+1}n^{r-1}$ hyperedges, while
condition~\ref{itm: BNS cond ii} gives
\[
 |G'|\leq q|G|\leq qt_i n^{r-1}=t_{i+1}n^{r-1}
\]
for every $G'\in\mathcal C(G)$.  This proves
\eqref{eq:dense-level-properties} for $i+1$.  Notice also that
Theorem~\ref{thm:BNS} may be applied whenever $G$ is replaced, since
\eqref{eq:dense-last-threshold} gives
$|G|/n^{r-1}\geq t_{i+1}\geq t_m\geq L$.

We next bound the size of each collection.  Fix
$i\in\{0,\ldots,m-1\}$.  A member $G\in\mathcal C_i$ that is retained
contributes one member to $\mathcal C_{i+1}$.  If $G$ is replaced and
$s\coloneqq |G|/n^{r-1}$, then $s\geq t_{i+1}$, so
Theorem~\ref{thm:BNS} gives
\[
 |\mathcal C(G)|
 \leq\exp\left(\frac{Ln^{r-1}}{s^{1/4}}\right)
 \leq\exp\left(\frac{Ln^{r-1}}{t_{i+1}^{1/4}}\right).
\]
If a member of $\mathcal C_{i+1}$ is contributed by more than one
member of $\mathcal C_i$, counting all these contributions only
increases the upper bound.  Therefore,
\[
 |\mathcal C_{i+1}|
 \leq |\mathcal C_i|
       \exp\left(\frac{Ln^{r-1}}{t_{i+1}^{1/4}}\right).
\]
Multiplying these $m$ inequalities gives
\begin{equation}\label{eq:dense-level-count}
 |\mathcal C_m|
 \leq
 \exp\left(
 n^{r-1}\sum_{i=1}^m\frac{L}{t_i^{1/4}}
 \right).
\end{equation}

For every $i\in[m]$, we have
$t_i=q^{i-m}t_m>q^{i-m+1}A$.  Since $q\geq1/2$ and
$1-q^{1/4}\geq1/(4L)$, it follows that
\[
 \sum_{i=1}^m\frac{L}{t_i^{1/4}}
 \leq \frac{Lq^{-1/4}}{A^{1/4}}
       \sum_{j\geq0}q^{j/4}
 \leq\frac{8L^2}{A^{1/4}}.
\]
Combining this estimate with~\eqref{eq:dense-level-count} shows that
$|\mathcal C_m|$ is at most
\[
 \exp\left(\frac{8L^2}{A^{1/4}}n^{r-1}\right)
 \leq 2^{\frac{\varepsilon}{4}(n/(r-1))^{r-1}}
 \leq2^{\frac{\varepsilon}{4}\binom{n}{r-1}},
\]
where we used~\eqref{eq:dense-choice-A} and
$\binom{n}{r-1}\geq(n/(r-1))^{r-1}$.
At every step, condition~\ref{itm: BNS cond i} ensures that replacing
$G$ by $\mathcal C(G)$ preserves the covering property.  Since
$\mathcal C_0$ covers every $C_k^{(r)}$-free $r$-graph on $[n]$, the
collection $\mathcal C_m$ satisfies condition~\ref{itm: Dense regime cond 1}.

Finally, \eqref{eq:dense-level-properties} and $t_m\leq A$ show that
every member of $\mathcal C_m$ has at most $An^{r-1}$ hyperedges.  Since
$\binom{n}{r-1}\geq(n/(r-1))^{r-1}$, it has at most
$C_0\binom{n}{r-1}$ hyperedges.  Thus $\mathcal C\coloneqq\mathcal C_m$
satisfies all three assertions.
\end{proof}

Throughout the remainder of the paper, we assume that $\varepsilon>0$
and that $k$ and $r$ are fixed. Let $C_0=C_0(\varepsilon,k,r)$ be the
constant given by Proposition~\ref{prop:Dense regime}. We choose
positive integers $C_1$ and $C_2$, and then $n$, according to the
hierarchy
\begin{equation}\label{defi:constant hierarchy}
\frac{1}{n}\ll \frac{1}{C_2}\ll \frac{1}{C_1}
\ll \frac{1}{C_0},\ \varepsilon,\ \frac{1}{k},\ \frac{1}{r}.
\end{equation}
Thus, after $\varepsilon$, $k$, $r$, and $C_0$ have been fixed, we first choose $C_1$ sufficiently large, then choose $C_2$ sufficiently large in terms of all previously fixed parameters, and finally take $n$ sufficiently large in terms of all these parameters.

We prove in Section~\ref{sec:EnumerationProof} the following result
about containers. It directly implies
Theorem~\ref{thm:numberOfCycleFree} when $(r,k)\neq(3,3)$.

\begin{theorem}\label{thm:MainContainer}
Let $\varepsilon>0$ and let $r,k\geq3$ be integers with
$(r,k)\neq(3,3)$. Then there exists an integer
$n_0=n_0(\varepsilon,r,k)$ such that, for every $n>n_0$, there is a
collection $\mathcal{C}$ of $r$-graphs satisfying the following:
\vspace{4pt}
\begin{enumerate}[label=\textnormal{(F\arabic*)},itemsep=4pt]
    \item\label{itm: Main container 1}
    every $C_k^{(r)}$-free $r$-graph with vertex set $[n]$ is contained in some $H\in\mathcal{C}$;
    \item\label{itm: Main container 2}
    $|H|\leq \left(\left\lfloor\frac{k-1}{2}\right\rfloor+\frac{\varepsilon}{2}\right)\binom{n}{r-1}$ for every $H\in\mathcal{C}$; 
    \item\label{itm: Main container 3}
    $|\mathcal{C}|\leq 2^{\frac{\varepsilon}{2}\binom{n}{r-1}}$.
\end{enumerate}
\end{theorem}

\subsection{Tools for hypergraphs with large shadow}

In the subsequent sections, we obtain copies of $C_k^{(r)}$ from
complete multipartite hypergraphs contained in the shadow $\partial H$.
The next two propositions give quantitative lower bounds for the
number of such multipartite hypergraphs when the shadow has at least
$\alpha n^q$ hyperedges. Their proofs rely on a
K\H ov\'ari--S\'os--Tur\'an-type argument.

For integers $q,t\geq2$, let $K_{t,\ldots,t}^{(q)}$ denote the
complete $q$-partite $q$-graph with $t$ vertices in each part.

We use the following consequence of Erd\H{o}s's theorem on complete
multipartite hypergraphs; see, for example,
\cite[Proposition~3.6]{Kostochka2015}.

\begin{theorem}[Erd\H{o}s]
\label{thm:Erdos-complete-partite}
For every $q,t\geq2$ and every $\beta>0$, there is an integer
$N_0=N_0(q,t,\beta)$ such that every $q$-graph $J$ on $N\geq N_0$
vertices with at least $\beta\binom Nq$ hyperedges contains a copy of
$K_{t,\ldots,t}^{(q)}$.
\end{theorem}

\begin{proposition}\label{prop:many-complete-partite}
For every $q,t\geq2$ and every $\alpha>0$, there are constants
$\gamma=\gamma(q,t,\alpha)>0$ and $n_0$ such that every $n$-vertex
$q$-graph $F$ with $|F|\geq\alpha n^q$ contains at least
$\gamma n^{qt}$ distinct copies of $K_{t,\ldots,t}^{(q)}$ whenever
$n\geq n_0$.
\end{proposition}

\begin{proof}
Replacing $\alpha$ by a smaller positive constant if necessary, we may
assume that $q!\alpha\leq1$.  By
Theorem~\ref{thm:Erdos-complete-partite}, there is an integer
$N=N(q,t,\alpha)\geq qt$ such that every $N$-vertex $q$-graph with at
least $(q!\alpha/2)\binom Nq$ hyperedges contains a copy of
$K_{t,\ldots,t}^{(q)}$.

Choose an $N$-element subset $U$ of $V(F)$ uniformly at random.  Then
each hyperedge $e\in F$ is contained in $U$ with probability
${\binom{n-q}{N-q}}/{\binom nN}
={\binom Nq}/{\binom nq}$.
Consequently, by linearity of expectation, the assumption
$|F|\geq\alpha n^q$, and the inequality
$\binom nq\leq n^q/q!$, we have
\[
 \mathbb E|F[U]|
 =|F|\frac{\binom Nq}{\binom nq}
 \geq\alpha n^q\frac{\binom Nq}{\binom nq}
 \geq q!\alpha\binom Nq.
\]
Let $p$ be the proportion of $N$-element subsets $U$ of $V(F)$ satisfying
\[
 |F[U]|\geq\frac{q!\alpha}{2}\binom Nq.
\]
Since $|F[U]|\leq\binom Nq$ for every $U$, we have
\[
 \mathbb E|F[U]|
 \leq
 p\binom Nq+(1-p)\frac{q!\alpha}{2}\binom Nq.
\]
Comparing this upper bound with the preceding lower bound for
$\mathbb E|F[U]|$ gives
\[
 q!\alpha
 \leq p+(1-p)\frac{q!\alpha}{2},
\]
and therefore
\begin{equation}\label{eq:many-partite-good-N-element-sets}
 p\geq\frac{q!\alpha}{2-q!\alpha}
 \geq\frac{q!\alpha}{2}.
\end{equation}
Each of these $N$-element subsets contains a copy of
$K_{t,\ldots,t}^{(q)}$, whereas a fixed copy is contained in exactly
$\binom{n-qt}{N-qt}$ choices of $U$.  Double counting the pairs
consisting of such an $N$-element subset and a copy contained in it,
and using \eqref{eq:many-partite-good-N-element-sets}, shows that the number of distinct
copies is at least
\[
 \frac{q!\alpha}{2}
 \frac{\binom nN}{\binom{n-qt}{N-qt}}
 =\frac{q!\alpha}{2\binom N{qt}}\binom n{qt}
 \geq\gamma n^{qt}
\]
for a suitable $\gamma>0$ and all sufficiently large $n$.
\end{proof}

For the proof of Lemma~\ref{lem: Supersat Large shadow odd}, we also
need the following quantitative   version of Proposition~\ref{prop:many-complete-partite} for graphs.

\begin{proposition}\label{prop: Many Ktt's}
For every $\alpha>0$ and integer $t\geq2$, there exists
$n_0=n_0(\alpha,t)$ such that every $n$-vertex graph with
$n\geq n_0$ and at least $\alpha n^2$ edges contains at least
\[
 \frac12\left(\frac{2\alpha}{e}\right)^{t^2}
 \frac{n^{2t}}{t^{2t}}
\]
 copies of $K_{t,t}$.
\end{proposition}

\begin{proof}
Let $\mathcal S$ be the set of copies of $K_{1,t}$.  Two applications
of Jensen's inequality give
\[
 |\mathcal S|
 =\sum_v\binom{d(v)}t
 \geq n\binom{2\alpha n}t
\]
and then the number of copies of $K_{t,t}$ is at least
\[
 \frac12\binom nt
 \binom{|\mathcal S|/\binom nt}{t}.
\]
Using $\binom xy\geq(x/y)^y$ and
$\binom nt\leq(en/t)^t$ in the last displayed inequality  yields the asserted
bound for all sufficiently large $n$.
\end{proof}

\subsection{Tools for linear triangles in uniformity three}
\label{subsec:C33-tools}

In this subsection, let $T\coloneqq C_3^{(3)}$.  We state the following four results
 used in Section~\ref{sec:C33-enumeration}.  The
first is the exact extremal theorem of Frankl and F\"uredi, stated in
our notation.

\begin{theorem}[Frankl--F\"uredi
{\cite[Theorem~3.3]{FranklFuredi1987}}]
\label{thm:C33-FF-extremal}
For all sufficiently large $n$, every $T$-free $3$-graph $H$ on
$[n]$ satisfies
\[
 |H|\leq\binom{n-1}{2}.
\]
Equality holds for the full star consisting of all triples containing
one fixed vertex.  In particular,
\[
 \mathrm{ex}_3(n,T)=\binom{n-1}{2}
\]
for all sufficiently large $n$.
\end{theorem}

Recall that a $3$-graph is linear if every pair of two distinct hyperedges intersect
in at most one vertex.  We use the Ruzsa--Szemer\'edi
$(6,3)$-theorem in the following equivalent linear-hypergraph form.

\begin{theorem}[Ruzsa--Szemer\'edi
{\cite{RuzsaSzemeredi1978}}]
\label{thm:C33-RS}
There is a function $\eta:\mathbb N\to[0,1]$, with
$\eta(n)\to0$, such that every linear $T$-free $3$-graph on $n$
vertices has at most $\eta(n)n^2$ hyperedges.
\end{theorem}

A \emph{digraph} $D$ is a pair $(V(D),A(D))$, where
\[
 A(D)\subseteq\{(x,y):x,y\in V(D),\ x\ne y\}.
\]
The elements of $A(D)$ are called \emph{arcs}; we write $x\to y$ for
the arc $(x,y)$. It is possible for both $x\to y$ and $y\to x$ to be present simultaneously.  We write
$a(D)\coloneqq|A(D)|$.  A \emph{directed triangle} is a set of three distinct
vertices $x,y,z$ for which $x\to y$, $y\to z$, and $z\to x$ are arcs.

\begin{theorem}[Brown--Harary
{\cite{BrownHarary1970}}]
\label{thm:C33-Brown-Harary}
If a digraph $D$ on $n$ vertices contains no directed triangle, then
\[
 a(D)\leq\left\lfloor\frac{n^2}{2}\right\rfloor.
\]
\end{theorem}

Finally, Theorem~\ref{thm:C33-multicolour-counting} is a
special case of the multicolour counting theorem of Falgas--Ravry,
O'Connell and Uzzell~\cite{FalgasRavryOConnellUzzell2019} used here.  

A \emph{four-colouring of the pairs of $[n]$} is a function
$c:\binom{[n]}2\to[4]$.  For every $n\geq1$, let $\mathcal Q_n$ be a
family of four-colourings of the pairs of $[n]$.  The sequence
$\mathcal Q=(\mathcal Q_n)_{n\geq1}$ is \emph{order-hereditary} if the
following holds.  Whenever
$m\leq n$, $c\in\mathcal Q_n$, and
$\phi:[m]\to[n]$ is an increasing injection, meaning that
$\phi(i)<\phi(j)$ whenever $i<j$, the colouring
$c\vert_\phi$ defined by
\[
 (c\vert_\phi)(ij)\coloneqq c(\phi(i)\phi(j))
 \qquad (ij\in\tbinom{[m]}2)
\]
belongs to $\mathcal Q_m$.

A \emph{four-colouring template} $\Phi$ on $[n]$ assigns a nonempty
set $\Phi_{xy}\subseteq[4]$, called the \emph{palette at $xy$}, to
every pair $xy\in\binom{[n]}2$.  A four-colouring $c$ is a
\emph{realization} of $\Phi$ if $c(xy)\in\Phi_{xy}$ for every pair
$xy$.  The template is \emph{$\mathcal Q_n$-valid} if all of its
realizations belong to $\mathcal Q_n$.  The base-$4$ entropy of $\Phi$
is
\[
 \operatorname{Ent}_4(\Phi)
 \coloneqq\sum_{xy\in\binom{[n]}2}\log_4|\Phi_{xy}|.
\]

\begin{theorem}[Multicolour counting theorem
{\cite[Corollary~2.15]
{FalgasRavryOConnellUzzell2019}}]
\label{thm:C33-multicolour-counting}
Let $\mathcal Q=(\mathcal Q_n)_{n\geq1}$ be an order-hereditary
sequence such that, for every $n$, $\mathcal Q_n$ is a nonempty family
of four-colourings of the pairs of $[n]$.  Then the limit
\[
 \pi(\mathcal Q)
 \coloneqq\lim_{n\to\infty}
 \frac{1}{\binom n2}
 \max\bigl\{\operatorname{Ent}_4(\Phi):
       \Phi\text{ is $\mathcal Q_n$-valid}\bigr\}
\]
exists, and
\[
 |\mathcal Q_n|
 =4^{(\pi(\mathcal Q)+o(1))\binom n2}.
\]
\end{theorem}
The theorem applies to order-hereditary sequences of families of
digraphs by encoding the four possible states of each pair as four
colours.

\section{Constructing cycles from shadows}\label{sec:Cycle from shadow}
In this section, we construct copies of $C_k^{(r)}$ from complete
$(r-1)$-partite $(r-1)$-graphs contained in the shadow $\partial H$.
The precise codegree hypotheses and the required part sizes are stated
in Lemmas~\ref{lem:Large Shadow},
\ref{lem:cycle-from-shadow-standard}, \ref{lem:Large Shadow odd} and
\ref{lem:a copy of Ck from Ktt}.
The arguments build on ideas from~\cite{Kostochka2015}.  We now provide
a brief overview of the proof ideas used in this section.

The first step is supplied by Lemmas~\ref{lem:Large Shadow}
and~\ref{lem:Large Shadow odd}.  In their respective parameter
ranges, both lemmas produce a set $T\subseteq V(H)$ and a large
$(r-1)$-graph $F\subseteq\partial H$ with
$V(F)\cap T=\varnothing$, such that every $f\in F$ has several
extension vertices in $T$.  The second lemma also gives the additional
codegree information needed when $r=3$ and $k$ is odd.  Both proofs
are based on placing each vertex independently in $T$ with probability
$1/2$ and taking $F$ to be a collection of $(r-1)$-edges
$f\in\partial H$ such that $f\subseteq V(H)\setminus T$ and there are
sufficiently many vertices $x\in T$ for which $f\cup\{x\}\in H$.
We call such vertices the available extension vertices of $f$.

We now describe how these available extensions are used to construct
copies of $C_k^{(r)}$, distinguishing three ranges of $(r,k)$.
When $k\geq4$ and either $r\geq4$, or $r=3$ and $k$ is even,
Lemma~\ref{lem:KMV-cycle-from-shadow} constructs a copy of
$C_k^{(r)}$ by extending hyperedges of a complete multipartite
subhypergraph of $F$.  The case $r\geq4$ and $k=3$ is not covered by
that lemma.  For this case, we prove the new
Lemma~\ref{lem:triangle-from-shadow}, which shows that the $r$-graph
obtained by assigning two extension vertices outside the multipartite
graph to every hyperedge of a complete $(r-1)$-partite
$(r-1)$-graph, and taking all the corresponding extensions, contains a
copy of $C_3^{(r)}$, provided that every part has at least three
vertices. 

When $r=3$ and $k\geq5$ is odd, the proof splits according to whether
some vertex of $T$ is an available extension vertex for many pairwise
disjoint edges of $F$.  If such a vertex exists, we extend a path in
$F$ to a copy of $P_{k-2}^{(3)}$ and complete it to a copy of
$C_k^{(3)}$ using two hyperedges containing that vertex.  Otherwise,
we first construct a copy of $C_{k-1}^{(3)}$ and use the additional
codegree information in Lemma~\ref{lem:Large Shadow odd} to replace one of
its hyperedges by two hyperedges, thereby obtaining a copy of
$C_k^{(3)}$.

Recall from the notation subsection that $L_G(f)$ is the set of
extension vertices outside $V(G)$ available to $f\in G$, and that
$\widehat G$ is the $r$-graph consisting of all corresponding
extensions $f\cup\{x\}$.

We need the following lemma of Kostochka, Mubayi and
Verstra\"ete~\cite{Kostochka2015}. It gives a linear $r$-uniform path
among the extensions of the hyperedges of a sufficiently long path in the
shadow.

\begin{lemma}[\cite{Kostochka2015}]\label{lem: A P_k^r from a path in the shadow}
Let $k\geq 3$, $H$ be an $r$-graph, and let $P$ be a path in  $\partial H$ consisting of the hyperedges $e_0,\ldots, e_{2^{2\left\lfloor \frac{k-1}{2}\right\rfloor+1}-1}$. If $\lvert L_{P}(e)\rvert\geq \left\lfloor\frac{k+1}{2}\right\rfloor$ for all $e\in P$, then $\hat{P}$ contains a copy of $P_{k}^{(r)}$.
Moreover, the first hyperedge of this copy contains $e_0$.
\end{lemma}

We also state a lemma of Kostochka, Mubayi and
Verstra\"ete in the quantitative form used here.  We apply it directly
when $r\geq4$ and $k\geq4$, and when $r=3$ and $k$ is even.  Its final
hypothesis is the additional condition needed when $r=3$ and $k$ is
odd.

\begin{lemma}[Kostochka--Mubayi--Verstra\"ete
{\cite[Lemma~4.4]{Kostochka2015}}]
\label{lem:KMV-cycle-from-shadow}
Let $r\geq3$ and $k\geq4$, and let
$t\coloneqq(r2^k)^{4r}$.
Let $H$ be an $r$-graph and let $G\subseteq\partial H$ be a complete
$(r-1)$-partite $(r-1)$-graph with $t$ vertices in each part. Suppose
that
\[
 |L_G(f)|\geq\left\lfloor\frac{k+1}{2}\right\rfloor
 \qquad\text{for every }f\in E(G).
\]
If $r=3$ and $k$ is odd, suppose additionally that, for every
$xy\in E(G)$, there is a vertex $\alpha\in L_G(xy)$ such that
\[
 \min\{d_H(x\alpha),d_H(y\alpha)\}\geq2
 \quad\quad\text{and}\quad\quad
 \max\{d_H(x\alpha),d_H(y\alpha)\}\geq3k+1.
\]
Then $H$ contains a copy of $C_k^{(r)}$.
\end{lemma}

The result in~\cite[Lemma~4.4]{Kostochka2015} is stated for all
sufficiently large $t$.  The proof there requires
\[
 s=2^{k-2}(r-1),
 \qquad
 \left\lfloor\frac{t}{s+2}\right\rfloor>(s+2)^{r+1}.
\]
This value of $t$ satisfies the requirement because
$s+2\leq r2^k$ and
\[
 \left\lfloor\frac{t}{s+2}\right\rfloor
 \geq (r2^k)^{4r-1}
 >(r2^k)^{r+1}
 \geq(s+2)^{r+1}.
\]

\subsection{The general parameter range}

The following lemma selects a large $(r-1)$-graph
$F\subseteq\partial H$ and a set $T\subseteq V(H)$ disjoint from
$V(F)$ such that every $f\in F$ has at least
$\lfloor(k+1)/2\rfloor$ extension vertices in $T$.

\begin{lemma}\label{lem:Large Shadow}
Let $r,k\geq3$ be integers.  Assume either that $r\geq4$, or that
$r=3$ and $k\geq4$ is even.  Let $\delta>0$, and let $H$ be an
$n$-vertex $\lfloor(k+1)/2\rfloor$-full $r$-graph satisfying
$|\partial H|>\delta n^{r-1}$.  Then there exist a set
$T\subseteq V(H)$ and an $(r-1)$-graph $F\subseteq\partial H$ such
that
\begin{enumerate}[label=\textnormal{(G\arabic*)},itemsep=4pt]
\item\label{itm: lemma Large shadow i}
$V(F)\subseteq V(H)\setminus T$;
\item\label{itm: lemma Large shadow ii}
$|N_H(f)\cap T|\geq\left\lfloor\frac{k+1}{2}\right\rfloor$
for every $f\in F$;
\item\label{itm: lemma Large shadow iii}
$|F|\geq\frac{\delta}{2^{k+r+2}}n^{r-1}$;
\end{enumerate}
\end{lemma}

\begin{proof}
Let $X\subseteq V(H)$ be a random vertex subset obtained by selecting each vertex independently with probability $1/2$. Let
\begin{equation*}
F_X\coloneqq\left\{f\in\partial H: f\subseteq V(H)\setminus X,\text{ and }|N_H(f)\cap X|\geq\left\lfloor\frac{k+1}{2}\right\rfloor\right\}.
\end{equation*}
For each $f\in \partial H$, choose
$\left\lfloor\frac{k+1}{2}\right\rfloor$ distinct hyperedges
$e_1,\ldots,e_{\left\lfloor\frac{k+1}{2}\right\rfloor}\in H$
containing $f$.  The probability that
$f\subseteq V(H)\setminus X$ and $(e_i\setminus f)\subseteq X$ for
every $i\in [\left\lfloor\frac{k+1}{2}\right\rfloor]$ is exactly
$2^{-(r-1)-\left\lfloor\frac{k+1}{2}\right\rfloor}$.
It follows that
\begin{equation*}
\mathbb{E}(|F_X|)\geq |\partial H|\cdot 2^{-(r-1)-\left\lfloor\frac{k+1}{2}\right\rfloor}>\frac{\delta n^{r-1}}{2^{k+r}}.
\end{equation*}
Thus, there exists a set $T\subseteq V(H)$ such that
$
|F_T|\geq \delta n^{r-1}/2^{k+r}.
$
Taking $F\coloneqq F_T$, conditions~\ref{itm: lemma Large shadow i} and~\ref{itm: lemma Large shadow ii} follow from the definition of $F_T$, while
$
|F|\geq \delta n^{r-1}/2^{k+r} \geq \delta n^{r-1}/2^{k+r+2},
$
which gives condition~\ref{itm: lemma Large shadow iii}.
\end{proof}

Call three hyperedges of a hypergraph $G$ a \emph{linear triangle} if they intersect
pairwise in three distinct vertices and have no other intersections. We first prove the extension statement for linear triangles, the only
cycle length not covered by Lemma~\ref{lem:KMV-cycle-from-shadow}.

\begin{lemma}\label{lem:triangle-from-shadow}
Let $q\geq3$, let $G$ be a complete $q$-partite $q$-graph with at least
three vertices in each part, and let $T$ be disjoint from $V(G)$.
Suppose that every $f\in E(G)$ has an assigned {\it list}, a two-element set
$\mathcal L(f)\subseteq T$.
Then the $(q+1)$-graph
\[
 \{f\cup\{z\}:f\in E(G),\ z\in\mathcal L(f)\}
\]
contains a copy of $C_3^{(q+1)}$.
\end{lemma}

\begin{proof}
Write $V_1,\ldots,V_q$ for the parts of $G$.  Let $e_1,e_2,e_3$ be the
hyperedges of a linear triangle in $G$.  A system of distinct
representatives for their lists is a choice of pairwise distinct
vertices
\[
 z_i\in\mathcal L(e_i)\qquad\text{for each }i\in[3].
\]
If such a choice exists, then
$e_1\cup\{z_1\},e_2\cup\{z_2\},e_3\cup\{z_3\}$ form a copy of
$C_3^{(q+1)}$.  Indeed, the three representatives are distinct and
belong to $T$, which is disjoint from $V(G)$, so the pairwise
intersections of the three extended hyperedges are exactly the three
pairwise intersections of $e_1,e_2,e_3$.

Suppose, for a contradiction, that the conclusion of the lemma fails.
The three lists belonging to any linear triangle in $G$ then have no
system of distinct representatives.  Hall's theorem implies that the
three lists are identical. To see this, note that Hall's condition
holds when any one list is considered, since each list has two
elements, and when any two lists are considered, since their union has
at least two elements.  It can therefore fail only when all three
lists are considered together.  Their union then has size at most two,
which forces all three two-element lists to be equal.

Choose disjoint hyperedges $f_1,f_2\in E(G)$; such hyperedges exist
because every part contains at least two vertices.  Fix distinct
indices $i,j\in[q]$, let $x$ be the vertex of $f_1$ in $V_i$, and let
$y$ be the vertex of $f_2$ in $V_j$.  For every
$s\in[q]\setminus\{i,j\}$, choose
$w_s\in V_s\setminus(f_1\cup f_2)$, and put
\[
 h\coloneqq\{x,y\}\cup
          \{w_s:s\in[q]\setminus\{i,j\}\}.
\]
The vertices $w_s$ can be chosen because every part has at least three
vertices.  Since $G$ is complete $q$-partite, $h\in E(G)$.  Moreover,
$h\cap f_1=\{x\}$ and $h\cap f_2=\{y\}$, where $x\neq y$.

We next verify that any two hyperedges $a,b\in E(G)$ satisfying
$a\cap b=\{v\}$ belong to a linear triangle in $G$.  Suppose that
$v\in V_m$.  Since $q\geq3$, we may choose distinct indices
$s,t\in[q]\setminus\{m\}$.  Let $a_s$ be the vertex of $a$ in
$V_s$, and let $b_t$ be the vertex of $b$ in $V_t$.  For every
$u\in[q]\setminus\{s,t\}$, choose
$c_u\in V_u\setminus(a\cup b)$, and define
\[
 c\coloneqq\{a_s,b_t\}\cup
          \{c_u:u\in[q]\setminus\{s,t\}\}.
\]
These choices are possible because every part has at least three
vertices.  Since $G$ is complete $q$-partite, $c\in E(G)$, and
\[
 a\cap b=\{v\},\qquad a\cap c=\{a_s\},\qquad
 b\cap c=\{b_t\}.
\]
The vertices $v,a_s,b_t$ lie in three distinct parts, so $a,b,c$
form a linear triangle.

Applying this fact to $f_1,h$ and to $h,f_2$, we obtain hyperedges
$g_1,g_2\in E(G)$ such that $f_1,h,g_1$ and $h,f_2,g_2$ are linear
triangles.  The consequence of Hall's theorem established above
therefore gives
\[
 \mathcal L(f_1)=\mathcal L(h)=\mathcal L(f_2).
\]
Write this common two-element set as $\{\alpha,\beta\}$.  Then
\[
 f_1\cup\{\alpha\},\qquad
 f_2\cup\{\alpha\},\qquad
 h\cup\{\beta\}
\]
have pairwise intersections $\{\alpha\}$, $\{x\}$, and $\{y\}$,
respectively, and have no other intersections.  Since
$\alpha,x,y$ are distinct, these three extended hyperedges form a
copy of $C_3^{(q+1)}$, a contradiction.
\end{proof}

The preceding two extension lemmas give the following consequence for
all parameter ranges except the case where $r=3$ and $k$ is odd.

\begin{lemma}\label{lem:cycle-from-shadow-standard}
Let $r,k\geq3$ be integers.  Assume that either $r\geq4$, or that
$r=3$ and $k\geq4$ is even, and put $t\coloneqq(r2^k)^{4r}$.
Let $H$ be an $r$-graph and let $G\subseteq\partial H$ be a complete
$(r-1)$-partite $(r-1)$-graph with $t$ vertices in each part.  For
every $f\in E(G)$, let
\[
 \mathcal L(f)\subseteq N_H(f)\setminus V(G),
 \qquad
 |\mathcal L(f)|=\left\lfloor\frac{k+1}{2}\right\rfloor.
\]
Then $H$ contains a copy of $C_k^{(r)}$ in which every hyperedge is of
the form $f\cup\{z\}$ for some $f\in E(G)$ and
$z\in\mathcal L(f)$.
\end{lemma}

\begin{proof}
Let
\[
 H_G\coloneqq\{f\cup\{z\}:f\in E(G),\ z\in\mathcal L(f)\}.
\]
If $r\geq4$ and $k=3$, then
Lemma~\ref{lem:triangle-from-shadow}, applied with $q=r-1$ and
$T=\bigcup_{f\in E(G)}\mathcal L(f)$, gives a copy of $C_3^{(r)}$ in
$H_G$.  In all the remaining parameter ranges,
Lemma~\ref{lem:KMV-cycle-from-shadow}, applied to $H_G$ and
$G\subseteq\partial H_G$, gives a copy of $C_k^{(r)}$ in $H_G$.
The definition of $H_G$ gives the required form of every hyperedge.
\end{proof}

\subsection{Odd cycles in uniformity three}

We now turn to the remaining range, where $r=3$ and $k\geq5$ is odd.
The next lemma again selects a large graph $F\subseteq\partial H'$ and
a set $T\subseteq V(H)$ disjoint from $V(F)$.  In addition, for every
$xy\in F$, it provides an extension vertex in $T$ satisfying the
codegree conditions needed below.

\begin{lemma}\label{lem:Large Shadow odd}
Let $k\geq5$ be odd, let $\delta>0$ and $C_0>0$, and let $C_2$ be an
integer satisfying $C_2>\max\{6C_0/\delta,6k\}$.  Let $H$ be an
$n$-vertex $3$-graph,
and let $H'\subseteq H$ be a $(k+1)/2$-full $3$-graph satisfying
$|\partial H'|>\delta n^2$.  Assume that $|H|\leq C_0n^2$ and that
every hyperedge $xyz\in H'$ satisfies
\[
 \max\{d_H(xy),d_H(yz),d_H(xz)\}\geq C_2.
\]
Then there exist a set $T\subseteq V(H)$ and a graph
$F\subseteq\partial H'$ such that
\begin{enumerate}[label=\textnormal{(H\arabic*)},itemsep=4pt]
\item\label{itm: lemma Large shadow odd i}
$V(F)\subseteq V(H)\setminus T$;
\item\label{itm: lemma Large shadow odd ii}
$|N_{H'}(xy)\cap T|\geq(k+1)/2$ for every $xy\in F$;
\item\label{itm: lemma Large shadow odd iii}
$|F|\geq\delta n^2/2^{k+5}$;
\item\label{itm: lemma Large shadow odd iv}
for every $xy\in F$, there is a vertex
$\alpha\in N_{H'}(xy)\cap T$ such that
\begin{gather*}
 |N_H(x\alpha)\cap T|
 \geq\frac{|N_H(x\alpha)|-1}{2},
 \qquad
 |N_H(y\alpha)\cap T|
 \geq\frac{|N_H(y\alpha)|-1}{2},\\
 \text{and}\qquad
 \max\{d_H(x\alpha),d_H(y\alpha)\}\geq C_2.
\end{gather*}
\end{enumerate}
\end{lemma}

\begin{proof}
Let
$
E\coloneqq \{xy\in\partial H':d_H(xy)<C_2 \}.
$
Since every $xy\in\partial H'\setminus E$ satisfies $d_H(xy)\geq C_2$, we have
\begin{equation*}
C_2|\partial H'\setminus E|\leq\sum_{xy\in\partial H'\setminus E}d_H(xy)\leq 3|H|.
\end{equation*}
It follows that
\begin{equation*}
|\partial H'\setminus E|\leq\frac{3|H|}{C_2}\leq\frac{3C_0n^2}{C_2}\leq\frac{\delta n^2}{2}\leq\frac{|\partial H'|}{2}.
\end{equation*}
Therefore,
\begin{equation}\label{eq:boundE}
|E|\geq\frac{|\partial H'|}{2}\geq\frac{\delta n^2}{2}.
\end{equation}

For every $xy\in E$, fix a vertex $\alpha(xy)\in N_{H'}(xy)$.  In the
remainder of the proof, whenever a pair $xy\in E$ is under
consideration, write $\alpha\coloneqq\alpha(xy)$. Since $xy\alpha\in H'$, the hypothesis that every hyperedge of $H'$ contains a pair whose codegree is at least $C_2$ implies that at least one of $xy$, $x\alpha$, and $y\alpha$ has codegree at least $C_2$ in $H$. As $d_H(xy)<C_2$, we obtain
\begin{equation}\label{eq: defi alpha}
\max\{d_H(x\alpha),d_H(y\alpha)\}\geq C_2.
\end{equation}
Moreover, since $H'$ is $\left\lfloor\frac{k+1}{2}\right\rfloor$-full and $x\alpha,y\alpha\in\partial H'$, we have
\begin{equation*}
d_H(x\alpha)\geq d_{H'}(x\alpha)\geq\left\lfloor\frac{k+1}{2}\right\rfloor\geq2
\end{equation*}
and, similarly, $d_H(y\alpha)\geq2$.

Now let $X\subseteq V(H)$ be a random vertex subset obtained by including each vertex independently with probability $1/2$. Let $F_X$ be the set of pairs $xy\in E$ such that $x,y\notin X$, $\alpha\in X$,
\begin{equation*}
|N_{H'}(xy)\cap X|\geq\left\lfloor\frac{k+1}{2}\right\rfloor,
\end{equation*}
and
\begin{equation}\label{ineq:LowerCodegreesFT}
|N_H(x\alpha)\cap X|\geq\frac{|N_H(x\alpha)|-1}{2}
\quad\quad\text{and}\quad\quad
|N_H(y\alpha)\cap X|\geq\frac{|N_H(y\alpha)|-1}{2}.
\end{equation}

Fix an $xy\in E$. Let $A_1$ denote the event that $x,y\notin X$ and $\alpha\in X$. Then $\mathbb{P}(A_1)=1/8$. Conditional on $A_1$, let $A_2$ denote the event that
\begin{equation*}
|N_{H'}(xy)\cap X|\geq\left\lfloor\frac{k+1}{2}\right\rfloor.
\end{equation*}
Since $\alpha\in N_{H'}(xy)\cap X$ and $d_{H'}(xy)\geq\left\lfloor\frac{k+1}{2}\right\rfloor$, we may choose $\left\lfloor\frac{k+1}{2}\right\rfloor-1$ additional  vertices from $N_{H'}(xy)\setminus\{\alpha\}$. Hence,
\begin{equation*}
\mathbb{P}(A_2\mid A_1)\geq 2^{-\left(\left\lfloor\frac{k+1}{2}\right\rfloor-1\right)}.
\end{equation*}
Conditional on $A_1$, let $A_3$ denote the event that
\begin{equation*}
|(N_H(x\alpha)\cap X)\setminus\{y\}|\geq\frac{|N_H(x\alpha)|-1}{2},
\end{equation*}
and let $A_4$ denote the event
\[
 |(N_H(y\alpha)\cap X)\setminus\{x\}|
 \geq\frac{|N_H(y\alpha)|-1}{2}.
\]
Since
\begin{equation*}
|(N_H(x\alpha)\cap X)\setminus\{y\}|\sim\operatorname{Bin}\left(|N_H(x\alpha)|-1,\frac12\right),
\end{equation*}
we have $\mathbb{P}(A_3\mid A_1)\geq1/2$. Similarly, $\mathbb{P}(A_4\mid A_1)\geq1/2$.

Conditional on $A_1$, the events $A_2$, $A_3$, and $A_4$ are increasing events determined by the remaining independent vertex-selection trials. Therefore, Proposition~\ref{prop:Harris} gives
\begin{equation*}
\mathbb{P}(xy\in F_X)\geq\frac18\cdot 2^{-\left(\left\lfloor\frac{k+1}{2}\right\rfloor-1\right)}\cdot\frac12\cdot\frac12=2^{-\left\lfloor\frac{k+1}{2}\right\rfloor-4}.
\end{equation*}
The bound
$\mathbb P(xy\in F_X)\geq2^{-\lfloor(k+1)/2\rfloor-4}$ and
\eqref{eq:boundE} give
\begin{equation*}
\mathbb{E}(|F_X|)\geq |E|\cdot2^{-\left\lfloor\frac{k+1}{2}\right\rfloor-4}\geq\frac{\delta n^2}{2^{\left\lfloor\frac{k+1}{2}\right\rfloor+5}}\geq\frac{\delta n^2}{2^{k+5}}.
\end{equation*}
Thus, there exist sets $T\subseteq V(H)$ and $F\subseteq E\subseteq\partial H'$ satisfying
\begin{equation}\label{eq:large-shadow-F-lower-bound}
|F|\geq\frac{\delta n^2}{2^{k+5}}.
\end{equation}
Conditions~\ref{itm: lemma Large shadow odd i} and~\ref{itm: lemma Large shadow odd ii} follow from the definition of $F_X$, while~\eqref{eq:large-shadow-F-lower-bound} gives condition~\ref{itm: lemma Large shadow odd iii}. Finally, for every $xy\in F$, the fixed vertex $\alpha(xy)\in N_{H'}(xy)\cap T$ satisfies~\eqref{eq: defi alpha} and~\eqref{ineq:LowerCodegreesFT}. Hence condition~\ref{itm: lemma Large shadow odd iv} also holds.
\end{proof}

We next prepare the extension argument in this odd $3$-uniform
range.  In the proof of the extension lemma for this range, Lemma~\ref{lem:a copy of Ck from Ktt}, we first
construct a copy of $C_{k-1}^{(3)}$ containing a hyperedge
$xy\alpha$.  After interchanging $x$ and $y$ if necessary, we have
$d_H(x\alpha)\geq C_2$.  We replace $xy\alpha$ by hyperedges
$y\alpha\zeta$ and $x\alpha\zeta'$.  The graph
$G_{\{y\alpha\}}^{(1)}$ selects $\zeta$, while
$G_{\{x\alpha\}}^{(2)}$ provides $\lceil C_2/4\rceil$ initial
choices for $\zeta'$.  Conversely, a fixed vertex $\zeta$ has at most
two neighbors in the first graph, and a fixed vertex $\zeta'$ has at
most $C_2$ neighbors in the second.  These bounds control how many of
the resulting cycles can contain a prescribed replacement hyperedge.

We define both auxiliary graphs through the following general
construction.

\begin{definition}\label{defi: Auxiliary graph large shadow}
Let $S_1=\{a_0,\ldots,a_{s-1}\}$ and
$S_2=\{b_0,\ldots,b_{m-1}\}$ be non-empty disjoint sets satisfying
$s\leq m+1$.  For an integer $q$ with $1\leq q\leq m$, let
$G_q(S_1,S_2)$ be the bipartite graph with parts $S_1$ and $S_2$ in
which
\[
 a_i b_{(iq+j)\bmod m}\in E\bigl(G_q(S_1,S_2)\bigr)
\]
for every $0\leq i<s$ and $0\leq j<q$.  Define
\[
 G^{(1)}(S_1,S_2)\coloneqq G_1(S_1,S_2).
\]
If $m\geq\lceil C_2/4\rceil$, also define
\[
 G^{(2)}(S_1,S_2)
 \coloneqq G_{\lceil C_2/4\rceil}(S_1,S_2).
\]
\end{definition}

We record the degree bounds that will be used later.  Every vertex of
$S_1$ has degree $q$ in $G_q(S_1,S_2)$.  Moreover, as $(i,j)$ ranges
over the pairs satisfying $0\leq i<s$ and $0\leq j<q$, the integers
$iq+j$ range from $0$ to $sq-1$.  Consequently, every vertex of
$S_2$ has degree at most
\[
 \left\lceil\frac{sq}{m}\right\rceil
 \leq\left\lceil q+\frac qm\right\rceil
 \leq q+1.
\]
It follows that
\begin{equation}\label{degree S2 auxiliary}
 \max_{v\in S_2}d_{G^{(1)}(S_1,S_2)}(v)\leq2,
 \qquad
 \max_{v\in S_2}d_{G^{(2)}(S_1,S_2)}(v)
 \leq\left\lceil\frac{C_2}{4}\right\rceil+1\leq C_2,
\end{equation}
where the second bound applies whenever $G^{(2)}(S_1,S_2)$ is
defined.  Every vertex $v\in S_1$ satisfies
\begin{equation}\label{degree S1 auxiliary}
 d_{G^{(1)}(S_1,S_2)}(v)=1,
 \qquad
 d_{G^{(2)}(S_1,S_2)}(v)
 =\left\lceil\frac{C_2}{4}\right\rceil,
\end{equation}
where the second equality again applies whenever $G^{(2)}(S_1,S_2)$
is defined.

We next apply this construction to the sets provided by
Lemma~\ref{lem:Large Shadow odd}.  Let $r=3$ and $k$ be odd, let
$H,H',T$, and $F$ be as in that lemma, and fix $xy\in F$.
Fix a linear ordering of $V(H)$, and use the induced orderings of the
sets $S_1$ and $S_2$ whenever Definition~\ref{defi: Auxiliary graph large shadow}
is applied below.  Thus, once the pair $u\alpha$ is fixed, the
auxiliary graphs associated with that pair are fixed as well.
Condition~\ref{itm: lemma Large shadow odd iv} provides a vertex
$\alpha\in N_{H'}(xy)\cap T$ for which the three inequalities in that
condition hold.  For each $u\in\{x,y\}$, define
\[
 S_1(u,\alpha)\coloneqq N_H(u\alpha)\cap V(F),
 \qquad
 S_2(u,\alpha)\coloneqq N_H(u\alpha)\cap T.
\]
The sets $S_1(u,\alpha)$ and $S_2(u,\alpha)$ are disjoint by
condition~\ref{itm: lemma Large shadow odd i}.  Hence
\[
 d_H(u\alpha)
 \geq |S_1(u,\alpha)|+|S_2(u,\alpha)|.
\]
Condition~\ref{itm: lemma Large shadow odd iv} also gives
\[
 |S_2(u,\alpha)|\geq\frac{d_H(u\alpha)-1}{2}.
\]
Combining these inequalities yields
\begin{equation}\label{S1S2appliedrelation}
 |S_1(u,\alpha)|\leq |S_2(u,\alpha)|+1.
\end{equation}

Both sets are non-empty.  Indeed, if $v$ is the other vertex of the
pair $xy$, then $v\in V(F)$ and $uv\alpha=xy\alpha\in H'$, so
$v\in S_1(u,\alpha)$.  Moreover, $u\alpha\in\partial H'$, and the
$\lfloor(k+1)/2\rfloor$-fullness of $H'$ gives
\[
 d_H(u\alpha)\geq d_{H'}(u\alpha)
 \geq\left\lfloor\frac{k+1}{2}\right\rfloor\geq2.
\]
Thus condition~\ref{itm: lemma Large shadow odd iv} implies
$|S_2(u,\alpha)|\geq1/2$.  Since $|S_2(u,\alpha)|$ is an integer,
$S_2(u,\alpha)$ is non-empty.

Definition~\ref{defi: Auxiliary graph large shadow} therefore applies
to $S_1(u,\alpha)$ and $S_2(u,\alpha)$.  For each $u\in\{x,y\}$,
define
\begin{equation}\label{defi auxiliary G1}
 G_{\{u\alpha\}}^{(1)}
 \coloneqq
 G^{(1)}\bigl(S_1(u,\alpha),S_2(u,\alpha)\bigr).
\end{equation}
If $d_H(u\alpha)\geq C_2$, then
\[
 |S_2(u,\alpha)|
 \geq\frac{C_2-1}{2}
 \geq\left\lceil\frac{C_2}{4}\right\rceil,
\]
where the last inequality uses that $C_2$ is an integer and $C_2>3$.
In this case, also define
\begin{equation}\label{defi auxiliary G2}
 G_{\{u\alpha\}}^{(2)}
 \coloneqq
 G^{(2)}\bigl(S_1(u,\alpha),S_2(u,\alpha)\bigr).
\end{equation}

We now prove the extension statement needed in the remaining range
$r=3$ and odd $k\geq5$.

\begin{lemma}\label{lem:a copy of Ck from Ktt}
Let $k\geq5$ be odd.  Let $\delta>0$ and $C_0>0$, and let $C_2$ be an
integer satisfying $C_2>\max\{6C_0/\delta,6k\}$.
Let $H$ be a $3$-graph, let
$H'\subseteq H$ be a
$\frac{k+1}{2}$-full $3$-graph, and let
$T\subseteq V(H)$ and $F\subseteq\partial H'$ be as guaranteed by
Lemma~\ref{lem:Large Shadow odd} when applied to the pair $(H,H')$.
For every $f\in F$, let
$\mathcal{L}(f)\subseteq N_{H'}(f)\cap T$
be a set satisfying
$|\mathcal{L}(f)|=(k+1)/2$.
Assume additionally that, for every
$xy\in F$, the set $\mathcal{L}(xy)$ contains a vertex for which the
three inequalities in~\ref{itm: lemma Large shadow odd iv} hold.
Put $t\coloneqq(3\cdot2^k)^{12}$, and let $G$ be a copy of $K_{t,t}$.
If $F$ contains $G$, then one of the following holds:
\vspace{0.08in}
\begin{enumerate}
[label=\textnormal{(J\arabic*)},itemsep=4pt]
    \item\label{itm: cycle from shadow 1} There exists a copy of $C_{k}^{(3)}$ in $H$ such that every hyperedge of the cycle is of the form $f\cup\{z\}$, where $f\in G$ and $z\in\mathcal{L}(f)$.
    \item\label{itm: cycle from shadow 2} There are at least $\lceil C_2/12\rceil$ distinct copies of $C_k^{(3)}$ in $H$, indexed by $m\in [\lceil C_2/12\rceil]$.  The $m$-th copy has hyperedges $e_1,\ldots,e_k$, where $e_i\coloneqq w_iw_{i+1}z_i$ for every $i\in [k-2]$, $e_{k-1}\coloneqq w_1z_{k-1}z_k$, and $e_k\coloneqq w_{k-1}z_{k-1}z_{k+1}^{(m)}$; here $w_iw_{i+1}\in F$ and $z_i\in\mathcal{L}(w_iw_{i+1})$ for every $i\in [k-2]$, while $w_1w_{k-1}\in F$ and $z_{k-1}\in\mathcal{L}(w_1w_{k-1})$. Moreover, $z_{k}\in N_{G_{\{w_{1} z_{k-1}\}}^{(1)}}(w_{k-1})\subseteq T$ and $z_{k+1}^{(m)}\in N_{G_{\{w_{k-1}z_{k-1}\}}^{(2)}}(w_1)\subseteq T$ for every $m\in [\lceil C_2/12 \rceil ]$, where $G_{\{w_1z_{k-1}\}}^{(1)}$ and $G_{\{w_{k-1}z_{k-1}\}}^{(2)}$ are defined as in~\eqref{defi auxiliary G1} and~\eqref{defi auxiliary G2}, respectively, by taking $\alpha=z_{k-1}$ and $u=w_1$ for the first graph and $u=w_{k-1}$ for the second. (See Figure~\ref{fig:Odd Cycle} for an example with $k=7$.)
\end{enumerate}

\end{lemma}

\begin{proof}
Let $X_1$ and $X_2$ be the two parts of $G$, each of size $t$.  We
split the proof into two cases.
    
\vspace{0.08in}
\noindent \textbf{Case }\customlabel{lem: cycle from shadow Case 1}{\textbf{1}}\textbf{:} \textit{There exists a vertex $\alpha\in T$ and $t_1\coloneqq 2^{2\lfloor \frac{k-1}{2}\rfloor+1}+1$ pairwise disjoint edges $f_1,\ldots, f_{t_1}\in G$ such that $\alpha\in \mathcal{L}(f_i)$ for every $i\in [t_1]$.}
\vspace{0.08in}

Write $f_i=u_iv_i$, where the vertices $u_1,\ldots,u_{t_1}$
alternate between the two parts of $G$.  The pairs
\[
 u_1u_2,u_2u_3,\ldots,u_{t_1-1}u_{t_1}
\]
form a path $P^*$ of length
$2^{2\lfloor(k-1)/2\rfloor+1}$ in $G$.  Every edge $f$ of $P^*$ has
at least $\lfloor(k-1)/2\rfloor$ available extensions in
$\mathcal L(f)\setminus\{\alpha\}$.  Let $P'$ be the initial subpath
of $P^*$ with $2^{2\lfloor(k-3)/2\rfloor+1}$ edges.
Apply Lemma~\ref{lem: A P_k^r from a path in the shadow}, with $k-2$
in place of $k$, to the $3$-graph
\[
 \{f\cup\{z\}:f\in P',\ z\in\mathcal L(f)\setminus\{\alpha\}\}.
\]
Lemma~\ref{lem: A P_k^r from a path in the shadow} gives a copy $P$ of
$P_{k-2}^{(3)}$ whose hyperedges use only extensions different from
$\alpha$.

Let $u_i$ and $u_j$ be degree-one vertices of the first and last
hyperedges of $P$, respectively. Every vertex of $P$ belongs to
$\{u_1,\ldots,u_{t_1}\}\cup T$, whereas $v_i$ and $v_j$ belong to
$V(G)\setminus\{u_1,\ldots,u_{t_1}\}$. In particular,
$v_i,v_j\notin V(P)$. Hence the two hyperedges $u_iv_i\alpha$ and
$u_jv_j\alpha$, together with $P$, form a copy of $C_k^{(3)}$ in
$H$.  This copy satisfies condition~\ref{itm: cycle from shadow 1}.

\vspace{0.08in}
\noindent \textbf{Case }\customlabel{lem: cycle from shadow Case 2}{\textbf{2}}\textbf{:} \textit{For every vertex $\alpha\in T$, there are at most $t_1-1=2^{2\left\lfloor\frac{k-1}{2}\right\rfloor+1}$ pairwise disjoint edges $f\in G$ such that $\alpha\in\mathcal{L}(f)$.}
\vspace{0.08in}

Let $s\coloneqq(3\cdot2^k)^6$. We choose $s$ vertex-disjoint
subgraphs $F_1,\ldots,F_s$ of $G$, each of which is a copy of
$K_{t_1+1,t_1+1}$. This is possible since
\[
s(t_1+1)\leq(3\cdot2^k)^6\cdot3\cdot2^k
\leq(3\cdot2^k)^{12}=t.
\]
For every subgraph $Q\subseteq F$, let
\[
\mathcal{L}(Q)\coloneqq\bigcup_{f\in Q}\mathcal{L}(f).
\]
It follows that
\begin{equation}\label{eq: upper size LFi}
|\mathcal{L}(F_i)|\leq\frac{k+1}{2}\,|E(F_i)|
=\frac{k+1}{2}(t_1+1)^2<(3\cdot2^k)^3.
\end{equation}

In particular, $|\mathcal{L}(F_1)|<(3\cdot2^k)^3$.  By the
assumption of Case~\ref{lem: cycle from shadow Case 2}, each vertex
$\alpha\in\mathcal{L}(F_1)\subseteq T$ belongs to at most
$2^{2\lfloor(k-1)/2\rfloor+1}$ of the lists
$\mathcal{L}(F_i)$. It follows that at most
\[
 (3\cdot2^k)^3\cdot2^{2\lfloor(k-1)/2\rfloor+1}
\]
of the lists $\mathcal{L}(F_i)$ meet $\mathcal{L}(F_1)$.

Since
\[
 s=(3\cdot2^k)^6
 >(3\cdot2^k)^3\cdot2^{2\lfloor(k-1)/2\rfloor+1},
\]
there is an integer $i\in[s]$ such that
$\mathcal{L}(F_1)\cap\mathcal{L}(F_i)=\emptyset$.

Without loss of generality, we may assume
$\mathcal{L}(F_1)\cap \mathcal{L}(F_2)=\emptyset$.

For $i\in\{1,2\}$, we again label the vertices of each part of $F_i=\{u_{i,1},\ldots u_{i,t_1+1}\}\cup\{v_{i,1},\ldots,v_{i,t_1+1}\}$  so that the vertices $u_{i,j}$ and $v_{i',j'}$ are adjacent in $G$ (which is a copy of $K_{t,t}$) for every $i,i'\in \{1,2\}$ and $j,j'\in [t_1+1]$. Since $k$ is odd, in this subcase we have $|\mathcal{L}(u_{i,j}v_{i',j'})|=\lfloor \frac{k+1}{2}\rfloor=\frac{k+1}{2}$ for every $i,i'\in \{1,2\}$ and $j,j'\in [t_1+1]$.

Fix a vertex
$\alpha\in\mathcal L(u_{1,1}v_{2,1})$ for which the three
inequalities in~\ref{itm: lemma Large shadow odd iv} hold.  Since
$\mathcal L(F_1)\cap\mathcal L(F_2)=\emptyset$, after interchanging
$F_1$ and $F_2$ if necessary, we may assume that
$\alpha\notin\mathcal L(F_2)$.

Put $a_0\coloneqq u_{1,1}$ and $b_0\coloneqq v_{2,1}$.  Relabel the
two vertices $a_0,b_0$ as $x,y$ so that $d_H(x\alpha)\geq C_2$.
The graph $G_{\{y\alpha\}}^{(1)}$ is defined by
\eqref{defi auxiliary G1}, and $x$ belongs to its vertex class
$N_H(y\alpha)\cap V(F)$.  Let $\zeta$ be the unique neighbor of
$x$ in this auxiliary graph.  Thus
\begin{equation}\label{eq:fixed-replacement-edge}
 y\alpha\zeta\in H,
 \qquad
 \zeta\in N_{G_{\{y\alpha\}}^{(1)}}(x)\subseteq T.
\end{equation}

Set $p\coloneqq(k-3)/2$.  For $j\in[t_1+1]$, define
\[
 (a_j,b_j)\coloneqq
 \begin{cases}
  (u_{1,j},v_{2,j}),&\text{if $p$ is even},\\
  (v_{1,j},u_{2,j}),&\text{if $p$ is odd}.
 \end{cases}
\]
The edges $a_jb_j$ are pairwise disjoint.  By the assumption of
Case~\ref{lem: cycle from shadow Case 2}, at most $t_1-1$ of their
lists contain $\alpha$.  We may therefore choose
$j\in[t_1+1]\setminus\{1\}$ such that
$\alpha\notin\mathcal L(a_jb_j)$.

The complete bipartite graph $F_1$ contains a path $P_1$ of length
$p$ from $a_0$ to $a_j$, and $F_2$ contains a path $P_2$ of length
$p$ from $b_0$ to $b_j$.  Indeed, the prescribed endpoints lie in the
same part when $p$ is even and in different parts when $p$ is odd;
the required paths can be chosen with distinct internal vertices
because each part has $t_1+1>p$ vertices.  The two paths are
vertex-disjoint because $F_1$ and $F_2$ are vertex-disjoint.

Form a bipartite graph whose left vertex class is $E(P_1)$ and whose
right vertex class is $T\setminus\{\alpha,\zeta\}$, joining
$f\in E(P_1)$ to $z$ precisely when $z\in\mathcal L(f)$.  Every left
vertex has at least
\[
 |\mathcal L(f)\setminus\{\alpha,\zeta\}|
 \geq\frac{k+1}{2}-2=p
\]
neighbors, and the left class has $p$ vertices.  Hall's theorem gives
a matching covering $E(P_1)$.  Denote the distinct matched vertices by
$\beta_1,\ldots,\beta_p$.  Extending each edge of $P_1$ by its matched
vertex gives a copy $P_1^*$ of $P_p^{(3)}$ in $H'$.

Since $\alpha\notin\mathcal L(a_jb_j)$,
\[
 \left|\mathcal L(a_jb_j)\setminus
 \{\alpha,\zeta,\beta_1,\ldots,\beta_p\}\right|
 \geq\frac{k+1}{2}-(p+1)=1.
\]
Choose $\alpha'$ from this set.  Next form the analogous bipartite
graph for $P_2$, with right vertex class
\[
 T\setminus\{\alpha,\alpha',\zeta,
                 \beta_1,\ldots,\beta_p\}.
\]
The vertex $\alpha$ belongs to no list on $F_2$, and none of the
$\beta_i$ belongs to a list on $F_2$ because
$\mathcal L(F_1)\cap\mathcal L(F_2)=\emptyset$.  Thus every edge of
$P_2$ has at least $(k+1)/2-2=p$ available neighbors.  Hall's theorem
again gives a matching covering $E(P_2)$; denote its matched vertices
by $\gamma_1,\ldots,\gamma_p$.  Extending the edges of $P_2$ by these
vertices gives a copy $P_2^*$ of $P_p^{(3)}$ in $H'$.

The paths $P_1^*$ and $P_2^*$ together with the two hyperedges
$a_0b_0\alpha$ and $a_jb_j\alpha'$ form a copy $C^*$ of
$C_{k-1}^{(3)}$ in $H'$.  The hyperedge $a_0b_0\alpha=xy\alpha$ is
the only hyperedge of $C^*$ that we now replace.  By
\eqref{degree S1 auxiliary}, the vertex $y$ has
$\lceil C_2/4\rceil$ neighbors in
$G_{\{x\alpha\}}^{(2)}$.  At most $k$ of these neighbors belong to
\[
 \{\alpha,\alpha',\zeta,
   \beta_1,\ldots,\beta_p,\gamma_1,\ldots,\gamma_p\}.
\]
Since $C_2\geq6k$, there are at least
\[
 \left\lceil\frac{C_2}{4}\right\rceil-k
 \geq\left\lceil\frac{C_2}{12}\right\rceil
\]
choices for a vertex
\[
 \zeta'\in N_{G_{\{x\alpha\}}^{(2)}}(y)
\]
outside this set.  For each such $\zeta'$, replace
$xy\alpha$ in $C^*$ by the two hyperedges
$y\alpha\zeta$ and $x\alpha\zeta'$.  The resulting hyperedges form a
copy of $C_k^{(3)}$, and different choices of $\zeta'$ give distinct
copies.

To obtain the notation in condition~\ref{itm: cycle from shadow 2},
list the intersection vertices along the path in
$C^*\setminus\{xy\alpha\}$ from $y$ to $x$ as
$w_1,\ldots,w_{k-1}$, so that $w_1=y$ and $w_{k-1}=x$.  Let
$z_1,\ldots,z_{k-2}$ be the extension vertices on the corresponding
path hyperedges, and set
\[
 z_{k-1}\coloneqq\alpha,
 \qquad z_k\coloneqq\zeta,
 \qquad z_{k+1}^{(m)}\coloneqq\zeta'.
\]
Then~\eqref{eq:fixed-replacement-edge} and the choice of $\zeta'$ give
exactly the two auxiliary-graph conditions in
\ref{itm: cycle from shadow 2}.  This proves that alternative
\ref{itm: cycle from shadow 2} holds. This concludes Case~\ref{lem: cycle from shadow Case 2} and the proof
of Lemma~\ref{lem:a copy of Ck from Ktt}.

\begin{figure}[ht]
\centering
\begin{tikzpicture}[scale=0.8]
\node[shape=circle,draw=black, minimum size=1cm, minimum size=1cm] (W3) at (0,0) {$w_3$};
\node[shape=circle,draw=black, minimum size=1cm] (W2) at (2,0) {$w_2$};
\node[shape=circle,draw=black, minimum size=1cm] (W1) at (4,0) {$w_1$};

\node[shape=circle,draw=black, minimum size=1cm] (W4) at (0,-2) {$w_4$};
\node[shape=circle,draw=black, minimum size=1cm] (W5) at (2,-2) {$w_{5}$};
\node[shape=circle,draw=black, minimum size=1cm] (W6) at (4,-2) {$w_{6}$};

\node[shape=rectangle,draw=black, minimum size=1cm] (U2) at (1,2) {$z_2$};
\node[shape=rectangle,draw=black, minimum size=1cm] (U1) at (3,2) {$z_1$};
\node[shape=rectangle,draw=black, minimum size=1cm] (U3) at (-2,-1) {$z_3$};
\node[shape=rectangle,draw=black, minimum size=1cm] (U4) at (1,-4) {$z_{4}$};
\node[shape=rectangle,draw=black, minimum size=1cm] (U5) at (3,-4) {$z_{5}$};
\node[shape=rectangle,draw=black, minimum size=1cm] (U6) at (6,-1) {$z_6$};
\node[shape=rectangle,draw=black, minimum size=1cm] (U7) at (6,1.3) {$z_7$};
\node[shape=rectangle,draw=black, minimum size=1cm] (U8) at (6,-3.3) {$z_8^{(m)}$};

\graph{(W3)--(W2), (W2)--(W1), (W6)--(W5), (W5)--(W4), (W3)--(W4), (W3)--(U2),(W2)--(U1), (W2)--(U2), (W1)--(U1), (W3)--(U3), (W4)--(U3), (W4)--(U4),(W5)--(U4), (W5)--(U5), (W6)--(U5), (U6)--(U7), (U6)--(U8), (W1)--(U7), (W6)--(U8), (W1)--(U6), (W6)--(U6)};

\draw [dashed] (W1)--(W6);

\end{tikzpicture}
\caption{A copy of $C_{7}^{(3)}$ obtained in Case~\ref{lem: cycle from shadow Case 2}. The cycle contains $8$ vertices from $T$ (marked with a square) and $6$ vertices from $V(F)$ (marked with a circle). The hyperedge $w_1 w_6 z_6$ is replaced by $w_1 z_6 z_7$ and $w_6 z_6 z_{8}^{(m)}$ for $m \in [\lceil C_2/12\rceil ]$.}
\label{fig:Odd Cycle}
\end{figure}
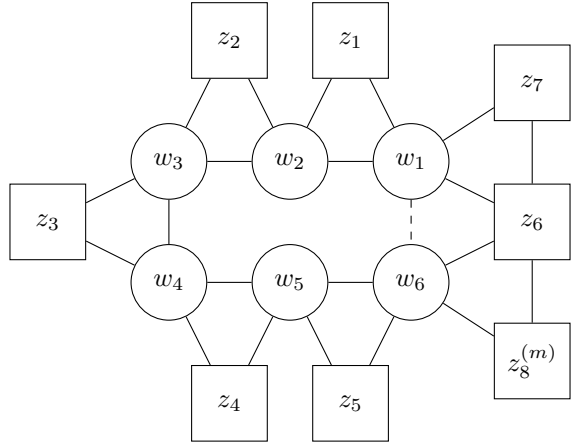

\end{proof}

\section{Proofs of the Supersaturation Lemmas}\label{sec:Supersaturation}

Let $C_0 = C_0(r, k, \varepsilon)$ be the constant given by Proposition~\ref{prop:Dense regime}. Throughout this section we assume that $H$ is an $r$-graph satisfying \[
\left(\left\lfloor\frac{k-1}{2}\right\rfloor+\frac{\varepsilon}{2}\right)\binom{n}{r-1} \leq |H| \leq C_0 \binom{n}{r-1}.\]

Recall that we classify hypergraphs according to the cases \ref{itm: Odd Case linear}, \ref{itm: Odd Case Codegree} and \ref{itm: Odd Case Shadow}  given by Proposition~\ref{prop:kOdd}, and the cases~\ref{itm: Even Case Codegrees} and \ref{itm: Even Case Shadow} given by Proposition~\ref{prop:kEven}. In this section we prove balanced supersaturation lemmas for each of these cases.

In Subsection~\ref{subsec:linear subgraph} we treat condition
\ref{itm: Odd Case linear}: there is a $(2,C_2)$-sparse $3$-graph
$H_0\subseteq H$ with
$|H_0|\geq(\varepsilon/4)\binom n2$. The argument uses
the graph regularity lemma and the graph counting lemma stated in
Lemmas~\ref{lem:graph-regularity} and~\ref{lem:graph-counting} to
obtain the initial collection of copies of $C_k^{(3)}$.

In Subsection~\ref{subsec:large codegrees}, we obtain a balanced
supersaturation result for $r$-graphs satisfying
condition~\ref{itm: Odd Case Codegree} or
condition~\ref{itm: Even Case Codegrees}. In the first case there is a
$C_1$-full $3$-graph $H_2\subseteq H$ with
$|H_2|\geq(\varepsilon/8)\binom n2$; in the second there is a
$C_1$-full $r$-graph $H_2\subseteq H$ with
$|H_2|\geq(\varepsilon/4)\binom n{r-1}$. The construction applies to every
$r,k\geq3$.  It uses bounded-degree switches over $(r-1)$-sets to
construct the cycle and to bound the number of constructed cycles
containing any prescribed tuple of hyperedges.

In Subsection~\ref{subsec:large shadow even} we obtain a balanced
supersaturation result for hypergraphs satisfying
condition~\ref{itm: Even Case Shadow}.  We combine
Lemma~\ref{lem:Large Shadow} with
Proposition~\ref{prop:many-complete-partite} to find many complete
$(r-1)$-partite $(r-1)$-graphs in the shadow, and then apply
Lemma~\ref{lem:cycle-from-shadow-standard} to obtain cycles.  Adding these
cycles one at a time while avoiding saturated tuples produces the
required balanced collection for every $r\geq4$, $k\geq3$, and also
for $r=3$ and even $k\geq4$.

Finally, in Subsection~\ref{subsec:large shadow}, we prove a
supersaturation result for hypergraphs satisfying condition \ref{itm: Odd Case Shadow}, which consists of an ambient $3$-graph $H$ containing a
$\left\lfloor\frac{k+1}{2}\right\rfloor$-full subhypergraph
$H_1\subseteq H$ satisfying
$|H_1|\geq(\varepsilon/4)\binom n2$ and
$|\partial H_1|\geq(\varepsilon/(8C_1))\binom n2$, where every hyperedge of $H_1$
contains a pair having codegree at least $C_2$ in $H$. In this
setting, the initial collection of copies of $C_k^{(3)}$ may be
smaller than the collection obtained in
Subsection~\ref{subsec:large shadow even}. We use the additional
conclusions of Lemma~\ref{lem:Large Shadow odd} and
Lemma~\ref{lem:a copy of Ck from Ktt} to construct the required
balanced collection of copies of $C_k^{(3)}$.

\subsection{Supersaturation from a large linear subhypergraph}\label{subsec:linear subgraph}

The supersaturation result in this subsection applies to $3$-graphs obtained from condition \ref{itm: Odd Case linear} of Proposition~\ref{prop:kOdd}. Recall that a $3$-graph $H$ is \emph{$(t,C)$-sparse} if every set of $t$ vertices is contained in at most $C$ hyperedges. Also recall that 
for a collection $\KK$ of cycles $C_{k}^{(r)}$ in an $r$-graph $H$, $\Delta_{\ell}(\KK)$ denotes the maximum number of cycles $C_{k}^{(r)}$ containing a given set of $\ell$ hyperedges of $H$.

\begin{lemma}\label{lem: Supersat Large Linear subgraph}
Let $0<\varepsilon<1$, and let $k\geq5$ be an odd integer. Let $C_2$ be chosen according to~\eqref{defi:constant hierarchy}. Then there is an integer $n_0=n_0(\varepsilon,k,C_2)$ such that, for every $n>n_0$, the following holds. If $H_0$ is an $n$-vertex $(2,C_2)$-sparse $3$-graph with $|H_0|\geq \frac{\varepsilon}{4}\binom{n}{2}$, then there exists a collection $\KK$ of copies of $C_k^{(3)}$ contained in $H_0$ such that

        \vspace{4pt}
\begin{enumerate}[label=\textnormal{(L\arabic*)},itemsep=4pt]
    \item\label{itm: supersat linear 1} $|\KK|\geq \frac{D^{k-1}\cdot n^{2}}{K}$, and
    \item\label{itm: supersat linear 2} $\Delta_{\ell}(\KK)\leq D^{k-\ell}$ for $\ell\in[k]$, 
\end{enumerate}
where $D= n^{(k-2)/(k-1)}$ and $K=\log n$.
\end{lemma}
\begin{proof}
Let $G_{H_0}$ be the auxiliary graph with vertex set $V(G_{H_0})\coloneqq E(H_0)$, in which two vertices $e_1,e_2\in E(H_0)$ are adjacent if and only if $|e_1\cap e_2|=2$. Since $H_0$ is $(2,C_2)$-sparse, we have $\Delta(G_{H_0})\leq 3C_2-3$. A greedy vertex coloring of $G_{H_0}$ uses at most $\Delta(G_{H_0})+1\leq3C_2-2$ colors, and hence $\chi(G_{H_0})<3C_2$.

Let $H_0^*\subseteq H_0$ be the $3$-graph whose hyperedge set is a largest color class in a proper vertex coloring of $G_{H_0}$ with $\chi(G_{H_0})$ colors. Since $\chi(G_{H_0})<3C_2$, it follows that
\begin{equation}
\label{eq:lowerboundH0*}
|H_0^*|\geq \frac{|H_0|}{3C_2}
\ge
\frac{\varepsilon}{12C_2}\binom{n}{2}.
\end{equation}
Moreover, since $E(H_0^*)$ is a color class in a proper coloring of $G_{H_0}$, no two hyperedges of $H_0^*$ intersect in two vertices. Thus, $H_0^*$ is linear.

Since $H_0^*$ is linear, we have $|\partial H_0^*|=3|H_0^*|\geq \frac{\varepsilon}{4C_2}\binom{n}{2}$. Set
\[
\varepsilon'\coloneqq \left(\frac{\varepsilon}{500\cdot k\cdot C_2}\right)^{k}.
\]
Applying Lemma~\ref{lem:graph-regularity} to the graph
$\partial H_0^*$ with $\eta=\varepsilon'$ and
$m_0=\lceil1/\varepsilon'\rceil$, we obtain integers
$M=M(\varepsilon')$ and $m$ satisfying
\[
\frac{1}{\varepsilon'}\leq m\leq M,
\]
and a partition
\[
V(\partial H_0^*)=V_0\cup V_1\cup\cdots\cup V_m
\]
such that $|V_0|\leq \varepsilon' n$, the sets $V_1,\ldots,V_m$ have equal size, and, for all but at most $\varepsilon' m^2$ pairs $1\leq i<j\leq m$, the bipartite subgraph of $\partial H_0^*$ induced by $V_i$ and $V_j$ is $\varepsilon'$-regular.

We delete from $\partial H_0^*$ all edges lying between $\varepsilon'$-irregular pairs $(V_i,V_j)$, all edges having both endpoints in the same cluster $V_i$, all edges incident to a vertex in $V_0$, and all edges lying between $\varepsilon'$-regular pairs $(V_i,V_j)$ of density at most 
\begin{equation}\label{eq:definition delta}
\delta\coloneqq(8k\varepsilon')^{1/k}=\left(8k\left(\frac{\varepsilon}{500\cdot k \cdot C_2}\right)^{k}\right)^{1/k} <\frac{2\cdot k\cdot \varepsilon}{500\cdot k\cdot C_2}=\frac{\varepsilon}{250 C_2}.
\end{equation}
The total number of removed edges is at most
\begin{equation}\label{eq: removed edges}
\varepsilon' m^2\cdot  |V_1|^2+m\cdot |V_1|^2+n|V_0|+ \delta \cdot \binom{n}{2} \leq  \varepsilon' n^2+ \frac{1}{m}n^2+\varepsilon' n^2+\frac{\delta}{2} n^2\leq 3 \varepsilon' n^2 + \frac{\delta}{2} n^2 < \frac{\varepsilon}{100 C_2} n^2,
\end{equation}
where we used that the sets $V_1, \ldots, V_m$ have the same size, $|V_1| \le n/m$, $|V_0|\leq \varepsilon' n$, $m\geq \frac{1}{\varepsilon'}$,  \eqref{eq:definition delta} and $\varepsilon'<\frac{\varepsilon}{500 \cdot C_2}$.

Note that $\partial H_0^*$ contains at least $|H_0^*|$ triangles, namely, the triangles corresponding to the hyperedges of $H_0^*$. Since $H_0^*$ is linear, each edge of $\partial H_0^*$ is contained in exactly one such triangle. Therefore, deleting an edge from $\partial H_0^*$ destroys at most one triangle corresponding to a hyperedge of $H_0^*$. Since, by \eqref{eq: removed edges}, the total number of deleted edges is at most $\frac{\varepsilon}{100 C_2} n^2$, the resulting graph still contains at least
\[
|H_0^*|-\frac{\varepsilon}{100 \cdot C_2}n^2\overset{\eqref{eq:lowerboundH0*}}{\geq} \frac{\varepsilon}{12C_2}\binom{n}{2}-\frac{\varepsilon}{100C_2}n^2> \frac{\varepsilon}{24C_2}\binom{n}{2}
\]
triangles. Define a graph $R$ with vertex set $[m]$ by joining $i$ and
$j$ if $(V_i,V_j)$ is $\varepsilon'$-regular and has density greater
than $\delta$ in $\partial H_0^*$. Every edge that remains after the
deletions lies between $V_i$ and $V_j$ for some $ij\in E(R)$.
Consequently, each triangle that remains gives a triangle in $R$.
Thus, $R$ contains a triangle $i_1i_2i_3$.

For $n$ sufficiently large with respect to $C_2$ and $\varepsilon$, each cluster $V_i$ satisfies
\[
|V_i|=\frac{n-|V_0|}{m}\geq \frac{(1-\varepsilon')n}{M}>\left(\frac{500\cdot k\cdot C_2}{\varepsilon}\right)^{k}=(\varepsilon')^{-1},
\]
where $M=M(\varepsilon')=M(\varepsilon,k,C_2)$ is supplied by
Lemma~\ref{lem:graph-regularity}.  Since $\varepsilon'<1/2$, the same
inequality also gives $|V_i|\geq n/(2M)$ for every $i\in[m]$.
Moreover, since
$\delta=(8k\varepsilon')^{1/k}$, we have
\begin{equation}\label{eq:counting lemma prerequisite}
\frac{\delta^k}{4k}=\frac{8k\varepsilon'}{4k}=2\varepsilon'>\varepsilon'.
\end{equation}
Fix a proper $3$-coloring
$c:[k]\to\{i_1,i_2,i_3\}$ of $C_k$, and for every $j\in[k]$ set
$W_j\coloneqq V_{c(j)}$. For each edge $j(j+1)$ of $C_k$, with indices taken
modulo $k$, the pair $(W_j,W_{j+1})$ is $\varepsilon'$-regular and
has density greater than $\delta$. Moreover,
$|W_j|=|V_1|>(\varepsilon')^{-1}$ for every $j\in[k]$, and
\eqref{eq:counting lemma prerequisite} gives
$\varepsilon'<\delta^k/(4k)$. Lemma~\ref{lem:graph-counting}, applied
with $J=C_k$, therefore gives at least
$2^{-k}\delta^k|V_1|^k$ labelled copies of $C_k$ in
$\partial H_0^*$. Every distinct copy of $C_k$ arises from at most
$2k$ such labelled copies. Hence there is a collection
$\mathcal{K}_0$ of distinct copies of $C_k$ contained in
$\partial H_0^*$ such that
\begin{equation}\label{eq:lower K0 linear subgraph}
|\KK_0|\geq \frac{2^{-k}\delta^k|V_1|^k}{2k}
\geq \frac{4\varepsilon'n^k}{4^kM^k}.
\end{equation}

Having obtained many copies of $C_k$ in $\partial H_0^*$, we now use
the fact that $H_0^*$ is a linear $3$-graph to show that many of these
copies can be extended to copies of $C_k^{(3)}$ in $H_0^{*}$.

Let $\mathcal{K}_1$ be the collection of multisets of $k$ hyperedges
of $H_0^*$ constructed as follows.  For each copy
$Q\in\mathcal{K}_0$ and each edge $e\in E(Q)$, the inclusion
$e\in\partial H_0^*$ gives a hyperedge of $H_0^*$ containing $e$, and
the linearity of $H_0^*$ makes this hyperedge unique; denote it by
$\widehat e$.  Associate with $Q$ the multiset
$\{\widehat e:e\in E(Q)\}$, in which a hyperedge is repeated if it is
the unique extension of more than one edge of $Q$, and let
$\mathcal{K}_1$ be the collection of all multisets obtained in this
way.  A member of $\mathcal{K}_1$ need not consist of distinct
hyperedges or form a copy of $C_k^{(3)}$ in $H_0^*$.  Moreover, each
member of $\mathcal{K}_1$ arises from at most $3^k$ members of
$\mathcal{K}_0$, because each occurrence of a hyperedge contains only
three possible graph edges. Therefore,
\begin{equation}\label{eq:Lower K1 linear subgraph}
|\mathcal{K}_1|\geq \frac{|\mathcal{K}_0|}{3^k}
\geq \frac{4\varepsilon'n^k}{12^kM^k}.
\end{equation}

Let $\KK_{2}$ denote the members of $\KK_{1}$ that consist of $k$
distinct hyperedges forming a copy of $C_{k}^{(3)}$.
\begin{claim} \label{claim2}
There is an integer $n_0=n_0(\varepsilon, C_2)$ such that for $n>n_0$ we have
\begin{equation}
\label{eq:claim2}
      |\KK_{2}|\geq \frac{2\varepsilon' n^k}{12^{k}\cdot M^k}.  
\end{equation}
\end{claim}
\begin{claimproof}
We distinguish two cases according to how a member of $\mathcal{K}_1$ can fail to belong to $\mathcal{K}_2$ during its construction from a copy $Q$ of $C_k$ in $\mathcal{K}_0$. For each of these two cases, we bound the number of such members of $\mathcal{K}_1$ separately.

\vspace{0.08in}
\noindent \textbf{Case} \customlabel{case 1 supersat linear}{\textbf{1}}\textbf{:} \textit{There exist distinct edges $e_1,e_2\in E(Q)$ such that the unique hyperedges $\widehat e_1,\widehat e_2\in H_0^*$ containing $e_1$ and $e_2$, respectively, satisfy $\widehat e_1\setminus e_1 = \widehat e_2\setminus e_2$.}

\vspace{0.08in}

In this case, our aim is to bound the number of members of $\mathcal{K}_1$ arising from copies $Q\in\mathcal{K}_0$ satisfying Case \ref{case 1 supersat linear}. For convenience, write $\widehat e_1\setminus e_1=\widehat e_2\setminus e_2=\{v\}$. 

There are $|H_0^*|$ choices for $\widehat e_1$ and at most $3$ choices for the vertex $v\in\widehat e_1$. Since $H_0^*$ is linear, for each fixed choice of $v$, there are at most $n$ choices for $\widehat e_2$ containing $v$. 

Consider the edges of $Q$ in the cyclic order induced by the cycle. Note that since $H_0^{*}$ is linear, $e_1$ and $e_2$ are not consecutive in this ordering. There are at most $k^2$ possible choices for the positions occupied by $e_1$ and $e_2$ in this ordering. Moreover, once $e_1$, $e_2$, and their positions in $Q$ have been fixed, there are at most $n^{k-4}$ ways to choose the remaining vertices of a copy $Q\in\mathcal{K}_0$ of $C_k$. Furthermore, each such copy $Q$ determines a unique member of $\mathcal{K}_1$, since $H_0^*$ is linear. Consequently, the total number of members of $\mathcal{K}_1$ arising from copies $Q\in\mathcal{K}_0$ satisfying Case \ref{case 1 supersat linear} is at most 
\[
3|H_0^*|\cdot n\cdot k^2\cdot n^{k-4}\leq k^2\cdot \binom{n}{2}\cdot n^{k-3}\leq \frac{k^2}{2}n^{k-1},
\]
where we used the fact that $H_0^*$ is linear, and hence $3|H_0^*|\leq \binom{n}{2}$.

\vspace{0.08in}
\noindent \textbf{Case} \customlabel{case 2 supersat linear}{\textbf{2}}\textbf{:} \textit{There exist distinct edges $e_1$, $e_2\in E(Q)$ such that for some vertex $w\in e_2$, we have $\widehat{e}_1=e_1\cup \{w\}\in H_0^*$.}
\vspace{0.08in}

In this case we bound the number of members of $\KK_1$ arising from copies of $\KK_0$ satisfying Case \ref{case 2 supersat linear}. There are $|H_0^{*}|$ choices for $\widehat{e}_1$ and $3$ choices for $w\in \widehat{e}_1$.

Consider the edges of $Q$ in their cyclic order.  There are at most
$k^2$ choices for the positions occupied by $e_1$ and $e_2$.  If
these positions are nonconsecutive, then $e_1\cup e_2$ has four
vertices.  There are at most $n$ choices for $e_2$ containing $w$ and,
after $e_1$ and $e_2$ have been fixed, at most $n^{k-4}$ choices for
the remaining vertices of $Q$.  If the positions are consecutive,
then $e_2$ consists of $w$ and one of the two vertices of $e_1$.
There are at most two choices for $e_2$ and at most $n^{k-3}$ choices
for the remaining vertices of $Q$.  Thus, for fixed
$\widehat e_1$, $w$, and the two positions, the two cases give at most
$3n^{k-3}$ choices in total.  Every resulting graph cycle determines
a unique member of $\mathcal K_1$.  Hence the number of members
arising in Case~\ref{case 2 supersat linear} is at most
\[
 9|H_0^*|k^2n^{k-3}
 \leq\frac{3k^2}{2}n^{k-1},
\]
where the inequality uses $3|H_0^*|\leq\binom n2$.

\vspace{0.08in}
From the bounds obtained in Case \ref{case 1 supersat linear} and  Case \ref{case 2 supersat linear},  it follows that 
\[
|\mathcal{K}_1\setminus \mathcal{K}_2|\leq 2k^2 \cdot n^{k-1},
\]
which together with \eqref{eq:Lower K1 linear subgraph}
implies that, for large enough $n$ with respect to $\varepsilon$ and $C_2$, we have 
\[
|\KK_2|\geq \frac{4\varepsilon' n^k}{12^{k}\cdot M^{k}}-2k^2 \cdot n^{k-1}> \frac{2\varepsilon' n^k}{12^{k}\cdot M^k}.
\]
This concludes the proof of Claim~\ref{claim2}.
\end{claimproof}

Lastly, we run a cleaning step on the collection $\KK_2$ to find a balanced subcollection $\KK \subseteq \KK_2$ of copies of $C_k^{(3)}$ in $H^*_0$.
We begin by setting $\KK\coloneqq\emptyset$ and note that the empty collection satisfies the condition \ref{itm: supersat linear 2} of Lemma~\ref{lem: Supersat Large Linear subgraph}. We fix an arbitrary ordering for the members of the collection $\KK_2$. We consider the cycles in $\KK_2$ one at a time, according to this ordering, and add a cycle to $\KK$ whenever doing so preserves condition~\ref{itm: supersat linear 2}. We continue this process until condition~\ref{itm: supersat linear 1} of Lemma~\ref{lem: Supersat Large Linear subgraph} is satisfied.

Let $\sigma$ be an unordered collection of $\ell$ distinct hyperedges of $H_0^*$, where $1\leq \ell\leq k-1$. We say that $\sigma$ is $D$-\emph{saturated} if $\sigma$ is contained in at least $\lfloor D^{k-\ell}\rfloor$ copies of $C_k^{(3)}$ from $\KK$. Since $D$ is fixed by the lemma, for convenience, we simply call such an $\ell$-tuple \emph{saturated}. 

Note that, since $\KK$ satisfies condition \ref{itm: supersat linear 2}, no $\ell$-tuple of hyperedges is contained in more than $D^{k-\ell}$ members of $\KK$, and thus, a saturated $\ell$-tuple is contained in exactly $\left\lfloor D^{k-\ell}\right\rfloor$ members of $\KK$. In particular, a single hyperedge is saturated if it belongs to exactly $\lfloor D^{k-1}\rfloor$ copies of $C_{k}^{(3)}$ from $\KK$ (as $\ell = 1$).

By double-counting the pairs $({e_1,\ldots,e_\ell},C^*)$ such that $C^*\in \KK$ and ${e_1,\ldots,e_\ell}\subseteq C^*$ is a saturated $\ell$-tuple, we deduce that the number of saturated $\ell$-tuples $\{e_1,\ldots, e_\ell\} \subset H_0^*$ is at most $\binom{k}{\ell}|\KK|/\left\lfloor D^{k-\ell}\right\rfloor$.
If condition \ref{itm: supersat linear 1} does not hold, then we have $|\KK|\leq \frac{D^{k-1}\cdot n^{2}}{K}$. Therefore, the number of saturated $\ell$-tuples $\{e_1,\ldots, e_\ell\}\subset H_0^{*}$ is at most
\begin{equation}
\label{eq:boundsatl-tuples}
  \binom{k}{\ell}\cdot \frac{|\KK|}{\left\lfloor D^{k-\ell}\right\rfloor}\leq \binom{k}{\ell} \frac{2\cdot D^{\ell-1}\cdot n^{2}}{K}.  
\end{equation}
Here we used that $D^{k-\ell}\geq2$ for all sufficiently large $n$,
and hence $\lfloor D^{k-\ell}\rfloor\geq D^{k-\ell}/2$.

We now bound the number of copies of $C_k^{(3)}$ in $\KK_2$
containing a fixed saturated $\ell$-tuple $\sigma$.  Write
$\sigma=\{e_1,\ldots,e_\ell\}\subseteq H_0^*$.  There are at most $k^\ell$
ways to choose the positions occupied by these hyperedges in the cyclic
order.  For every selected position, there are at most six ways to
identify the two degree-$2$ vertices and the private vertex of the
corresponding hyperedge.  The selected hyperedges contain at least
$\ell+1$ degree-$2$ vertices of the cycle.  After the positions and
identifications have been fixed, there are therefore at most
$n^{k-(\ell+1)}$ choices for the remaining degree-$2$ vertices.
Since $H_0^*$ is linear, every pair of consecutive degree-$2$ vertices
has at most one extension in $H_0^*$.  Consequently, the number of
copies in $\KK_2$ containing the fixed saturated tuple is at most
\[
 (6k)^\ell \cdot n^{k-(\ell+1)}.
\]

Let us now bound the number of copies of $C_k^{(3)}$ in $\KK_2$ that contain a saturated $\ell$-tuple for some $\ell \in [k-1]$. Recall that $D=n^{(k-2)/(k-1)}<n$ and $K=\log n$. By \eqref{eq:boundsatl-tuples}, there are at most $\binom{k}{\ell} \frac{2\cdot D^{\ell-1}\cdot n^{2}}{K}$ saturated $\ell$-tuples $\{e_1,\ldots,e_\ell\}\subseteq H_0^*$. Hence, for all sufficiently large $n$, the number of copies of $C_k^{(3)}$ in $\KK_2$ that contain at least one saturated $\ell$-tuple for some $\ell \in [k-1]$ is at most
\begin{equation}
  \label{eq:boundsaturatedltuple}
 \sum_{\ell=1}^{k-1} \left((6k)^\ell n^{k-(\ell+1)} \right)\cdot\binom{k}{\ell}\cdot \frac{2\cdot D^{\ell-1}\cdot n^2}{K}< \frac{2\cdot (6k)^{2k+1} \cdot n^{k}}{\log n}<\frac{\varepsilon' n^k}{2\cdot12^{k}\cdot M^k} \overset{\eqref{eq:claim2}}{\leq} \frac{|\KK_2|}{4}.
\end{equation}
 Note that if $\KK$ does not satisfy condition \ref{itm: supersat linear 1}, then we also have
\[
|\KK|< \frac{D^{k-1} n^{2}}{K}=\frac{n^k}{\log n}< \frac{\varepsilon' n^k}{2\cdot12^{k}\cdot M^k} \overset{\eqref{eq:claim2}}{\leq} \frac{|\mathcal{K}_2|}{4}.
\]
Thus, if condition~\ref{itm: supersat linear 1} fails, then $|\KK_2\setminus \KK|>3|\KK_2|/4$. Moreover, by \eqref{eq:boundsaturatedltuple}, at most $|\KK_2|/4$ members of $\KK_2$ contain a saturated $\ell$-tuple for some $\ell \in [k-1]$. Hence, at least $|\KK_2\setminus \KK|-|\KK_2|/4\geq |\KK_2|/2$ members of $\KK_2\setminus \KK$ contain no saturated $\ell$-tuple for any $\ell \in [k-1]$.  We may therefore add one such member of $\KK_2\setminus\KK$ to $\KK$ while preserving condition~\ref{itm: supersat linear 2}; here we use that the condition~\ref{itm: supersat linear 2} holds trivially when $\ell=k$, since any collection of $k$ distinct hyperedges is contained in at most one copy of $C_k^{(3)}$ in the resulting collection.

Repeating this process until condition~\ref{itm: supersat linear 1} holds, we obtain a collection $\KK$ satisfying both conditions~\ref{itm: supersat linear 1} and~\ref{itm: supersat linear 2}, as desired. This completes the proof of Lemma~\ref{lem: Supersat Large Linear subgraph}.
\end{proof}

\subsection{Supersaturation from a large subhypergraph with large codegrees}
\label{subsec:large codegrees}

The construction in this subsection builds on the skeleton construction
of Balogh, Narayanan and
Skokan~\cite{Balogh2019NumberCyclefree}.  For each $f\in\partial H_2$,
an auxiliary graph on $N_{H_2}(f)$ will specify which two extension
vertices of $f$ may replace one another.  We use these auxiliary
graphs first to choose the intersection vertices of the cycle and
then to replace the vertices that should not be shared by two cycle
hyperedges.

For an ordered copy of $C_k^{(r)}$ with hyperedges $e_1,\ldots,e_k$,
let $v_i$ be the unique vertex in $e_{i-1}\cap e_i$, with indices
modulo $k$, and choose a vertex
$u_i\in e_i\setminus\{v_i,v_{i+1}\}$.  The set
\[
 S\coloneqq\{v_1,\ldots,v_k,u_1,\ldots,u_k\}
\]
is the \emph{support}, and
$y_i\coloneqq\{v_i,v_{i+1},u_i\}$ is the $i$th \emph{boundary triple}.
The vertices in $e_i\setminus\{v_i,v_{i+1}\}$ belong to no other hyperedge of the cycle
and will be called the \emph{private vertices of $e_i$}.
Besides the boundary triples, for $2\leq i\leq k-1$ put
\[
 q_i\coloneqq\{v_1,v_i,v_{i+1}\}.
\]
The graph whose vertices are the $2k-2$ triples
$y_1,\ldots,y_k,q_2,\ldots,q_{k-1}$, with two triples adjacent when
they share two vertices, is a tree.  Indeed,
$q_2,q_3,\ldots,q_{k-1}$ form a path; $y_i$ is adjacent to $q_i$ for
$2\leq i\leq k-1$; and $y_1$ and $y_k$ are adjacent to $q_2$ and
$q_{k-1}$, respectively.  We call this graph the
\emph{triangulation tree}; Figure~\ref{fig:triangulation-tree} shows
the configuration and its triangulation tree when $k=6$.  Rooting the
triangulation tree at a boundary
triple means distinguishing that triple and, for every other triple,
calling its unique neighbor closer to the distinguished triple its
\emph{parent}.  A triple and its parent share two vertices.  Thus, if
the triples are processed in increasing order of their distance from
the root, each triple after the root is obtained from its parent by
replacing exactly one vertex.

\begin{lemma}\label{lem: Supersat Large codegrees}
Let $r,k\geq3$, and let $C_1$ be chosen according
to~\eqref{defi:constant hierarchy}.  If $H_2$ is an $n$-vertex
$C_1$-full $r$-graph, then there is a collection $\KK$ of copies of
$C_k^{(r)}$ in $H_2$ such that
\begin{enumerate}[label=\textnormal{(M\arabic*)},itemsep=4pt]
 \item\label{M1firstitem}
 $|\KK|\geq D^{k-1}|H_2|/K$;
 \item\label{M2seconditem}
 $\Delta_j(\KK)\leq D^{k-j}$ for every $j\in[k]$,
\end{enumerate}
where
\[
 D\coloneqq (C_1)^{r-1},\qquad
 K\coloneqq c_{r,k}C_1,
 \qquad
 c_{r,k}\coloneqq2^{(r-1)(k-1)-1}\,2k((r-2)!)^k.
\]
\end{lemma}

\begin{proof}
For each $f\in\partial H_2$, apply
Proposition~\ref{prop:Graph with fixed degrees} to the vertex set
$N_{H_2}(f)$.  We thereby fix a graph $\Gamma_f$ on $N_{H_2}(f)$ in
which every degree is $C_1-1$ or $C_1$.  If $x$ and $x'$ are adjacent
in $\Gamma_f$, then both $f\cup\{x\}$ and $f\cup\{x'\}$ are
hyperedges of $H_2$.  We call the replacement of $x$ by $x'$ a
\emph{switch over $f$}.  In particular, for fixed $f$ and $x$, there
are at least $C_1-1$ choices for $x'$, while for fixed $f$ and $x'$
there are at most $C_1$ choices for $x$.

For every hyperedge $e\in H_2$, fix an ordering of its vertices.  A
\emph{construction record} consists of an initial hyperedge $e_1$ and
the sequence of new vertices chosen at the switches below, listed in
the fixed order in which the switches are performed.

To construct such a record, choose a hyperedge $e_1\in H_2$ and use
its fixed ordering to write
\[
 e_1=\{v_1,v_2,u_1\}\cup R,
 \qquad
 R\coloneqq(z_1,\ldots,z_{r-3}),
\]
where a union with $R$ means a union with the underlying set
$\{z_1,\ldots,z_{r-3}\}$.  Thus $e_1=y_1\cup R$.  We next define
$v_3,\ldots,v_k,u_2,\ldots,u_k$, and hence all the remaining triples
$y_i$ and $q_i$, by following the triangulation tree (see Figure~\ref{fig:triangulation-tree}).

Root the triangulation tree at $y_1$ and process its triples in a fixed
order of increasing distance from $y_1$.  Suppose that the parent triple
$\{a,b,x\}$ has already been processed and its child is
$\{a,b,x'\}$.  At this stage $\{a,b,x\}\cup R$ is a hyperedge of $H_2$.
Put $f\coloneqq\{a,b\}\cup R$.  Choose $x'$ adjacent to $x$ in $\Gamma_f$
and distinct from every vertex chosen earlier.  Then
$\{a,b,x'\}\cup R=f\cup\{x'\}\in H_2$.  This procedure performs
$2k-3$ switches, one for each non-root triple of the triangulation tree.

After all the non-root triples have been processed, let
\[
 S\coloneqq\{v_1,\ldots,v_k,u_1,\ldots,u_k\}
\]
be the resulting support.  At this point, $y_i\cup R$ is a hyperedge
of $H_2$ for every $i\in[k]$, but all these hyperedges contain $R$.
We next replace the vertices of $R$, separately for every
$i\in\{2,\ldots,k\}$, so that no two resulting cycle hyperedges share
a vertex of $R$.

For each $i\in\{2,\ldots,k\}$, start with
$E_{i,0}\coloneqq y_i\cup R$.  For $h=1,\ldots,r-3$, put
$f_{i,h}\coloneqq E_{i,h-1}\setminus\{z_h\}$.

Before choosing the replacement for $z_h$, let $B_{i,h}$ be the set
of vertices that are forbidden at this switch: the support vertices,
the original vertices $z_1,\ldots,z_{r-3}$, and the replacement
vertices chosen at earlier switches.
\[
 B_{i,h}\coloneqq S\cup\{z_1,\ldots,z_{r-3}\}
 \cup\{w_{i',h'}:(i',h')\text{ precedes }(i,h)
                     \text{ in lexicographic order}\}.
\]
Choose
\[
 w_{i,h}\in N_{\Gamma_{f_{i,h}}}(z_h)\setminus B_{i,h}
\]
and set
\[
 E_{i,h}\coloneqq f_{i,h}\cup\{w_{i,h}\}.
\]
Finally set
\[
 e_i\coloneqq E_{i,r-3}\quad(2\leq i\leq k).
\]
The resulting hyperedges $e_1,\ldots,e_k$ form a copy of $C_k^{(r)}$:
with indices taken modulo $k$, the intersection $e_i\cap e_{i+1}$ is
exactly $\{v_{i+1}\}$, and the avoidance condition makes every vertex
outside $\{v_1,\ldots,v_k\}$ private to one hyperedge.  Hence
nonconsecutive hyperedges are disjoint.

At each switch fewer than $k(r-1)+2k$ vertices are forbidden.  By the
choice of $C_1$, there are at least $C_1/2$ choices.  The total number
of switches is
\begin{equation}\label{eq:number-switches}
 M\coloneqq2k-3+(r-3)(k-1)=(r-1)(k-1)-1.
\end{equation}
By definition, a construction record is determined by $e_1$ and the
new vertex chosen at each of these $M$ switches.  Hence the construction
produces at least
$|H_2|\cdot (C_1/2)^M$ records.

A fixed cycle has at most $2k$ choices for its first hyperedge and
orientation.  Once these are fixed, each hyperedge has at most $r-2$
choices for $u_i$, and the remaining $r-3$ private vertices may be
ordered in at most $(r-3)!$ ways.  Thus a fixed cycle arises from at
most
\[
 B_{r,k}\coloneqq2k((r-2)!)^k
\]
records.  Let $\KK$ be the collection of all cycles produced.  Since
$M=(r-1)(k-1)-1$, we obtain
\[
 |\KK|\geq\frac{|H_2|\cdot (C_1/2)^M}{B_{r,k}}
 =\frac{D^{k-1}|H_2|}{c_{r,k}C_1}
 =\frac{D^{k-1}|H_2|}{K},
\]
which proves~\ref{M1firstitem}.

We next prove~\ref{M2seconditem}.  Fix $1\leq j\leq k-1$ and a set
$\sigma$ of $j$ hyperedges of $H_2$.  We count construction records
whose resulting cycle contains every hyperedge in $\sigma$.  There
are at most $k^j$ ways to place these hyperedges in the cyclic order and at most
$(r!)^j$ ways to identify their intersection vertices, their vertices
$u_i$, and their ordered remaining private vertices.  After these
choices are fixed, the corresponding $j$ boundary triples contain at
least $2j+1$ support vertices: the $j$ vertices $u_i$ are distinct,
and a proper set of $j$ hyperedges of a cycle has at least $j+1$ distinct
endpoints among $v_1,\ldots,v_k$.

First suppose that $e_1\in\sigma$.  Then $R$ is already known.  Starting
from any known boundary triple and traversing the triangulation tree,
each unknown support vertex has at most $C_1$ choices.  Hence the
remaining support has at most $(C_1)^{2k-2j-1}$ choices.  There are
$k-j$ unspecified non-root cycle hyperedges, and each requires $r-3$
switches.  The number of records extending the fixed identification is
therefore at most
\[
 (C_1)^{2k-2j-1+(r-3)(k-j)}
 =(C_1)^{(r-1)(k-j)-1}.
\]

Now suppose that $e_1\notin\sigma$.  Choose one specified cycle hyperedge
$e_i$.  Reversing its $r-3$ switches recovers the ordered set $R$, with
at most $(C_1)^{r-3}$ choices.  The remaining support again has at most
$(C_1)^{2k-2j-1}$ choices.  There are now $k-1-j$ unspecified non-root
hyperedges.  Thus the same upper bound results:
\[
 (C_1)^{r-3+2k-2j-1+(r-3)(k-1-j)}
 =(C_1)^{(r-1)(k-j)-1}.
\]
Consequently,
\[
 \Delta_j(\KK)
 \leq k^j(r!)^j (C_1)^{(r-1)(k-j)-1}
 \leq (C_1)^{(r-1)(k-j)}
 =D^{k-j},
\]
where the penultimate inequality follows from the choice of $C_1$ given by \eqref{defi:constant hierarchy}.
Finally, $\Delta_k(\KK)\leq1=D^0$, because a set of $k$ distinct
hyperedges determines at most one member of the collection.  This completes the
proof.
\end{proof}
 
\subsection{Supersaturation from a large shadow, general case}\label{subsec:large shadow even}

We now treat the hypergraphs satisfying
condition~\ref{itm: Even Case Shadow} of Proposition~\ref{prop:kEven}.
The proof applies to every $r\geq4$ and $k\geq3$, as well as to
$r=3$ and even $k\geq4$.

We first provide a brief sketch of the proof.
Lemma~\ref{lem:Large Shadow} gives a set $T$ and a large
$(r-1)$-graph $F\subseteq\partial H_1$ with $V(F)\cap T=\varnothing$.
It also ensures that every $f\in F$ has at least
$\lfloor(k+1)/2\rfloor$ vertices $z\in T$ for which
$f\cup\{z\}\in H_1$, and in the proof we fix exactly this many such
vertices for each $f$.  Proposition~\ref{prop:many-complete-partite}
then gives at least $\gamma n^{(r-1)t}$ copies of
$K_{t,\ldots,t}^{(r-1)}$ in $F$, where $\gamma>0$ is a constant.
From each such copy, Lemma~\ref{lem:cycle-from-shadow-standard} allows
us to choose a copy $C$ of $C_k^{(r)}$ in which every hyperedge is the
union of a hyperedge of $F$ and one vertex of $T$.

Put $s\coloneqq|V(C)\cap T|$.  Since each vertex of a linear cycle
belongs to at most two of its hyperedges, $\lceil k/2\rceil\leq s\leq
k$.  If $s=k$, deleting the unique vertex of $T$ from every hyperedge
of $C$ gives a copy of $C_k^{(r-1)}$ in $F$, and $C$ contains exactly
$(r-2)k$ vertices of $V(F)$.  If $s<k$, the resulting hyperedges of
$F$ form $k-s$ vertex-disjoint linear paths, and $C$ contains
$(r-2)k+(k-s)$ vertices of $V(F)$.  We choose the value of $s$ that
occurs most often.  The number of vertices of $V(F)$ in the chosen
cycle is $(r-2)k+(k-s)$, so the same counting argument, with $s$ left
as a parameter, applies whether or not $s=k$.  The choices of vertices
in $V(F)$ give the factor $n^{(r-2)k+(k-s)}$ in the lower bound for
the number of cycles, whereas choosing the vertices from the fixed
lists contributes only a constant factor.

\begin{lemma}\label{lem: Supersat Large shadow even}
Let $r,k\geq3$ be integers. Assume either that $r\geq4$, or that
$r=3$ and $k\geq4$ is even. Let $\varepsilon>0$, and let $C_1$ be chosen according to~\eqref{defi:constant hierarchy}. Then there exists an integer
$n_0=n_0(\varepsilon,k,r,C_1)$
such that, for every $n\geq n_0$, the following holds. If $H_1$ is an $n$-vertex $\left\lfloor\frac{k+1}{2}\right\rfloor$-full $r$-graph with
$|H_1|\geq\frac{\varepsilon}{2}\binom{n}{r-1}$
and
$|\partial H_1|\geq\frac{\varepsilon}{4C_1}\binom{n}{r-1}$,
then there exists a collection $\KK$ of copies of $C_k^{(r)}$ contained in $H_1$ satisfying:
    \vspace{4pt}
    \begin{enumerate}[label=\textnormal{(N\arabic*)},itemsep=4pt]
\item\label{itm: supersat shadow even 1} $|\KK|\geq \frac{D^{k-1} n^{r-1}}{K}$,
\item\label{itm: supersat shadow even 2} $\Delta_{\ell}(\KK)\leq D^{k-\ell}$ for $\ell\in [k]$,
    \end{enumerate}
    \vspace{4pt}
   where $D\coloneqq n^{1/(k-1)}$ and $K\coloneqq\log n$.
\end{lemma}

\begin{proof}
Set $\delta\coloneqq\frac{\varepsilon}{4C_1\cdot r^r}$.  Since $n$ is sufficiently large, we have
$\binom{n}{r-1}>\frac{n^{r-1}}{r^r}$, and hence
\[
|\partial H_1|
\geq\frac{\varepsilon}{4C_1}\binom{n}{r-1} >
\frac{\varepsilon}{4C_1\cdot r^r}n^{r-1}
=\delta n^{r-1}.
\]
We may therefore apply Lemma~\ref{lem:Large Shadow} to $H_1$.  This
yields a set $T\subseteq V(H_1)$ and an $(r-1)$-graph
$F\subseteq\partial H_1$ such that
$V(F)\subseteq V(H_1)\setminus T$,
$|N_{H_1}(f)\cap T|\geq\left\lfloor\frac{k+1}{2}\right\rfloor$
for every $f\in F$, and
\begin{equation}\label{Sec3:lowerF}
|F|\geq
\frac{\varepsilon}{r^r\cdot 2^{k+r+4}\cdot C_1}n^{r-1}.
\end{equation}
For every $f\in F$, choose a set
$\mathcal{L}(f)\subseteq N_{H_1}(f)\cap T$
of size
$|\mathcal{L}(f)|=\left\lfloor\frac{k+1}{2}\right\rfloor$.
We use these sets as the lists associated with the hyperedges of $F$ throughout the remainder of the proof.

Let $t=(r2^k)^{4r}$ be the integer specified in
Lemma~\ref{lem:cycle-from-shadow-standard}. Set
\begin{equation}\label{defi: alpha Even case}
\alpha\coloneqq\frac{\varepsilon}{r^r\cdot 2^{k+r+4}\cdot C_1}.
\end{equation}

Let $\mathcal{T}_1$ be the collection of distinct copies of
$K_{t,\ldots,t}^{(r-1)}$ contained in $F$.  By
Proposition~\ref{prop:many-complete-partite}, there is a constant
$\gamma=\gamma(\varepsilon,r,k,C_1)>0$ such that
\begin{equation}\label{eq:number of Ktts}
|\mathcal{T}_1|\geq \gamma n^{(r-1)t}.
\end{equation}
Define $Q=Q(r,k)$ by
\begin{equation}\label{eq:multipartite-completion-factor}
 Q\coloneqq(r-1)^{(r-1)t}.
\end{equation}
For every $(r-1)$-graph $A\subseteq F$ with
$a\coloneqq|V(A)|\leq(r-1)t$, the number of members $G\in\mathcal T_1$
that contain $A$ is at most
\begin{equation}\label{eq:multipartite-completion-bound}
 Qn^{(r-1)t-a}.
\end{equation}
Indeed, there are at most $n^{(r-1)t-a}$ ways to choose the remaining
vertices of $G$.  Once these vertices have been chosen, assigning all
$(r-1)t$ vertices to the $(r-1)$ parts in their copy of
$K_{t,\ldots,t}^{(r-1)}$ can be done in at most
$(r-1)^{(r-1)t}=Q$ ways.

Let $\prec_2$ be an arbitrary linear ordering of the copies of $C_k^{(r)}$ contained in $H_1$. Applying Lemma~\ref{lem:cycle-from-shadow-standard} with
$H=H_1$ and with the list system satisfying
$\mathcal L(f)\subseteq N_{H_1}(f)\cap T$ and
$|\mathcal L(f)|=\lfloor(k+1)/2\rfloor$ for every $f\in F$, we obtain that for each $G\in\mathcal{T}_1$, there exists a copy of $C_k^{(r)}$ in $H_1$ such that every hyperedge consists of a hyperedge of $G$ together with exactly one vertex from $T$. For each $G\in\mathcal{T}_1$, we assign exactly one such copy of $C_k^{(r)}$; if there are multiple choices, we choose the first one under $\prec_2$. Let $\KK_1$ be the collection of copies of $C_k^{(r)}$ obtained from $\mathcal{T}_1$ in this way. Lemma~\ref{lem:cycle-from-shadow-standard} ensures that each hyperedge of every member of $\KK_1$ consists of exactly one vertex from $T$ and $r-1$ vertices from $V(F)\subseteq V(H_1)\setminus T$. Since any vertex of $T$ belongs to at most two hyperedges of a copy of $C_k^{(r)}$, it follows that every member of $\KK_1$ contains between $\lceil k/2\rceil$ and $k$ vertices from $T$.

For $\lceil k/2\rceil \le s \le k$, a copy of $C_k^{(r)}$ in $\KK_1$ is said to be of \emph{type}-$s$ if its vertex set contains exactly $s$ vertices from $T$. Since each copy of $C_{k}^{(r)}$ consists of $(r-1)k$ vertices, each type-$s$ cycle consists of $(r-1)k-s$ vertices from $V(F)$ and $s$ vertices from $T$. Note that a type-$k$ copy of $C_{k}^{(r)}$ consists of a cycle $C_{k}^{(r-1)}$ contained in $F$ together with $k$ vertices from $T$, where each vertex is added to a different hyperedge of the cycle. Observe that for type-$s$ members of $\KK_1$ with $s<k$, each copy of $C_{k}^{(r)}$ consists of $k-s$ disjoint $(r-1)$-paths contained in $F$, and each individual path contains at least two hyperedges of $F$. 

We say that a member of $\mathcal{T}_1$ is of \emph{type}-$s$ if the copy of $C_k^{(r)}$ in $\KK_1$ which is assigned to it is of type-$s$. Let $\mathcal{T}_2$ be the subcollection of copies in $\mathcal{T}_1$ whose type occurs most frequently among all possible types. By the pigeonhole principle, from~\eqref{eq:number of Ktts} it follows that 
\begin{equation}\label{eq: Lower bound T2}
|\mathcal{T}_2| \geq \frac{|\mathcal{T}_1|}{k}\geq \frac{\gamma}{k} n^{(r-1)t}.
\end{equation}

Let $\KK_2$ be the collection of distinct copies of $C_{k}^{(r)}$ which are assigned to members of $\mathcal{T}_2$. We now construct the desired balanced subcollection $\KK\subseteq \KK_2$ satisfying conditions~\ref{itm: supersat shadow even 1} and~\ref{itm: supersat shadow even 2}. We begin by setting $\KK\coloneqq\emptyset$ and then add cycles from $\KK_2$ to $\KK$ one at a time until condition~\ref{itm: supersat shadow even 1} is satisfied while ensuring that condition~\ref{itm: supersat shadow even 2} is preserved at every step. Note that the empty collection clearly satisfies condition~\ref{itm: supersat shadow even 2}. An unordered $\ell$-tuple of hyperedges $\{e_1,\ldots, e_\ell\}\subset H_1$ with $1\leq\ell\leq k-1$ is called \emph{saturated} if it is contained in exactly $\lfloor D^{k-\ell}\rfloor $ cycles from $\KK$. In particular, a single hyperedge is saturated if it belongs to $\lfloor D^{k-1}\rfloor$ cycles from  $\KK$. 

For any fixed $1 \leq \ell \le k-1$, we bound the number of saturated $\ell$-tuples $\sigma = \{e_1,\ldots, e_\ell\} \subseteq H_1$ as follows. We double count pairs $(\{e_1,\ldots, e_{\ell}\}, C^*)$ such that $C^*\in \KK$ and $\{e_1,\ldots, e_{\ell}\}\subset C^{*}$ is saturated. This yields that the number of saturated $\ell$-tuples $\sigma = \{e_1, \ldots, e_{\ell}\} \subseteq H_1$ is at most 
\[
 \frac{\binom{k}{\ell}|\KK|}{\lfloor D^{k-\ell}\rfloor}.
\]

Consequently, if $\KK$ fails to satisfy condition~\ref{itm: supersat shadow even 1}, then, for every $1\leq \ell\leq k-1$, the number of saturated $\ell$-tuples $\sigma=\{e_1,\ldots,e_\ell\}\subseteq H_1$ is bounded above by

\begin{equation}
\label{boundonsatltuplesinH1}
  \frac{\binom{k}{\ell}\cdot |\KK|}{\lfloor D^{k-\ell}\rfloor}\leq \frac{2\cdot \binom{k}{\ell}\cdot D^{\ell-1}\cdot n^{r-1}}{K}.   
\end{equation}

Let $s$ be the common type of the members of $\mathcal T_2$.  We
first derive a lower bound for $|\KK_2|$ by obtaining an upper bound
on how many members from $\mathcal{T}_2$ get assigned a fixed member
$\widehat{C}\in \KK_2$. Since every member of $\KK_2$ is of type-$s$,
each copy of $C_k^{(r)}$ in $\KK_2$ contains exactly $s$ vertices from
$T$ and hence $(r-2)k+(k-s)$ vertices from $V(F)$.  Let $A$ be the
$(r-1)$-graph obtained by deleting from each hyperedge of $\widehat C$
its unique vertex in $T$.  Every member of $\mathcal T_2$ assigned to
$\widehat C$ contains $A$, and $|V(A)|=(r-2)k+(k-s)$.  Therefore,
\eqref{eq:multipartite-completion-bound} shows that at most
$Qn^{(r-1)t-(r-2)k-(k-s)}$ members of $\mathcal{T}_2$ are assigned to
$\widehat C$.  This
upper bound and~\eqref{eq: Lower bound T2} give
\begin{equation}\label{ineq:lowerKstar}
\begin{split}|\KK_2|&\geq \frac{|\mathcal{T}_2|}{Qn^{(r-1)t-(r-2)k-(k-s)}}\geq\frac{\gamma}{kQ} n^{(r-2)k+(k-s)}. 
\end{split}
\end{equation}

We now bound the number of cycles in $\KK_2$ containing a fixed
saturated $\ell$-tuple
$\sigma=\{e_1,\ldots,e_\ell\}$ of hyperedges of $H_1$.  There are at
most $k^\ell$ ways to specify the positions occupied by these
hyperedges in a copy of $C_k^{(r)}$.  The hyperedges in $\sigma$
contain at least $\ell(r-2)+1$ distinct vertices of $V(F)$.  Indeed,
because $\ell<k$, the subgraph of the underlying graph cycle induced
by these $\ell$ hyperedges is a union of paths and therefore has at
most $\ell-1$ pairs of consecutive hyperedges.  The corresponding
$(r-1)$-sets in $F$ consequently have union of size at least
$\ell(r-1)-(\ell-1)=\ell(r-2)+1$.
There are at most $\binom{rk}{s}<(rk)^s$ ways to specify the positions of the vertices of $T$ within a copy of $C_k^{(r)}$ in $\KK_2$ containing $\sigma$. Since each copy of $C_k^{(r)}$ in $\KK_2$ contains exactly $(r-2)k+(k-s)$ vertices from $F$, and since $\sigma$ already spans at least $\ell(r-2)+1$ of these vertices, there are at most $n^{(r-2)k+(k-s)-(\ell(r-2)+1)}=n^{(r-2)(k-\ell)+(k-s)-1}$ ways to choose the remaining vertices from $F$ lying in a copy of $C_k^{(r)}$ in $\KK_2$ that contains $\sigma$. Moreover, once the vertices of the copy of $C_k^{(r)}$ belonging to $F$ are fixed, using the assumption that $|\mathcal{L}(f)|=\left\lfloor\frac{k+1}{2}\right\rfloor$ for every $f\in F$, there are at most $\left\lfloor\frac{k+1}{2}\right\rfloor^{s}<k^s$ choices for the vertices from $T$ that complete them to a copy of $C_k^{(r)}$ in $\KK_2$. Thus, for any $\ell \in [k-1]$, the number of members of $\KK_2$ containing any fixed saturated $\ell$-tuple $\sigma$ is at most

\begin{equation}\label{eq:shadow-even-fixed-tuple}
(rk)!k^{\ell}\cdot(rk)^{s}\cdot k^s\cdot n^{(r-2)(k-\ell)+(k-s)-1}
<
 (rk)!(rk)^{3k}\cdot n^{(r-2)(k-\ell)+(k-s)-1}.
\end{equation}
Here the factor $(rk)!$ is an upper bound for the number of ways to
assign the vertices already present in $\sigma$ to vertex positions
in a copy of $C_k^{(r)}$.  Combining~\eqref{eq:shadow-even-fixed-tuple} with
\eqref{boundonsatltuplesinH1}, we see that the total number of members
of $\KK_2$ containing a saturated $\ell$-tuple for some
$\ell\in[k-1]$ is at most
\begin{equation}\label{boundsatcopies}
\begin{aligned}
&\sum_{\ell=1}^{k-1}(rk)!(rk)^{3k}
 n^{(r-2)(k-\ell)+(k-s)-1} \cdot
 \frac{2\binom{k}{\ell}D^{\ell-1}n^{r-1}}{K}\\
&\quad\leq
 \frac{2(rk)!(rk)^{3k}n^{(r-2)k+(k-s)}}{K}
 \left(k+k^{k+1}D^{k-2}n^{-1}\right) <\frac{|\KK_2|}{2}.
\end{aligned}
\end{equation}
For the first inequality, factor out
$2(rk)!(rk)^{3k}n^{(r-2)k+(k-s)}/K$.  The term with $\ell=1$ then
equals $k$.  For $2\leq\ell\leq k-1$, we use
$n^{-(r-2)(\ell-1)}\leq n^{-1}$,
$D^{\ell-1}\leq D^{k-2}$, and
$\sum_{\ell=2}^{k-1}\binom{k}{\ell}\leq k^{k+1}$.
The last inequality follows from~\eqref{ineq:lowerKstar}, since $\gamma$
depends only on $\varepsilon$, $k$, $r$, and $C_1$, while
$D=n^{1/(k-1)}$, $K=\log n$, and
$D^{k-2}n^{-1}=n^{-1/(k-1)}=o(1)$.

If $\KK$ does not satisfy condition~\ref{itm: supersat shadow even 1},
then, since $(r-2)k\geq r$ in the ranges of the lemma, we also have
\[
|\KK|\leq \frac{D^{k-1}n^{r-1}}{K}
=\frac{n^r}{\log n}
\leq \frac{n^{(r-2)k}}{\log n}
\overset{\eqref{ineq:lowerKstar}}{<}\frac{|\KK_2|}{4},
\]
for sufficiently large $n$. The inequalities $|\KK|<|\KK_2|/4$
and~\eqref{boundsatcopies} imply that at least
$|\KK_2\setminus \KK|-|\KK_2|/2\geq |\KK_2|/4$ members of
$\KK_2\setminus \KK$ contain no saturated $\ell$-tuple for
$\ell\in[k-1]$. We may therefore add such a member of
$\KK_2\setminus\KK$ to $\KK$ while preserving
condition~\ref{itm: supersat shadow even 2}. This condition holds
trivially when $\ell=k$, since any collection of $k$ distinct
hyperedges is contained in at most one copy of $C_k^{(r)}$ in the
resulting collection. Repeating this process until
condition~\ref{itm: supersat shadow even 1} holds, we obtain a
collection $\KK$ satisfying conditions~\ref{itm: supersat shadow even 1}
and~\ref{itm: supersat shadow even 2}. This proves
Lemma~\ref{lem: Supersat Large shadow even}.
\end{proof}

\subsection{Supersaturation from large shadow, odd case} \label{subsec:large shadow}

To conclude this section, we obtain a supersaturation result for
cycles $C_k^{(3)}$ with odd length $k\geq5$ in an ambient
$3$-graph $H$ containing a full subhypergraph $H_1\subseteq H$ as in
condition~\ref{itm: Odd Case Shadow} of
Proposition~\ref{prop:kOdd}. Among the supersaturation results established in this paper, the one obtained in this subsection yields the smallest balanced collection with respect to the parameter $D$.

The two cases in the proof correspond to the two alternatives in
Lemma~\ref{lem:a copy of Ck from Ktt}.  Under
alternative~\ref{itm: cycle from shadow 1}, each copy of $K_{t,t}$
supplies a cycle whose hyperedges are obtained using the fixed lists.  Under
alternative~\ref{itm: cycle from shadow 2}, each copy of $K_{t,t}$
supplies $\Theta(C_2)$ cycles, each obtained by replacing one
hyperedge of a copy of $C_{k-1}^{(3)}$ by two hyperedges.  These two
collections have different multiplicities and require different
estimates for cycles containing a prescribed tuple of hyperedges, so
we treat them separately.

\begin{lemma}\label{lem: Supersat Large shadow odd}
Let $k\geq 5$ be an odd integer and let $0<\varepsilon<1$. Let $C_0$, $C_1$, and $C_2$ be chosen according to~\eqref{defi:constant hierarchy}. There exists an integer
$n_0=n_0(\varepsilon,k,C_0,C_1,C_2)$
such that the following holds for every $n\geq n_0$. Let $H$ be an $n$-vertex $3$-graph satisfying
$
|H|\leq C_0\binom{n}{2},
$
and let $H_1\subseteq H$ be a $\left\lfloor\frac{k+1}{2}\right\rfloor$-full $3$-graph such that
$
|H_1|\geq\frac{\varepsilon}{4}\binom{n}{2}
\quad\text{and}\quad
|\partial H_1|\geq\frac{\varepsilon}{8C_1}\binom{n}{2}.
$
Suppose further that, for every $xyz\in H_1$,
$
\max\{d_H(xy),d_H(yz),d_H(xz)\}\geq C_2.
$
Then there exists a collection $\KK$ of copies of $C_k^{(3)}$ contained in $H$ satisfying:

            \vspace{4pt}
    \begin{enumerate}[label=$\mathrm{(O\arabic*)}$,itemsep=4pt]    
        \item\label{itm: supersat shadow odd 1} $|\KK|\geq \frac{D^{k-1} n^{2}}{K}$,
        \item\label{itm: supersat shadow odd 2} $\Delta_{\ell}(\KK)\leq D^{k-\ell}$ for $\ell\in [k]$,
    \end{enumerate}
    \vspace{4pt}
    where $D \coloneqq C_2$ and $K \coloneqq (C_1)^{2^{100k}}$.
\end{lemma}
\begin{proof}

Set $\delta\coloneqq\frac{\varepsilon}{2^{10}C_1}$. By~\eqref{defi:constant hierarchy}, we may assume that
$C_2>\max\{6C_0/\delta,6k\}$.
Moreover, since
\[
|\partial H_1|\geq\frac{\varepsilon}{8C_1}\binom{n}{2}>\delta n^2
\]
for $n$ sufficiently large, we may apply
Lemma~\ref{lem:Large Shadow odd} to the pair $(H,H_1)$.  We obtain a
set $T\subseteq V(H)$ and a graph $F\subseteq\partial H_1$
satisfying conditions~\ref{itm: lemma Large shadow odd i}--
\ref{itm: lemma Large shadow odd iv}. In particular,
\[
|F|\geq \frac{\varepsilon}{2^{k+15}\cdot C_1}n^{2}.
\]
For every $xy\in F$, choose a vertex
$\alpha_{xy}\in N_{H_1}(xy)\cap T$ for which the three inequalities
in~\ref{itm: lemma Large shadow odd iv} hold, and then choose a set
$\mathcal{L}(xy)\subseteq N_{H_1}(xy)\cap T$
of size
$|\mathcal{L}(xy)|=\left\lfloor\frac{k+1}{2}\right\rfloor$
that contains $\alpha_{xy}$. We use this fixed list system throughout the remainder of the proof.
Let 
\begin{equation}\label{eq: defi t odd case}
t\coloneqq3^{12}\cdot 2^{12k}<2^{20k},
\end{equation} be as defined in Lemma~\ref{lem:a copy of Ck from Ktt} and let \begin{equation}\label{defi alpha odd case}
\alpha\coloneqq \frac{\varepsilon}{C_1\cdot 2^{k+15}}.
\end{equation}
Let $\mathcal{T}_1$ be the collection of distinct copies of $K_{t,t}$ contained in $F$. Applying Proposition~\ref{prop: Many Ktt's} to $F$, we obtain  
\begin{equation}
\label{eq:number of Ktts2}
|\mathcal{T}_1|\geq  \frac{1}{2}\left(\frac{2\alpha}{e}\right)^{t^2} \frac{n^{2t}}{t^{2t}} = \frac{2^{t^2}\cdot \alpha^{t^{2}}}{2\cdot e^{t^2}\cdot t^{2t}}\cdot n^{2t}.
\end{equation}
Let $\prec_2$ be an arbitrary linear order on the copies of
$C_{k}^{(3)}$ contained in $H$.  Apply
Lemma~\ref{lem:a copy of Ck from Ktt} with $H$ and $H_1$ playing the
roles of $H$ and $H'$, respectively, and with the fixed list system
$\mathcal L$.  For every $G\in\mathcal T_1$, either alternative
\ref{itm: cycle from shadow 1} holds, or alternative
\ref{itm: cycle from shadow 2} holds and supplies at least
$\lceil C_2/12\rceil$ copies of $C_k^{(3)}$.

Let $\mathcal T^{(1)}$ be the members of $\mathcal T_1$ for which
alternative~\ref{itm: cycle from shadow 1} holds.  If
$|\mathcal T^{(1)}|\geq|\mathcal T_1|/2$, put
$\mathcal T_2\coloneqq\mathcal T^{(1)}$; this is
Case~\ref{lem: supersat odd case 1}.  Otherwise, put
$\mathcal T_2\coloneqq\mathcal T_1\setminus\mathcal T^{(1)}$; every member of
this collection satisfies alternative~\ref{itm: cycle from shadow 2},
and $|\mathcal T_2|>|\mathcal T_1|/2$.  This is
Case~\ref{lem: supersat odd case 2}.
If Case~\ref{lem: supersat odd case 1} holds, then, as in the proof of Lemma~\ref{lem: Supersat Large shadow even}, we say that a copy of $C_k^{(3)}$ contained in $H$ is of \emph{type}-$s$ if it contains exactly $s$ vertices from $T$. For each $G\in\mathcal{T}_2$, which is a copy of $K_{t,t}$ contained in $F$, we associate the first copy of $C_k^{(3)}$ satisfying condition~\ref{itm: cycle from shadow 1} with respect to the ordering $\prec_2$; note that the assumption of Case~\ref{lem: supersat odd case 1}, together with Lemma~\ref{lem:a copy of Ck from Ktt}, guarantees that at least one such cycle exists. We say that a member of $\mathcal{T}_2$ is of type-$s$ if the copy of $C_k^{(3)}$ associated with it is of type-$s$. Let $\mathcal{T}_3$ be the collection of copies of $K_{t,t}$ of the most common type in $\mathcal{T}_2$ and let $\KK_1$ be the collection of copies of $C_{k}^{(3)}$ contained in $H$ associated to members of $\mathcal{T}_3$.

If Case~\ref{lem: supersat odd case 2} holds, then for each member of
$\mathcal{T}_2$, we associate with it the first
$\lceil C_2/12\rceil$ copies of $C_k^{(3)}$ satisfying
condition~\ref{itm: cycle from shadow 2} with respect to the ordering
$\prec_2$. Let $\KK_1$ be the collection of copies of $C_{k}^{(3)}$
contained in $H$ associated to members of $\mathcal{T}_2$. Thus, each
member of $\mathcal{T}_2$ is associated with exactly
$\lceil C_2/12\rceil$ members of $\KK_1$.

We now construct the desired balanced subcollection $\KK\subseteq \KK_1$ satisfying conditions~\ref{itm: supersat shadow odd 1} and~\ref{itm: supersat shadow odd 2}. We begin by setting $\KK\coloneqq\emptyset$ and then add members of $\KK_1$ to $\KK$ one at a time until condition~\ref{itm: supersat shadow odd 1} is satisfied while ensuring that condition~\ref{itm: supersat shadow odd 2} is preserved at every step. Note that the empty collection clearly satisfies condition~\ref{itm: supersat shadow odd 2}.
An unordered $\ell$-tuple of distinct hyperedges $\{e_1,\ldots, e_\ell\}\subseteq H$ with $1\leq\ell\leq k-1$ is called \emph{saturated} if it is contained in exactly $\lfloor D^{k-\ell}\rfloor$ cycles from $\KK$. For any fixed $1\leq\ell\leq k-1$, we bound the number of saturated $\ell$-tuples $\sigma=\{e_1,\ldots,e_\ell\}\subseteq H$ by double counting pairs $(\{e_1,\ldots,e_\ell\},C^*)$ such that $C^*\in\KK$ and $\{e_1,\ldots,e_\ell\}\subseteq C^*$ is saturated. This yields that, for any $\ell\in[k-1]$, the number of saturated $\ell$-tuples $\sigma=\{e_1,\ldots,e_\ell\}\subseteq H$ is at most $\frac{\binom{k}{\ell}|\KK|}{\lfloor D^{k-\ell}\rfloor}$. 

Therefore, if $\KK$ does not satisfy condition \ref{itm: supersat shadow odd 1}, then we obtain that, for any $\ell \in [k-1]$, the number of saturated $\ell$-tuples $\sigma=\{e_1,\ldots,e_\ell\}\subseteq H$ is at most 
\begin{equation}
\label{boundonsatltuplesinH1Odd}
  \frac{\binom{k}{\ell}\cdot |\KK|}{\lfloor D^{k-\ell}\rfloor}\leq \frac{2\cdot \binom{k}{\ell}\cdot D^{\ell-1}\cdot n^{2}}{K}\leq \frac{2\cdot k^\ell\cdot D^{\ell-1}\cdot n^2}{K}.   
\end{equation}

We now treat Case~\ref{lem: supersat odd case 1}, in which
$|\mathcal T^{(1)}|\geq|\mathcal T_1|/2$, and
Case~\ref{lem: supersat odd case 2}, in which
$|\mathcal T^{(1)}|<|\mathcal T_1|/2$.

\vspace{0.08in}
\noindent \textbf{Case }\customlabel{lem: supersat odd case 1}{\textbf{1}}\textbf{:}
\textit{$|\mathcal T^{(1)}|\geq|\mathcal T_1|/2$.}\vspace{0.08in}

We use the same type parameter as in
Lemma~\ref{lem: Supersat Large shadow even}, and give the required
multiplicity bounds explicitly.  Recall that a copy of
$C_{k}^{(3)}$ from $\KK_1$ is of type-$s$ if it contains exactly $s$
vertices from $T$ and a copy of $K_{t,t}$ from $\mathcal{T}_3$ is of
type-$s$ if its associated cycle is type-$s$. Recall also that
$\mathcal{T}_3$ is the subcollection of the most common type in
$\mathcal{T}_2$ and $\KK_1$ is the collection of copies of
$C_{k}^{(3)}$ contained in $H$ associated to members of
$\mathcal{T}_3$. Put $Q_3\coloneqq2^{2t}$. Once the vertices of a completion
to $K_{t,t}$ have been chosen, there are at most $Q_3$ possible
bipartitions.  A type-$s$ cycle contains $k+(k-s)$ vertices of $F$,
so it is associated with at most
$Q_3n^{2t-k-(k-s)}$ members of $\mathcal T_3$.  Since
$|\mathcal T_3|\geq|\mathcal T_2|/k$ and
$|\mathcal T_2|\geq|\mathcal T_1|/2$, double counting gives
\begin{equation}\label{lower K2 odd case case 1}
\begin{split}
|\KK_1|&\geq \frac{|\mathcal{T}_3|}{Q_3n^{2t-k-(k-s)}}
\geq \frac{|\mathcal{T}_2|}{kQ_3 n^{2t-k-(k-s)}}
\geq \frac{|\mathcal{T}_1|}{2kQ_3 n^{2t-k-(k-s)}} \\
&\overset{\eqref{eq:number of Ktts2}}{\geq}
 \frac{2^{t^2}\cdot\alpha^{t^{2}}}
 {4kQ_3\cdot e^{t^{2}}\cdot t^{2t}}
  n^{k+(k-s)}.
\end{split}
\end{equation}
Recall that a saturated $\ell$-tuple is a tuple $\sigma$ of $\ell$
hyperedges of $H$ contained in exactly
$\left\lfloor D^{k-\ell}\right\rfloor$ members of $\KK$.  Fix such a
tuple.  There are at most $k^\ell$ choices for the positions of its
hyperedges and at most $6^\ell$ ways to assign the vertices of those hyperedges
to their roles in the cycle.  The two-vertex parts in $F$ of the
specified hyperedges contain at least $\ell+1$ vertices.  A type-$s$ cycle
contains $k+(k-s)$ vertices of $F$, so its remaining vertices in $F$
have at most $n^{(k-\ell)+(k-s)-1}$ choices.  After these vertices are
fixed, the list system supplies fewer than $k^s$ choices for the
distinct vertices in $T$.  Consequently, for every
$\ell\in[k-1]$, the number of members of $\KK_1$ containing $\sigma$
is at most
\begin{equation}\label{eq:shadow-odd-fixed-tuple-case1}
 (3k)^{3k}n^{(k-\ell)+(k-s)-1}.
\end{equation}
Combining~\eqref{eq:shadow-odd-fixed-tuple-case1} with
\eqref{boundonsatltuplesinH1Odd}, we obtain that the number of members
of $\KK_1$ containing a saturated $\ell$-tuple for some
$\ell\in[k-1]$ is at most
\begin{equation}\label{boundsatcopies3}
\begin{split}
\sum_{\ell=1}^{k-1}\left( (3k)^{3k} n^{(k-\ell)+(k-s)-1} \frac{2 \binom{k}{\ell}\cdot D^{\ell-1}\cdot n^{2}}{K}\right)&\leq \frac{2 (3k)^{3k} n^{k+(k-s)}}{K}\left(k+ k^{k+1} D^{k-2} n^{-1}\right)\\
&\overset{\eqref{lower K2 odd case case 1}}{<} \frac{\lvert\KK_1\rvert}{2},
\end{split}
\end{equation}
where the last inequality follows from~\eqref{defi:constant hierarchy}, since $n$ is sufficiently large and $K=(C_1)^{2^{100k}}$ is sufficiently large in terms of $\varepsilon$ and $k$.

If $\KK$ does not satisfy condition \ref{itm: supersat shadow odd 1}, then using that $k\geq 5$, we obtain
\[
|\KK|\leq \frac{D^{k-1}n^{2}}{K} = \frac{(C_2)^{k-1}\cdot n^{2}}{(C_1)^{2^{100k}}}\overset{\eqref{lower K2 odd case case 1} }{<}\frac{|\KK_1|}{4},
\]
for $n$ large enough. The inequalities $|\KK|<|\KK_1|/4$
and~\eqref{boundsatcopies3} show that, if $\KK$ does not satisfy
condition~\ref{itm: supersat shadow odd 1}, then at least
$|\KK_1\setminus \KK|-|\KK_1|/2\geq |\KK_1|/4$ members of
$\KK_1\setminus \KK$ contain no saturated $\ell$-tuple for any
$\ell\in[k-1]$. We may therefore add such a member of
$\KK_1\setminus\KK$ to $\KK$ while preserving
condition~\ref{itm: supersat shadow odd 2}. This condition holds
trivially when $\ell=k$ because any collection of $k$ distinct
hyperedges is contained in at most one copy of $C_k^{(3)}$ in the
resulting collection $\KK$. Repeating this process until
condition~\ref{itm: supersat shadow odd 1} holds gives a collection
$\KK$ satisfying conditions~\ref{itm: supersat shadow odd 1}
and~\ref{itm: supersat shadow odd 2}.

\vspace{0.08in}
\noindent \textbf{Case }\customlabel{lem: supersat odd case 2}{\textbf{2}}\textbf{:}
\textit{$|\mathcal T^{(1)}|<|\mathcal T_1|/2$.}\vspace{0.08in}

Recall that $\mathcal{T}_2$ is the collection of copies of $K_{t,t}$ from $\mathcal{T}_1$ covered by this case, and $\KK_1$ is the collection of copies of $C_{k}^{(3)}$ obtained by applying Lemma~\ref{lem:a copy of Ck from Ktt} to members of $\mathcal{T}_2$. First, we obtain a lower bound for $|\KK_1|$. 

For each member of $\KK_1$ with hyperedges $e_1,\ldots,e_k$ as described in \ref{itm: cycle from shadow 2}, note that the two hyperedges $e_{k-1}$ and $e_{k}$ each contain exactly two vertices from $T$ and one vertex from $F$. The other $k-2$ hyperedges $e_1,\ldots, e_{k-2}$ each contain exactly one vertex from $T$ and two vertices from $F$. Moreover, $\bigcup_{i=1}^{k}e_i$ consists of $k+1$ distinct vertices from $T$ and $k-1$ distinct vertices from $F$. 

Every copy of $C_k^{(3)}$ in $\KK_1$ contains exactly $k-1$ vertices
from $F$. Moreover, any fixed set of $k-1$ vertices from $F$ is
contained in at most $Q_3n^{2t-(k-1)}$ copies of $K_{t,t}$ in
$\mathcal T_2$, while each member of $\mathcal T_2$ is assigned at
least $C_2/12$ copies of $C_k^{(3)}$ from $\KK_1$. Consequently,
\begin{equation}\label{ineq:lowerKstarOdd}
\begin{split}
|\KK_1|&\geq \frac{C_2\cdot |\mathcal{T}_2|}{12Q_3n^{2t-(k-1)}}\geq \frac{C_2\cdot |\mathcal{T}_1|}{24Q_3n^{2t-(k-1)}} \overset{\eqref{eq:number of Ktts2}}{\geq} \frac{ C_2\cdot 2^{t^2}\cdot \alpha^{t^2} \cdot n^{2t}}{48Q_3\cdot e^{t^2}\cdot t^{2t}\cdot n^{2t-(k-1)}}\\
&\overset{\eqref{defi alpha odd case}}{\geq}\frac{ C_2\cdot 2^{t^2}\cdot \varepsilon^{t^2}}{48Q_3\cdot e^{t^2}\cdot t^{2t}\cdot  (C_1)^{t^2}\cdot 2^{(k+15)t^2}} n^{k-1}\overset{\eqref{eq: defi t odd case}}{\geq} \frac{C_2\cdot \varepsilon^{t^2}}{2^{(k+21)t^2}\cdot (C_1)^{t^2}}n^{k-1}.
\end{split}
\end{equation}

By~\eqref{defi:constant hierarchy}, $C_1$ is sufficiently large in terms of $\varepsilon$ and $k$. Since $t<2^{20k}$, it follows that
\begin{equation}\label{ineq: lower bound K Odd case}
K=(C_1)^{2^{100k}}> 12\cdot \frac{(2k)^{2k}\cdot 2^{(k+21)2^{40k}}\cdot (C_1)^{2^{40k}}}{\varepsilon^{2^{40k}}}>12\cdot \frac{(2k)^{2k}\cdot 2^{(k+21)t^{2}}\cdot (C_1)^{t^2}}{\varepsilon^{t^2}}.
\end{equation}

We now proceed to construct the desired balanced subcollection $\KK \subseteq \KK_1$ satisfying conditions \ref{itm: supersat shadow odd 1} and \ref{itm: supersat shadow odd 2} as in the proof of Lemma~\ref{lem: Supersat Large shadow even}. Let $\KK=\emptyset$ and iteratively add copies of $C_k^{(3)}$ from $\KK_1$ to $\KK$, one at a time, ensuring condition \ref{itm: supersat shadow odd 2} holds at each step, until $\KK$ satisfies condition \ref{itm: supersat shadow odd 1}. 

Recall that for any fixed $\ell \in [k-1]$, a saturated $\ell$-tuple $\sigma$ is an unordered $\ell$-tuple of distinct hyperedges $\{e_1,\ldots,e_\ell\}\subseteq H$ that is contained in exactly $\lfloor D^{k-\ell}\rfloor$ copies of $C_k^{(3)}$ from $\mathcal{K}$. We now bound the total number of copies of $C_k^{(3)}$ in $\mathcal{K}_1$ that contain a saturated $\ell$-tuple for some $\ell \in [k-1]$. To do this, we classify the saturated tuples $\sigma$ into three types based on the number of hyperedges in $\sigma$ that contain exactly two vertices from $T$, and bound the contribution of each type separately.

We first fix a saturated tuple $\sigma$ such that exactly two of its
hyperedges contain exactly two vertices from $T$, and bound the number
of members of $\KK_1$ containing $\sigma$.

For any copy of $C_k^{(3)}$ in $\mathcal K_1$, whose hyperedges are labeled as $e_1,\ldots,e_k$ according to \ref{itm: cycle from shadow 2}, the only hyperedges containing two vertices of $T$ are those corresponding to the labels $e_{k-1}=w_1z_{k-1}z_k$ and $e_k=w_{k-1}z_{k-1}z_{k+1}^{(m)}$ for some $m\in [\lceil C_2/12\rceil]$. Hence, if such a copy contains a saturated tuple $\sigma$, then the two hyperedges of $\sigma$ containing two vertices of $T$ must occupy the positions $e_{k-1}$ and $e_k$ in this labeling. Thus, there are at most two ways to assign these two hyperedges of $\sigma$ to the positions $e_{k-1}$ and $e_k$. This assignment determines the vertices of $\sigma$ corresponding to $w_1,w_{k-1},z_{k-1},z_k$ and $z_{k+1}^{(m)}$. Once these vertices are fixed, there are at most $(k-2)^{\ell-2}<k^\ell$ ways to choose the positions of the remaining members of $\sigma$ in the copy. After these $\ell-2$ positions are specified, there are at most $2^{\ell-2}$ ways to assign the vertices of the remaining hyperedges of $\sigma$ to the corresponding labels $w_i,w_{i+1}$ and $z_i$.

Now, at most $k-3-(\ell-2)=k-1-\ell$ vertices corresponding to the labels $w_2,\ldots, w_{k-2}$ have not been specified, and these can be chosen in at most $n^{k-1-\ell}$ ways. After choosing them, the vertices $w_1,\ldots,w_{k-1}$ together with $z_{k-1},z_k$, and $z_{k+1}^{(m)}$ have been specified. The remaining vertices corresponding to the labels $z_1,\ldots,z_{k-2}$ that do not belong to $\sigma$ can then be chosen in at most $\left\lfloor\frac{k+1}{2}\right\rfloor^{k-2-(\ell-2)}<k^{k-\ell}$ ways. This uniquely determines the vertices corresponding to the labels $w_1,\ldots,w_{k-1},z_1,\ldots,z_{k},z_{k+1}^{(m)}$, and thus, a copy of $C_{k}^{(3)}$ from $\KK_1$ that contains $\sigma$. It follows that the total number of copies of $C_{k}^{(3)}$ in $\KK_1$ containing $\sigma$ is at most 
\[
2\cdot  k^{\ell}\cdot 2^{\ell-2}\cdot n^{k-1-\ell}\cdot k^{k-\ell}\leq  (2k)^{k}\cdot n^{k-1-\ell}.
\]
Necessarily $\ell\geq2$.  From~\eqref{boundonsatltuplesinH1Odd}, it
follows that the number of copies of $C_k^{(3)}$ from $\KK_1$ that
contain a saturated tuple of this form is at most
\begin{equation}\label{subcase 1}
\begin{split}
\sum_{\ell=2}^{k-1} \left( (2k)^{k}\cdot n^{k-1-\ell}\right)\cdot \left(\frac{2\cdot k^\ell\cdot D^{\ell-1}\cdot n^2} {K}\right)
&\leq \frac{(2k)^{2k}\cdot D\cdot n^{k-1}}{K}+\frac{n^{k-1}}{\log n}\\
&\overset{\eqref{ineq:lowerKstarOdd}, \eqref{ineq: lower bound K Odd case}}{\leq} \frac{|\KK_1|}{6},
\end{split}
\end{equation}
where in the first inequality we use that $D=C_2$ and $K=(C_1)^{2^{100k}}$ are constants that depend only on $k$ and $\varepsilon$. In both inequalities we further used that $n$ is sufficiently large relative to $k$, $\varepsilon$, $C_1$, and $C_2$.

We next fix a saturated tuple $\sigma$ such that exactly one of its
hyperedges contains two vertices from $T$.  For any copy of
$C_k^{(3)}$ in $\mathcal K_1$, whose hyperedges are labelled as
$e_1,\ldots,e_k$ according to~\ref{itm: cycle from shadow 2}, the only
hyperedges containing two vertices of $T$ are
$e_{k-1}=w_1z_{k-1}z_k$ and
$e_k=w_{k-1}z_{k-1}z_{k+1}^{(m)}$ for some
$m\in[\lceil C_2/12\rceil]$.  Hence the unique hyperedge of $\sigma$
containing two vertices from $T$ must occupy one of these two
positions.  There are at most two choices for its position and, once
the position has been specified, at most two ways to assign the labels
to the vertices of this hyperedge.

Suppose first that the occupied position is $e_k=w_{k-1}z_{k-1}z_{k+1}^{(m)}$. By condition \ref{itm: cycle from shadow 2}, we have $w_1\in N_{G^{(2)}_{\{w_{k-1}z_{k-1}\}}}(z_{k+1}^{(m)})$, and \eqref{degree S2 auxiliary} implies that the degree of $z_{k+1}^{(m)}$ in the auxiliary graph $G^{(2)}_{\{w_{k-1}z_{k-1}\}}$ is at most $C_2$. Therefore, once the labels $w_{k-1}$, $z_{k-1}$ and $z_{k+1}^{(m)}$ have been specified, there are at most $C_2$ choices for the vertex $w_1$. By \eqref{degree S1 auxiliary} and  \ref{itm: cycle from shadow 2}, the vertex $z_k$ is then uniquely determined as the only neighbor of $w_{k-1}$ in $G^{(1)}_{\{w_{1}z_{k-1}\}}$.

Now suppose that the occupied position is $e_{k-1}=w_1 z_{k-1} z_{k}$. By condition \ref{itm: cycle from shadow 2}, we have $w_{k-1}\in N_{G^{(1)}_{\{w_{1}z_{k-1}\}}}(z_k)$, and \eqref{degree S2 auxiliary} implies that the degree of $z_k$ in the auxiliary graph $G^{(1)}_{\{w_1z_{k-1}\}}$ is at most $2$. Hence, once $w_1$, $z_{k-1}$ and $z_k$ have been specified, there are at most $2$ choices for $w_{k-1}$. Furthermore, since we have $z_{k+1}^{(m)}\in N_{G^{(2)}_{\{w_{k-1} z_{k-1}\}}}(w_1)$  by \ref{itm: cycle from shadow 2}
and $w_1$ has degree $\lceil C_2/4\rceil<C_2$ in $G^{(2)}_{\{w_{k-1}z_{k-1}\}}$ by \eqref{degree S1 auxiliary}, there are at most $C_2$ choices for $z_{k+1}^{(m)}$.

After the vertices corresponding to the labels
$w_1,w_{k-1},z_{k-1},z_k$ and $z_{k+1}^{(m)}$ have been specified,
there are at most $(k-2)^{\ell-1}<k^{\ell-1}$ ways to choose the
positions of the remaining $\ell-1$ hyperedges of $\sigma$ among
$e_1,\ldots,e_{k-2}$. Once these positions are fixed, there are at
most $2^{\ell-1}$ ways to assign their vertices to the corresponding
labels $w_i,w_{i+1}$ and $z_i$.  At most
$\max\{k-2-\ell,0\}$ vertices among $w_2,\ldots,w_{k-2}$ then remain
unspecified, so they can be chosen in at most
$n^{\max\{k-2-\ell,0\}}$ ways. The only remaining unspecified
vertices are the $k-1-\ell$ labels among $z_1,\ldots,z_{k-2}$ that do
not appear in $\sigma$. They have fewer than $k^{k-\ell}$ possible
values in total. These choices determine the copy of $C_k^{(3)}$ in
$\KK_1$ containing $\sigma$. Therefore, the total number of such
copies is at most
\[
2\cdot 2\cdot (2C_2)\cdot k^{\ell-1}\cdot 2^{\ell-1}\cdot
n^{\max\{k-2-\ell,0\}}\cdot k^{k-\ell}
\leq 8\cdot C_2\cdot (2k)^k
n^{\max\{k-2-\ell,0\}}.
\]

From~\eqref{boundonsatltuplesinH1Odd}, it follows that the number of
copies of $C_k^{(3)}$ from $\KK_1$ that contain at least one saturated
tuple of this form is at most
\begin{equation}\label{subcase 2}
\begin{split}
\sum_{\ell=1}^{k-1}\left(8\cdot C_2\cdot (2k)^{k}\cdot n^{\max\{k-2-\ell,0\}}\right)\left(\frac{2\cdot k^\ell\cdot D^{\ell-1}\cdot n^2} {K}\right)&\leq \frac{(2k)^{2k}\cdot C_2\cdot n^{k-1}}{K}+\frac{n^{k-1}}{\log n}\\
&\overset{\eqref{ineq:lowerKstarOdd},\eqref{ineq: lower bound K Odd case}}{\leq} \frac{|\KK_1|}{6}.
\end{split}
\end{equation}

Finally, fix a saturated tuple $\sigma$ such that none of its
hyperedges contains two vertices from $T$, and bound the number of
members of $\KK_1$ containing $\sigma$.

Necessarily $\ell\leq k-2$, because every member of $\KK_1$ has
only $k-2$ hyperedges containing exactly one vertex from $T$.

For any copy of $C_k^{(3)}$ in $\mathcal K_1$, whose hyperedges are labeled as $e_1,\ldots,e_k$ according to \ref{itm: cycle from shadow 2}, no hyperedge of $\sigma$ can occupy the positions corresponding to $e_{k-1}=w_1z_{k-1}z_k$ or $e_k=w_{k-1}z_{k-1}z_{k+1}^{(m)}$, since every hyperedge of $\sigma$ contains exactly one vertex from $T$. Hence, every hyperedge of $\sigma$ must occupy one of the positions corresponding to $e_1,\ldots,e_{k-2}$. There are therefore at most $(k-2)^\ell<k^\ell$ ways to choose these positions. Once they are specified, there are at most $2^\ell$ ways to assign the vertices of the hyperedges of $\sigma$ to the corresponding labels $w_i,w_{i+1}$ and $z_i$. This assignment specifies at least $\ell+1$ vertices corresponding to the labels $w_1,\ldots,w_{k-1}$. Thus, at most $k-1-(\ell+1)=k-2-\ell$ vertices corresponding to the labels $w_1,\ldots,w_{k-1}$ remain unspecified, and they can be chosen in at most $n^{k-2-\ell}$ ways. The vertices corresponding to the labels $z_1,\ldots,z_{k-1}$ that do not belong to $\sigma$ can then be chosen in at most $\left\lfloor\frac{k+1}{2}\right\rfloor^{k-1-\ell}<k^{k-1-\ell}$ ways.

Once the vertices corresponding to $w_1,w_{k-1}$ and $z_{k-1}$ have been specified, the vertex $z_{k+1}^{(m)}$ must be a neighbor of $w_1$ in the auxiliary graph $G^{(2)}_{\{w_{k-1}z_{k-1}\}}$. Hence, by \eqref{degree S1 auxiliary}, there are at most $\lceil C_2/4\rceil\le C_2$ choices for $z_{k+1}^{(m)}$. By \eqref{degree S1 auxiliary}, the vertex $z_k$ is then uniquely determined as the only neighbor of $w_{k-1}$ in $G^{(1)}_{\{w_1 z_{k-1}\}}$. These choices uniquely determine the vertices corresponding to the labels $w_1,\ldots,w_{k-1},z_1,\ldots,z_k,z_{k+1}^{(m)}$, and hence the copy of $C_k^{(3)}$ in $\mathcal K_1$ containing $\sigma$. Therefore, the total number of such copies is at most
\[
k^\ell\cdot 2^\ell\cdot n^{k-2-\ell}\cdot k^{k-1-\ell}\cdot C_2
\leq C_2\cdot (2k)^k\cdot n^{k-2-\ell}.
\]

From~\eqref{boundonsatltuplesinH1Odd}, it follows that the number of
copies of $C_k^{(3)}$ from $\KK_1$ that contain at least one saturated
tuple of this form is at most
\begin{equation}\label{subcase 3}
\begin{split}
\sum_{\ell=1}^{k-2}\left(C_2\cdot (2k)^{k}\cdot n^{k-2-\ell}\right)\left(\frac{2\cdot k^\ell\cdot D^{\ell-1}\cdot n^2} {K}\right)
&\leq \frac{(2k)^{2k}\cdot C_2\cdot n^{k-1}}{K}+\frac{n^{k-1}}{\log n}\\
&\overset{\eqref{ineq:lowerKstarOdd},\eqref{ineq: lower bound K Odd case}}{\leq} \frac{|\KK_1|}{6}.
\end{split}
\end{equation}
Every member of $\KK_1$ has exactly two hyperedges containing two
vertices from $T$.  Therefore, the three possibilities considered
above account for every member of $\KK_1$ containing a saturated
$\ell$-tuple for some $\ell\in[k-1]$.

From \eqref{subcase 1}, \eqref{subcase 2} and \eqref{subcase 3}, it follows that the number of members from $\KK_1$ containing at least one saturated $\ell$-tuple for some $\ell \in [k-1]$ is at most $|\KK_1|/2$. If $\KK$ does not satisfy condition \ref{itm: supersat shadow odd 1}, then, since $k\geq 5$, we also have
\[
|\KK|\leq \frac{D^{k-1}\cdot n^{2}}{K} \le \frac{n^{k-1}}{\log n}\overset{\eqref{ineq:lowerKstarOdd} }{<}\frac{|\KK_1|}{4},
\]
for sufficiently large $n$. It follows that at least $|\KK_1\setminus \KK|-|\KK_1|/2\geq |\KK_1|/4$ members of $\KK_1\setminus \KK$ contain no saturated $\ell$-tuple for any $\ell \in [k-1]$. We may therefore add such a member of $\KK_1\setminus \KK$ to $\KK$ while preserving the condition~\ref{itm: supersat shadow odd 2}, noting that this condition trivially holds when $\ell=k$. Repeating this process until condition~\ref{itm: supersat shadow odd 1} holds, we obtain a collection $\KK$ satisfying both conditions~\ref{itm: supersat shadow odd 1} and~\ref{itm: supersat shadow odd 2}. This completes Case~\ref{lem: supersat odd case 2} and the proof of Lemma~\ref{lem: Supersat Large shadow odd}.

\end{proof}

\section{Proof of the main container theorem}\label{sec:EnumerationProof}

To prove Theorem~\ref{thm:MainContainer}, we start with the collection
of containers supplied by Proposition~\ref{prop:Dense regime}.
Whenever a container has at least
$\bigl(\lfloor(k-1)/2\rfloor+\varepsilon/2\bigr)\binom{n}{r-1}$
hyperedges, we apply Proposition~\ref{prop:kOdd} or
Proposition~\ref{prop:kEven}, then we apply the corresponding
supersaturation lemma and Corollary~\ref{cor:containers-balanced} to
replace it by containers
with fewer hyperedges.  We continue until every container has fewer
hyperedges than
$\bigl(\lfloor(k-1)/2\rfloor+\varepsilon/2\bigr)\binom{n}{r-1}$.

\begin{proof}[Proof of Theorem~\ref{thm:MainContainer}]
It is enough to prove the theorem for $0<\varepsilon<1$: for a larger
value of $\varepsilon$, apply the result with
$\min\{\varepsilon,1/2\}$ in its place.  We therefore assume
$0<\varepsilon<1$.
Put
\[
 N\coloneqq\binom{n}{r-1}.
\]
Proposition~\ref{prop:Dense regime} gives a collection
$\mathcal C_0$ such that every $C_k^{(r)}$-free $r$-graph on $[n]$ is
contained in a member of $\mathcal C_0$,
\begin{equation}\label{eq:initial-general-containers}
 |\mathcal C_0|\leq2^{\varepsilon N/4}
 \quad \text{ and }\quad
 |H|\leq C_0N\quad\text{for every }H\in\mathcal C_0.
\end{equation}

Starting from $\mathcal C_0$, we define a refinement process.  At any
stage, a member of the collection under consideration is called a
\emph{current container}.

Suppose that a current container $H$ has at least
$\bigl(\lfloor(k-1)/2\rfloor+\varepsilon/2\bigr)N$ hyperedges.
If $r=3$ and $k\geq5$ is odd, apply
Proposition~\ref{prop:kOdd}.  According to which of
conditions~\ref{itm: Odd Case linear},
\ref{itm: Odd Case Codegree}, and~\ref{itm: Odd Case Shadow} holds,
use Lemma~\ref{lem: Supersat Large Linear subgraph},
Lemma~\ref{lem: Supersat Large codegrees}, or
Lemma~\ref{lem: Supersat Large shadow odd}, respectively.  If either
$r\geq4$, or $r=3$ and $k\geq4$ is even, apply
Proposition~\ref{prop:kEven} and then use
Lemma~\ref{lem: Supersat Large codegrees} or
Lemma~\ref{lem: Supersat Large shadow even}.  Fix an arbitrary priority
order when more than one condition holds.  In every case we obtain a
collection $\KK$ and parameters $D,K,s$ satisfying
\begin{equation}\label{eq:general-balanced-menu}
 |\KK|\geq\frac{D^{k-1}s}{K},
 \qquad
 \Delta_\ell(\KK)\leq D^{k-\ell}\quad(\ell\in[k]),
 \qquad
 s\geq\frac{\varepsilon}{8}N.
\end{equation}
More explicitly, the possible parameter pairs are
\begin{equation}\label{eq:parameter-menu}
\begin{array}{c|l|c|c}
\text{condition}&\text{source of }\KK&D&K\\ \hline
\ref{itm: Odd Case linear}&\text{large sparse subhypergraph}
 &n^{(k-2)/(k-1)}&\log n\\
\ref{itm: Odd Case Codegree},\ \ref{itm: Even Case Codegrees}
 &\text{large-codegree subhypergraph}&(C_1)^{r-1}&c_{r,k}C_1\\
\ref{itm: Even Case Shadow}&\text{general large shadow}
 &n^{1/(k-1)}&\log n\\
\ref{itm: Odd Case Shadow}&\text{odd $3$-uniform large shadow}
 &C_2&(C_1)^{2^{100k}}.
\end{array}
\end{equation}
Only the rows applicable to the fixed pair $(r,k)$ are used.  In the
large-codegree row, take $s=|H_2|$.  Condition~\ref{itm: Odd Case Codegree}
gives $|H_2|\geq\varepsilon N/8$, while
condition~\ref{itm: Even Case Codegrees} gives the stronger bound
$|H_2|\geq\varepsilon N/4$.  In each of the other rows, the
corresponding supersaturation lemma uses either $s=n^2$ or
$s=n^{r-1}$, and hence $s\geq N$.  Since $0<\varepsilon<1$, the last
bound also implies $s\geq\varepsilon N/8$.  This proves the final
inequality in~\eqref{eq:general-balanced-menu} in every case.

For sufficiently large $n$, the hypotheses of
Corollary~\ref{cor:containers-balanced} hold.  Replace $H$ by the
collection $\mathcal C(H)$ supplied by
Corollary~\ref{cor:containers-balanced}.  Thus every
$C_k^{(r)}$-free subhypergraph of $H$ remains covered, every child of $H$
has at most
\begin{equation}\label{eq:refinement-branching}
 2^{h(2(k-1)/D)|H|}
 \leq2^{C_0h(2(k-1)/D)N}
\end{equation}
possible choices, and every child has at most
\begin{equation}\label{eq:refinement-decrease}
 |H|-\frac{\delta s}{2K}
 \leq |H|-\frac{\delta\varepsilon N}{16K}
\end{equation}
hyperedges, where $\delta=2^{-k(k+1)}$.  Repeat this operation until every
container has fewer than
$\bigl(\lfloor(k-1)/2\rfloor+\varepsilon/2\bigr)N$ hyperedges.

It remains to bound the number of terminal containers.  Fix
$H_0\in\mathcal C_0$ and represent the refinements starting from
$H_0$ by a rooted tree, treating different occurrences of the same
container as distinct vertices.  Its root is $H_0$, its vertices are
all container occurrences obtained from $H_0$, and the children of a container are
the containers produced by one application of
Corollary~\ref{cor:containers-balanced}.  Its leaves are precisely the
terminal containers obtained from $H_0$.  A root-to-leaf chain is the
sequence of containers obtained by starting with $H_0$ and choosing
one child at each refinement until a leaf is reached.
For the fixed pair $(r,k)$, let $\mathcal R$ consist of the three
conditions~\ref{itm: Odd Case linear},
\ref{itm: Odd Case Codegree}, and~\ref{itm: Odd Case Shadow} when
$r=3$ and $k\geq5$ is odd, and of the two
conditions~\ref{itm: Even Case Codegrees} and
\ref{itm: Even Case Shadow} otherwise.  For $\tau\in\mathcal R$, let
$D_\tau,K_\tau$ be the parameters in the row of the table
in~\eqref{eq:parameter-menu} containing $\tau$.

By~\eqref{eq:initial-general-containers}, the container at the root
has at most $C_0N$ hyperedges.  Each refinement using condition
$\tau$ removes at least $\delta\varepsilon N/(16K_\tau)$ hyperedges
by~\eqref{eq:refinement-decrease}.  Consequently, if a root-to-leaf
chain contains $m_\tau$ such refinements, then
\[
 m_\tau\frac{\delta\varepsilon N}{16K_\tau}\leq C_0N,
\]
and hence $m_\tau\leq16C_0K_\tau/(\delta\varepsilon)$.
Thus every root-to-leaf chain contains at most
\begin{equation}\label{eq:number-refinements-row}
 L_\tau\coloneqq\left\lceil
 \frac{16C_0K_\tau}{\delta\varepsilon}
 \right\rceil
\end{equation}
refinements using condition $\tau$, regardless of the order in which
the conditions occur.

Set $L\coloneqq\sum_{\tau\in\mathcal R}L_\tau$.  By recording the
condition used at each refinement along a root-to-leaf chain, we
obtain a sequence of members of $\mathcal R$ of length at most $L$.
The number of such sequences, including the empty sequence, is at
most
\[
 \sum_{j=0}^L|\mathcal R|^j\leq(|\mathcal R|+1)^L.
\]
For each fixed sequence of conditions,
\eqref{eq:refinement-branching} gives at most
\[
\prod_{\tau\in\mathcal R}
 2^{C_0L_\tau h(2(k-1)/D_\tau)N}
\]
possible sequences of chosen children. Hence, for each fixed $H_0\in\mathcal C_0$, the number of terminal
containers produced from $H_0$ is at most
\begin{equation}\label{eq:terminal-containers-one-initial}
 (|\mathcal R|+1)^L\cdot 
 2^{C_0N\sum_{\tau\in\mathcal R}
 L_\tau h(2(k-1)/D_\tau)}.
\end{equation}

We next bound the base-two logarithm of the expression in
\eqref{eq:terminal-containers-one-initial}.
These contributions are
\[
 L\log_2(|\mathcal R|+1)
 \quad\text{and}\quad
 C_0N\sum_{\tau\in\mathcal R}
 L_\tau h\left(\frac{2(k-1)}{D_\tau}\right).
\]
Since $h(x)=O(x\log(1/x))$ as $x\to0$, the first and third rows
of~\eqref{eq:parameter-menu} satisfy, respectively,
\[
 K_\tau h\left(\frac{2(k-1)}{D_\tau}\right)
 =O_k\left(\frac{(\log n)^2}{n^{(k-2)/(k-1)}}\right)
\]
and
\[
 K_\tau h\left(\frac{2(k-1)}{D_\tau}\right)
 =O_k\left(\frac{(\log n)^2}{n^{1/(k-1)}}\right).
\]
Both quantities tend to zero as $n\to\infty$.  For the second row,
\[
 K_\tau h\left(\frac{2(k-1)}{D_\tau}\right)
 =O_{r,k}\left(\frac{\log C_1}{(C_1)^{r-2}}\right),
\]
which can be made arbitrarily small by choosing $C_1$.  With $C_1$
fixed, the final row satisfies
\[
 K_\tau h\left(\frac{2(k-1)}{D_\tau}\right)
 =O_k\left(\frac{(C_1)^{2^{100k}}\log C_2}{C_2}\right),
\]
which can be made arbitrarily small by choosing $C_2$.

Since $K_\tau\geq1$,~\eqref{eq:number-refinements-row} gives
$L_\tau\leq(16C_0/(\delta\varepsilon)+1)K_\tau$.  The set
$\mathcal R$ is finite, and the coefficient
$16C_0/(\delta\varepsilon)+1$ is independent of $C_1,C_2,n$.
We may therefore choose
$C_1$, then $C_2$, and finally $n$ so that
\begin{equation}\label{eq:all-container-costs}
 C_0\left(\frac{16C_0}{\delta\varepsilon}+1\right)
 \sum_{\tau\in\mathcal R}K_\tau
 h\left(\frac{2(k-1)}{D_\tau}\right)
 \leq\frac\varepsilon8.
\end{equation}
Moreover, $L=O_{\varepsilon,r,k,C_1,C_2}(\log n)$, so, after increasing
$n$ if necessary,
$L\log_2(|\mathcal R|+1)\leq\varepsilon N/8$.
The two terms in the base-two logarithm of
\eqref{eq:terminal-containers-one-initial} are therefore at most
$\varepsilon N/8$ each, so that this logarithm is at most
$\varepsilon N/4$.

Together with~\eqref{eq:initial-general-containers}, this shows that
the collection $\mathcal C$ of terminal containers satisfies
$|\mathcal C|\leq2^{\varepsilon N/2}$.  The covering property in
condition~\ref{itm: Main container 1} is
preserved at every refinement, while the stopping rule gives
$|H|<\bigl(\lfloor(k-1)/2\rfloor+\varepsilon/2\bigr)N$ for every
$H\in\mathcal C$, which is condition~\ref{itm: Main container 2}.
The bound on $|\mathcal C|$ is condition~\ref{itm: Main container 3}.
\end{proof}

\section{Enumeration of 3-graphs without a linear triangle}
\label{sec:C33-enumeration}

Recall that $T=C_3^{(3)}$.  Thus $T$ has six distinct vertices
$a,b,c,x,y,z$ and hyperedges
\[
 \{a,b,x\},\qquad \{b,c,y\},\qquad \{c,a,z\}.
\]
Equivalently, $T$ consists of three hyperedges whose pairwise
intersections are three distinct singletons.
Subsection~\ref{subsec:C33-tools} contains all external results used
in this section.  We now prove the remaining case of
the enumeration problem.

\begin{theorem}\label{thm:C33-enumeration}
As $n\to\infty$,
\[
 \mathrm{forb}_3(n,T)
 =2^{(1+o(1))\mathrm{ex}_3(n,T)}.
\]
\end{theorem}

Let us first outline the structure of the proof.
The construction in Section~\ref{sec:concluding-remarks} shows that the
parameters needed to apply Corollary~\ref{cor:containers-balanced}
cannot satisfy $K\log D/D=o(1)$ in this case.  We therefore count the
$T$-free $3$-graphs directly.

Fix a $T$-free $3$-graph $H$ on $[n]$.  In
Subsection~\ref{subsec:C33-low-codegree}, we choose a codegree
threshold $q=q(n)$ tending sufficiently slowly to infinity.  Let $P$
be the graph consisting of the pairs having codegree at least $q$ in
$H$, and let $H_0$ consist of the hyperedges of $H$ containing no pair
from $P$.  Lemmas~\ref{lem:C33-linear-count}
and~\ref{lem:C33-low-part-count} show that there are $2^{o(n^2)}$
possibilities for each of $H_0$ and $P$, and that $|P|=o(n^2)$.

In Subsection~\ref{subsec:C33-high-codegree}, for each $z\in[n]$ we
define the link graph of $z$ in $P$, $F_z$, consisting of the pairs $xy\in P$ for which
$xyz$ is a hyperedge of $H$.  Let $V_z$ denote the set of vertices
incident to an edge of $F_z$.  Then
Lemma~\ref{lem:C33-induced-link-solution} shows that $P$ is
triangle-free and $F_z=P[V_z]$.  We encode the sets $V_z$ by a digraph
$D$ on $[n]$, putting $x\to z$ precisely when $x\in V_z$.  The
$3$-graph $H$ is then determined by $H_0$, $P$, and $D$.

Let $D_0$ be obtained from $D$ by deleting every arc whose underlying
pair belongs to $P$.  By
Lemma~\ref{lem:C33-mixed-triangle-high-pair}, every directed triangle
in $D_0$ has both arcs present on each of its three underlying pairs.
Lemma~\ref{lem:C33-mixed-digraph-count}, proved in
Subsection~\ref{subsec:C33-digraph-count}, shows that there are at most
$2^{\binom n2+o(n^2)}$ possibilities for $D_0$.  The deleted arcs have
at most $4^{|P|}=2^{o(n^2)}$ possible configurations.  Combining these
bounds gives
\[
 \mathrm{forb}_3(n,T)\leq 2^{\binom n2+o(n^2)}.
\]
Finally, Subsection~\ref{subsec:C33-final-proof} combines this upper
bound with the subhypergraphs of a full star and
Theorem~\ref{thm:C33-FF-extremal} to prove
Theorem~\ref{thm:C33-enumeration}.  In
Section~\ref{sec:proof-main-theorem}, we then combine this result with
Theorem~\ref{thm:MainContainer} to prove
Theorem~\ref{thm:numberOfCycleFree}.

\subsection{The subhypergraph consisting of hyperedges whose pairs have small codegree}
\label{subsec:C33-low-codegree}

For a $3$-graph $H$ and a pair $xy\in\binom{V(H)}2$, define
\[
 N_H(xy)\coloneqq\{z\in V(H)\setminus\{x,y\}:xyz\in H\},
 \qquad d_H(xy)\coloneqq|N_H(xy)|.
\]
The \emph{$2$-shadow} of a $3$-graph $J$ is the graph
\[
 \partial J\coloneqq\{xy\in\tbinom{V(J)}2:xy\subseteq e
                 \text{ for some }e\in J\}.
\]
The following counting result is a well-known consequence of the Ruzsa-Szemer\'edi $(6,3)$-theorem (Theorem~\ref{thm:C33-RS}).  We include its proof for completeness.
\begin{lemma}\label{lem:C33-linear-count}
The number of linear $T$-free $3$-graphs on $[n]$ is $2^{o(n^2)}$.
\end{lemma}

\begin{proof}
Let $J$ be a linear $T$-free $3$-graph on $[n]$.  By
Theorem~\ref{thm:C33-RS},
\begin{equation}\label{eq:C33-linear-size}
 |J|\leq\eta(n)n^2,
\end{equation}
where $\eta(n)\to0$.  Moreover, $J$ is determined by its $2$-shadow.
Indeed, every hyperedge of $J$ gives a triangle in $\partial J$.  Conversely,
suppose that $xy,yz,zx\in\partial J$ but $xyz\notin J$.  The three
pairs are then contained in three distinct hyperedges of $J$: if one hyperedge
contained two of these pairs, it would be $xyz$.  By linearity, the
three supplying hyperedges intersect pairwise in the distinct vertices
$x,y,z$ and hence form a copy of $T$, a contradiction.  Thus the
triangles in $\partial J$ are precisely the shadows of the hyperedges of
$J$, so $\partial J$ determines $J$.

Every hyperedge of $J$ contributes three distinct pairs to $\partial J$;
linearity gives
\[
 |\partial J|=3|J|\leq3\eta(n)n^2.
\]
Consequently, the number of possible shadows, and hence the number of
possible $J$, is at most
\[
 \sum_{j\leq3\eta(n)n^2}
   \binom{\binom n2}{j}
 =2^{o(n^2)}.
\]
To justify the last equality, put
$N\coloneqq\binom n2$ and $\alpha_n\coloneqq3\eta(n)n^2/N$.  Then
$\alpha_n\to0$, so $\alpha_n\leq1/2$ for all sufficiently large $n$.
The entropy bound
$\sum_{j\leq \alpha N}\binom Nj\leq2^{h(\alpha)N}$ for
$0\leq\alpha\leq1/2$, where
$h(\alpha)=-\alpha\log_2\alpha-(1-\alpha)\log_2(1-\alpha)$ and
$h(\alpha)\to0$ as $\alpha\to0$, therefore gives the claimed
$2^{o(n^2)}$ bound.
\end{proof}

Let $\mathcal L_n$ be the family counted in
Lemma~\ref{lem:C33-linear-count}, and put
\[
 \lambda_n\coloneqq\frac{\log_2|\mathcal L_n|}{n^2}.
\]
Lemma~\ref{lem:C33-linear-count} gives $\lambda_n\to0$.  Since
$\mathcal L_n$ contains both the empty $3$-graph and every one-hyperedge
$3$-graph,
\begin{equation}\label{eq:C33-lambda-lower}
 \lambda_n
 \geq \frac{\log_2(1+\binom n3)}{n^2}
 =\Omega\left(\frac{\log n}{n^2}\right).
\end{equation}
In particular, $\lambda_n>0$ for $n\geq3$.  Define $q(n)$ by
\begin{equation}\label{eq:C33-q-choice}
 q(n)\coloneqq\max\left\{4,
          \left\lfloor\lambda_n^{-1/2}\right\rfloor\right\}.
\end{equation}
For the remainder of the proof, we abbreviate $q(n)$ by $q$.
Since $\lambda_n\to0$, equation~\eqref{eq:C33-q-choice} gives, for
all sufficiently large $n$,
$q=\lfloor\lambda_n^{-1/2}\rfloor$.  In particular, $q\to\infty$ and
$q\leq\lambda_n^{-1/2}$.
Moreover,
\[
 q\log_2|\mathcal L_n|
 =q\lambda_n n^2
 \leq\lambda_n^{1/2}n^2
 =o(n^2).
\]
Therefore
\begin{equation}\label{eq:C33-q-properties}
 q\to\infty
 \qquad\text{and}\qquad
 q\log_2|\mathcal L_n|=o(n^2).
\end{equation}
Equations~\eqref{eq:C33-q-choice} and~\eqref{eq:C33-lambda-lower}
also give
\begin{equation}\label{eq:C33-q-upper}
 q=O\left(\frac{n}{\sqrt{\log n}}\right).
\end{equation}

Fix a $T$-free $3$-graph $H$ on $[n]$.  Define
\begin{equation}\label{eq:C33-high-pairs}
 P(H)\coloneqq\{xy\in\tbinom{[n]}2:d_H(xy)\geq q\}.
\end{equation}
We abbreviate $P(H)$ by $P$ and regard it as a graph on $[n]$.  Let
\begin{equation}\label{eq:C33-low-part}
 H_0\coloneqq\{e\in H:\tbinom e2\cap P=\varnothing\};
\end{equation}
thus $H_0$ consists of the hyperedges of $H$ containing no pair from $P$.

\begin{lemma}\label{lem:C33-low-part-count}
As $H$ ranges over the $T$-free $3$-graphs on $[n]$, there are
$2^{o(n^2)}$ possibilities for $H_0$ and $2^{o(n^2)}$ possibilities
for $P$.
\end{lemma}

\begin{proof}
Define an auxiliary graph $R$ whose vertex set is $H_0$, joining two
distinct hyperedges of $H_0$ when they contain a common pair.  If $e\in H_0$
and $xy\in\binom e2$, then $xy\notin P$, so
$d_H(xy)\leq q-1$.  At most $q-2$ other hyperedges of $H_0$ contain $xy$.
Since $e$ contains three pairs,
\[
 \Delta(R)\leq3(q-2).
\]
Here $\Delta(R)$ denotes the maximum degree of $R$.
A greedy proper colouring of $R$ therefore uses at most $3q$ colours.
Each colour class is a linear subhypergraph of $H_0$, and it is $T$-free
because it is a subhypergraph of $H$.  Thus every possible $H_0$ is the union
of an ordered collection of at most $3q$ members of $\mathcal L_n$.
By~\eqref{eq:C33-q-properties}, the number of possibilities for $H_0$
is at most
\[
 |\mathcal L_n|^{3q}=2^{o(n^2)}.
\]

For all sufficiently large $n$, Theorem~\ref{thm:C33-FF-extremal}
gives $|H|\leq\binom{n-1}{2}$.  Double-counting the incidences between
pairs and hyperedges gives
\[
 q|P|
 \leq\sum_{xy\in\binom{[n]}2}d_H(xy)
 =3|H|<3n^2,
\]
and hence $|P|<3n^2/q=o(n^2)$.  Set
$N\coloneqq\binom n2$ and $m\coloneqq\lceil3n^2/q\rceil$.  By
\eqref{eq:C33-q-upper}, we have $n^2/q\to\infty$, and hence
$m=\Theta(n^2/q)$.  For all sufficiently large $n$, we also have
$m\leq N/2$, and the inequality
\[
 \sum_{j=0}^{m}\binom Nj
 \leq\left(\frac{eN}{m}\right)^m
 =2^{O(n^2\log q/q)}
 =2^{o(n^2)}.
\]
This is an upper bound for the number of possibilities for $P$.
\end{proof}

\subsection{The subhypergraph consisting of hyperedges containing a pair of large codegree}
\label{subsec:C33-high-codegree}

For each $z\in[n]$, define a graph $F_z$ as follows.  Its edge set and
vertex set are
\begin{equation}\label{eq:C33-high-link}
 E(F_z)\coloneqq\{xy\in P:xyz\in H\},
 \qquad
 V_z\coloneqq\bigcup_{xy\in E(F_z)}\{x,y\}.
\end{equation}
Thus $F_z=(V_z,E(F_z))$ has no isolated vertices.

\begin{lemma}\label{lem:C33-induced-link-solution}
The graph $P$ is triangle-free and
\begin{equation}\label{eq:C33-induced-link}
 F_z=P[V_z]
 \qquad\text{for every }z\in[n].
\end{equation}
\end{lemma}

\begin{proof}
Suppose first that $xy,xz,yz\in P$.  Each of the three sets
\[
 N_H(xy)\setminus\{z\},\qquad
 N_H(xz)\setminus\{y\},\qquad
 N_H(yz)\setminus\{x\}
\]
has size at least $q-1\geq3$.  Hall's condition therefore holds for
these three sets: the union of any nonempty subcollection has size at
least three, which is at least the number of sets in the subcollection.
Choose distinct representatives $a,b,c$ from the three sets,
respectively.  The hyperedges $xya,xzb,yzc$ then have pairwise intersections
$x,y,z$ and form a copy of $T$, a contradiction.  Hence $P$ is
triangle-free.

By~\eqref{eq:C33-high-link}, every edge of $F_z$ belongs to $P$ and has
both endpoints in $V_z$, so $F_z\subseteq P[V_z]$.  Suppose that
$xy\in P[V_z]\setminus F_z$.  Since $x,y\in V_z$, choose vertices
$a,b$ such that $xa,yb\in E(F_z)$.  If $a\ne b$, then
$d_H(xy)\geq q\geq4$ supplies
\[
 w\in N_H(xy)\setminus\{z,a,b\}.
\]
The hyperedges $zxa,zyb,xyw$ have pairwise intersections $z,x,y$ and form a
copy of $T$, a contradiction.  Therefore the two chosen neighbours
satisfy $a=b$ for every choice of an $F_z$-neighbour of $x$ and an
$F_z$-neighbour of $y$.  Fixing one choice on each side shows that $x$
and $y$ have the same unique neighbour $u$ in $F_z$.  In particular,
$xu,yu\in P$.  Together with $xy\in P$, these three pairs form a
triangle in $P$, contradicting the first conclusion of
Lemma~\ref{lem:C33-induced-link-solution}, which states that $P$ is
triangle-free.  Thus
$P[V_z]\subseteq F_z$, proving
\eqref{eq:C33-induced-link}.
\end{proof}

Define a digraph $D(H)$ on $[n]$, abbreviated by $D$, by
\begin{equation}\label{eq:C33-support-digraph}
 x\to z\text{ in }D
 \quad\Longleftrightarrow\quad
 x\in V_z.
\end{equation}

For distinct $x,y,z$ with $xy\in P$, we claim that
\begin{equation}\label{eq:C33-reconstruction}
 xyz\in H
 \quad\Longleftrightarrow\quad
 x\to z\text{ and }y\to z\text{ in }D.
\end{equation}
Indeed, if $xyz\in H$, then $xy\in E(F_z)$ by
\eqref{eq:C33-high-link}, so $x,y\in V_z$ and hence both arcs are
present in $D$.  Conversely, if $x\to z$ and $y\to z$ are arcs of
$D$, then $x,y\in V_z$ by~\eqref{eq:C33-support-digraph}; since
$xy\in P$, equation~\eqref{eq:C33-induced-link} gives
$xy\in E(F_z)$, and~\eqref{eq:C33-high-link} therefore gives
$xyz\in H$.

Consequently, $P$, $D$, and $H_0$ determine $H$: equation
\eqref{eq:C33-low-part} determines the hyperedges containing no pair of $P$,
and~\eqref{eq:C33-reconstruction} determines every remaining hyperedge.

An unordered pair $xy$ is a \emph{digon} of a digraph if both
$x\to y$ and $y\to x$ are arcs.  We call a directed triangle
$x\to y\to z\to x$ \emph{mixed} if at least one of its three
underlying pairs is not a digon.
In particular, a directed triangle having exactly the three arcs
$x\to y$, $y\to z$, and $z\to x$ is mixed.

\begin{lemma}\label{lem:C33-mixed-triangle-high-pair}
Every mixed directed triangle in $D$ contains an underlying pair from
$P$.
\end{lemma}

\begin{proof}
Suppose that
\begin{equation}\label{eq:C33-directed-cycle}
 x\to y,\qquad y\to z,\qquad z\to x
\end{equation}
are arcs of $D$, and suppose that
\begin{equation}\label{eq:C33-cycle-pairs-low}
 xy,yz,zx\notin P.
\end{equation}
We prove that all three reverse arcs are present, so the directed
triangle in~\eqref{eq:C33-directed-cycle} is not mixed.

By cyclic symmetry, it is enough to prove that $y\to x$ is present.
Suppose otherwise.  Then $y\notin V_x$.  From
\eqref{eq:C33-directed-cycle} and~\eqref{eq:C33-support-digraph}, choose
vertices $a,b,c$ such that
\begin{equation}\label{eq:C33-three-witnesses}
 xa\in E(F_y),\qquad
 yb\in E(F_z),\qquad
 zc\in E(F_x).
\end{equation}
It follows from~\eqref{eq:C33-high-link} that
\begin{equation}\label{eq:C33-three-supported-edges}
 xay,ybz,zcx\in H
 \qquad\text{and}\qquad
 xa,yb,zc\in P.
\end{equation}
Equations~\eqref{eq:C33-cycle-pairs-low} and
\eqref{eq:C33-three-supported-edges}, together with $y\notin V_x$,
imply
\begin{equation}\label{eq:C33-witness-inequalities}
 a\ne z,\qquad b\ne x,\qquad c\ne y.
\end{equation}
We also have $a\ne b$.  Indeed, if $a=b$, then
$ya=yb\in P$ and $xya\in H$.  Hence $ya\in E(F_x)$ by
\eqref{eq:C33-high-link}, so $y\in V_x$, contrary to the assumption
$y\notin V_x$.

If $b\ne c$ and $c\ne a$, then the three hyperedges in
\eqref{eq:C33-three-supported-edges} have pairwise intersections
$y,z,x$, respectively, and form a copy of $T$.  Consequently,
\begin{equation}\label{eq:C33-two-collisions}
 b=c\quad\text{or}\quad c=a.
\end{equation}

Suppose first that $b=c$.  For every
$w\in N_H(yb)\setminus\{x,a,z\}$, the hyperedges
\[
 ybw,\qquad xay,\qquad zbx
\]
have pairwise intersections $y,b,x$ and form a copy of $T$.  Since $H$
is $T$-free,
\[
 N_H(yb)\subseteq\{x,a,z\}.
\]
This gives $d_H(yb)\leq3$, contradicting $yb\in P$ and $q\geq4$.

Suppose next that $c=a$.  For every
$w\in N_H(za)\setminus\{x,y,b\}$, the hyperedges
\[
 zaw,\qquad xay,\qquad ybz
\]
have pairwise intersections $a,y,z$ and form a copy of $T$.  Therefore
$N_H(za)\subseteq\{x,y,b\}$ and $d_H(za)\leq3$, contradicting
$za=zc\in P$.  Both alternatives in~\eqref{eq:C33-two-collisions} are
impossible, so $y\to x$ is present.  Repeating the proof of
$y\to x$ with $(x,y,z)$ replaced first by $(y,z,x)$ and then by
$(z,x,y)$ gives $z\to y$ and $x\to z$.
\end{proof}

Let $D_0$ be obtained from $D$ by deleting both possible arcs on every
pair belonging to $P$.  Lemma~\ref{lem:C33-mixed-triangle-high-pair}
gives
\begin{equation}\label{eq:C33-cleaned-property}
 \text{Every directed triangle in $D_0$ has all three underlying pairs as digons.}
\end{equation}
Indeed, a directed triangle in $D_0$ contains no pair from $P$; if one
of its pairs were not a digon, it would contradict
Lemma~\ref{lem:C33-mixed-triangle-high-pair}.

\subsection{Counting the cleaned support digraphs}
\label{subsec:C33-digraph-count}

Let $\Gamma_n$ be the family of digraphs on $[n]$ satisfying
\eqref{eq:C33-cleaned-property}, and put
$\Gamma\coloneqq(\Gamma_n)_{n\geq1}$.  Every $\Gamma_n$ is nonempty,
because it contains the empty digraph on $[n]$.  For a digraph $G$ and
$U\subseteq V(G)$, define its induced subdigraph by
\[
 G[U]\coloneqq\bigl(U,A(G)\cap(U\times U)\bigr).
\]
If $G$ satisfies~\eqref{eq:C33-cleaned-property}, then $G[U]$ also
satisfies~\eqref{eq:C33-cleaned-property}.  Indeed, every directed
triangle in $G[U]$ is a directed triangle in $G$.  Since $G$ satisfies
\eqref{eq:C33-cleaned-property}, the three reverse arcs of this
directed triangle belong to $G$.  Their endpoints lie in $U$, so the
reverse arcs also belong to $G[U]$.  Thus all three underlying pairs
of the directed triangle are digons in $G[U]$.
If $\phi:[m]\to[n]$ is an increasing injection, meaning that
$\phi(i)<\phi(j)$ whenever $i<j$, then relabelling $G[\phi([m])]$ by
$\phi^{-1}$ gives a member of $\Gamma_m$.  Hence $\Gamma$ is
order-hereditary.

For a pair $xy$ with $x<y$, its four possible states are
\begin{equation}\label{eq:C33-four-states}
 0,\qquad x\to y,\qquad y\to x,
 \qquad x\leftrightarrow y,
\end{equation}
where $0$ means that neither arc is present and
$x\leftrightarrow y$ means that both arcs are present.  Thus digraphs
on $[n]$ are in bijection with four-colourings of
$\binom{[n]}2$, as required by
Theorem~\ref{thm:C33-multicolour-counting}.

\begin{lemma}\label{lem:C33-mixed-digraph-count}
We have
\[
 |\Gamma_n|\leq2^{\binom n2+o(n^2)}.
\]
\end{lemma}

\begin{proof}
Let $\Phi$ be a $\Gamma_n$-valid template.  Thus, for every pair
$xy$, the palette $\Phi_{xy}$ is a nonempty subset of the four states
in~\eqref{eq:C33-four-states}, and every choice of one state from each
palette gives a member of $\Gamma_n$.  It will be convenient to
measure the entropy of $\Phi$ using base-$2$ logarithms.  Define
\begin{equation}\label{eq:C33-binary-entropy}
 \operatorname{Ent}_2(\Phi)
 \coloneqq\sum_{xy\in\binom{[n]}2}\log_2|\Phi_{xy}|.
\end{equation}

Form a digraph $J(\Phi)$, abbreviated by $J$, as follows.  For every pair $xy$, take the
union of the arcs appearing in the states belonging to $\Phi_{xy}$,
and use this union as the set of arcs of $J$ on the pair $xy$.
Equivalently, $x\to y$ belongs to $J$ if and only if at least one state
in $\Phi_{xy}$ contains $x\to y$.  In particular, if the digon state
$x\leftrightarrow y$ belongs to $\Phi_{xy}$, then both $x\to y$ and
$y\to x$ belong to $J$, so $xy$ is a digon of $J$.

An arc $x\to y$ of $J$ is called \emph{loose} if $\Phi_{xy}$ contains
the single-arc state $x\to y$, regardless of whether it also contains
any of the other three states $0$, $y\to x$, or
$x\leftrightarrow y$.  Thus $x\to y$ and $y\to x$ may both be loose.
Let $L(\Phi)$ be the set of all loose arcs, and abbreviate it by $L$.

\begin{claim}\label{claim:C33-no-loose-triangle}
The digraph $J$ contains no directed triangle using a loose arc.
\end{claim}

\begin{claimproof}
Suppose that $x\to y\to z\to x$ is a directed triangle in $J$ and
that $x\to y$ is loose.  At $xy$, choose the state in which $x\to y$
is present and $y\to x$ is absent.  At the other two pairs, choose
states containing $y\to z$ and $z\to x$, respectively.  At each
remaining pair, choose any member of its nonempty palette.  The
resulting realization of $\Phi$ contains a mixed directed triangle,
contrary to the validity of $\Phi$.
\end{claimproof}

For every pair $xy\in\binom{[n]}2$, define
\[
 \ell_{xy}\coloneqq\bigl|\{x\to y,y\to x\}\cap L\bigr|.
\]
Thus $\ell_{xy}\in\{0,1,2\}$.  Let $B$ be the set of pairs $xy$ for
which $J$ contains both arcs $x\to y$ and $y\to x$ and
$\ell_{xy}\leq1$, and define
\[
 b_{xy}\coloneqq
 \begin{cases}
 1,&xy\in B,\\
 0,&xy\notin B.
 \end{cases}
\]
\begin{claim}\label{claim:C33-local-palette-entropy}
For every $xy\in\binom{[n]}2$,
\begin{equation}\label{eq:C33-local-palette-entropy}
 \log_2|\Phi_{xy}|\leq\ell_{xy}+b_{xy}.
\end{equation}
\end{claim}

\begin{claimproof}
Choose an arbitrary $xy\in\binom{[n]}2$.  If $\ell_{xy}=0$, then
$\Phi_{xy}\subseteq\{0,x\leftrightarrow y\}$.  If both of these
states belong to $\Phi_{xy}$, then $xy\in B$ and $b_{xy}=1$; otherwise
$|\Phi_{xy}|=1$.  In either case,
$\log_2|\Phi_{xy}|\leq b_{xy}$.

Suppose next that $\ell_{xy}=1$.  If $xy\notin B$, then $J$ does not
contain both arcs on $xy$, and hence $|\Phi_{xy}|\leq2$.  If
$xy\in B$, then $|\Phi_{xy}|\leq3$.  Thus
$\log_2|\Phi_{xy}|\leq\ell_{xy}+b_{xy}$ in both cases.  Finally, if
$\ell_{xy}=2$, then $b_{xy}=0$ and $|\Phi_{xy}|\leq4$, which again
gives~\eqref{eq:C33-local-palette-entropy}.
\end{claimproof}

Fix the natural order on $[n]$ and construct a digraph $K$ as follows.
First, include every arc in $L$.  If $xy\in B$ and $\ell_{xy}=1$, also
include the other arc on the pair $xy$.  If $xy\in B$ and
$\ell_{xy}=0$, include only the arc from the smaller endpoint to the
larger endpoint.  Hence
\begin{equation}\label{eq:C33-K-size}
 a(K)=\sum_{xy\in\binom{[n]}2}\ell_{xy}+|B|.
\end{equation}

\begin{claim}\label{claim:C33-K-triangle-free}
The digraph $K$ contains no directed triangle.
\end{claim}

\begin{claimproof}
Suppose, for a contradiction, that $K$ contains a directed triangle
on the vertices $x,y,z$.  If this triangle uses a loose arc, then it
is also a directed triangle in $J$ using a loose arc, contradicting
Claim~\ref{claim:C33-no-loose-triangle}.

It remains to consider the case in which the directed triangle uses
no loose arc.  Each of its three arcs was then added on a pair in
$B$.  Consequently, $J$ contains both arcs on each of the pairs
$xy$, $yz$, and $zx$.  Suppose that one of these pairs has one loose
arc; denote that arc by $u\to v$, and let $w$ be the third vertex.
Since the pairs $vw$ and $wu$ are bidirected in $J$, the arcs
$v\to w$ and $w\to u$ belong to $J$.  Hence
$u\to v\to w\to u$ is a directed triangle in $J$ using a loose arc,
again contradicting Claim~\ref{claim:C33-no-loose-triangle}.

We conclude that $\ell_{xy}=\ell_{yz}=\ell_{zx}=0$.  On each of
these three pairs, the construction of $K$ includes only the arc from
the smaller endpoint to the larger endpoint.  These three arcs cannot
form a directed triangle: the largest of $x,y,z$ has no outgoing arc
among the three.  This final contradiction proves the claim.
\end{claimproof}

Applying Theorem~\ref{thm:C33-Brown-Harary} to $K$, and then summing
\eqref{eq:C33-local-palette-entropy}, gives
\begin{equation}\label{eq:C33-template-entropy-bound}
 \operatorname{Ent}_2(\Phi)
 \leq\sum_{xy}\ell_{xy}+|B|
 =a(K)
 \leq\left\lfloor\frac{n^2}{2}\right\rfloor.
\end{equation}

The base-$4$ entropy in
Theorem~\ref{thm:C33-multicolour-counting} is
$\operatorname{Ent}_4(\Phi)=\operatorname{Ent}_2(\Phi)/2$.  Therefore
\eqref{eq:C33-template-entropy-bound} implies
\[
 \pi(\Gamma)
 \leq
 \lim_{n\to\infty}
 \frac{\lfloor n^2/2\rfloor}{2\binom n2}
 =\frac12.
\]
Theorem~\ref{thm:C33-multicolour-counting} now gives
\[
 |\Gamma_n|
 \leq4^{(1/2+o(1))\binom n2}
 =2^{\binom n2+o(n^2)},
\]
as required.
\end{proof}

\subsection{Proof of Theorem~\ref{thm:C33-enumeration}}
\label{subsec:C33-final-proof}

\begin{proof}[Proof of Theorem~\ref{thm:C33-enumeration}]
A full star is $T$-free and has $\binom{n-1}{2}$ hyperedges.  Taking
all of its subhypergraphs gives
\begin{equation}\label{eq:C33-enumeration-lower}
 \mathrm{forb}_3(n,T)\geq2^{\binom{n-1}{2}}.
\end{equation}

For the upper bound, let $q=q(n)$ be chosen as
in~\eqref{eq:C33-q-choice}, and fix a $T$-free $3$-graph $H$ on
$[n]$.  Let $P=P(H)$ and $H_0$ be defined
by~\eqref{eq:C33-high-pairs} and~\eqref{eq:C33-low-part},
respectively.  For each $z\in[n]$, let $V_z$ be defined
by~\eqref{eq:C33-high-link}, and recall that the digraph $D=D(H)$ is
defined by putting $x\to z$ in $D$ exactly when $x\in V_z$.  Let
$D_0$ be obtained from $D$ by deleting every arc whose underlying
unordered pair belongs to $P$.

Lemma~\ref{lem:C33-low-part-count} gives $2^{o(n^2)}$ possibilities
for each of $H_0$ and $P$.  Moreover, for all sufficiently large $n$,
$q|P|\leq\sum_{xy\in\binom{[n]}2}d_H(xy)=3|H|
\leq3\binom{n-1}{2}$ by
Theorem~\ref{thm:C33-FF-extremal}.  Since $q\to\infty$
by~\eqref{eq:C33-q-properties}, we have $|P|=o(n^2)$.  Once $P$ has
been fixed, there are consequently at most
$4^{|P|}=2^{o(n^2)}$ possibilities for the states of $D$ on the pairs
belonging to $P$.

By~\eqref{eq:C33-cleaned-property}, $D_0\in\Gamma_n$.  Hence
Lemma~\ref{lem:C33-mixed-digraph-count} gives at most
$2^{\binom n2+o(n^2)}$ possibilities for $D_0$.  The digraph $D_0$,
together with the states of $D$ on the pairs in $P$, determines $D$.
Finally, $H_0$, $P$, and $D$ determine $H$ by
\eqref{eq:C33-low-part} and~\eqref{eq:C33-reconstruction}.
Multiplying the four bounds gives
\begin{equation}\label{eq:C33-enumeration-upper}
 \mathrm{forb}_3(n,T)
 \leq2^{\binom n2+o(n^2)}
 =2^{\binom{n-1}{2}+o(n^2)}.
\end{equation}

Theorem~\ref{thm:C33-FF-extremal} gives
$\mathrm{ex}_3(n,T)=\binom{n-1}{2}$ for all sufficiently large $n$.
Together with~\eqref{eq:C33-enumeration-lower}
and~\eqref{eq:C33-enumeration-upper}, this proves the theorem.
\end{proof}

\section{Proof of the main theorem}\label{sec:proof-main-theorem}

\begin{proof}[Proof of Theorem~\ref{thm:numberOfCycleFree}]
If $(r,k)=(3,3)$, the result follows from
Theorem~\ref{thm:C33-enumeration}.  Assume henceforth that
$(r,k)\neq(3,3)$.

Fix $\varepsilon>0$ and apply Theorem~\ref{thm:MainContainer}.  Writing
$N=\binom{n}{r-1}$, every $C_k^{(r)}$-free $r$-graph is a subhypergraph of
one of at most $2^{\varepsilon N/2}$ containers, each having at most
$(\lfloor(k-1)/2\rfloor+\varepsilon/2)N$ hyperedges.  Therefore
\[
 \operatorname{forb}_r(n,C_k^{(r)})
 \leq2^{(\lfloor(k-1)/2\rfloor+\varepsilon)N}.
\]
Since $\varepsilon>0$ is arbitrary and
\[
 \operatorname{ex}_r(n,C_k^{(r)})
 =\left(\left\lfloor\frac{k-1}{2}\right\rfloor+o(1)\right)N,
\]
this is the required upper bound.  The lower bound follows by taking
all subhypergraphs of an extremal $C_k^{(r)}$-free $r$-graph.
\end{proof}
 
\section{Concluding remarks}\label{sec:concluding-remarks}

Theorem~\ref{thm:numberOfCycleFree} completely answers
Question~\ref{conj:enumeration result} in the affirmative.  The proof of the final case
$(r,k)=(3,3)$ does not proceed through balanced supersaturation.
Instead, it shows that a structural encoding followed by an entropy
argument can recover the correct exponent.

The analysis of pair-codegrees in the proof of Frankl and
F\"uredi~\cite{FranklFuredi1987} motivates separating pairs according
to their codegrees, as we do in Section~\ref{sec:C33-enumeration}.
The following construction shows, however, that the existence of many
linear triangles does not by itself yield a balanced collection satisfying
the degree bounds required by
Corollary~\ref{cor:containers-balanced}.  Fix
$0<\varepsilon<1$ and an integer
$\alpha>4(1+\varepsilon)$.  Let $n$ be divisible by $4\alpha$, and
partition the vertex set into sets $X$ and $Y$ of size $n/2$.
Partition $X$ into $n/(4\alpha)$ sets of size $2\alpha$, and on each
such set, place a copy of $K_{\alpha,\alpha}$.  Let $B$ be the union of
these pairwise vertex-disjoint copies.  Thus $B$ is an
$\alpha$-regular bipartite graph on $X$ and
\[
 |B|=\frac{\alpha n}{4}.
\]
Define a $3$-graph $H$ by declaring $xx'y$ to be a hyperedge precisely
when $xx'\in B$ and $y\in Y$.  Its shadow is the union of $B$ and the
complete bipartite graph between $X$ and $Y$.
We have
\[
 |H|=|B||Y|=\frac{\alpha n^2}{8}
 >(1+\varepsilon)\binom n2.
\]

For a linear triangle with hyperedges $abu$, $bcv$, and $caw$, call the
triple $abc$ its \emph{core}.  The three pairs contained in the core
belong to the shadow of $H$.  Every triangle in this shadow has two
vertices in $X$ joined by an edge of $B$ and one vertex in $Y$, and is
therefore a hyperedge of $H$.  Consequently, the core of every linear
triangle in $H$ is itself a hyperedge of $H$.

Fix a core $xx'y$, where $xx'\in B$ and $y\in Y$.  The three
hyperedges of a linear triangle with this core must have the form
\[
 xx'y_0,\qquad xay,\qquad x'by,
\]
where $N_B(v)$ denotes the set of neighbors of a vertex $v$ in $B$,
and
\[
 y_0\in Y\setminus\{y\},\qquad
 a\in N_B(x)\setminus\{x'\},\qquad
 b\in N_B(x')\setminus\{x\}.
\]
Conversely, every such choice gives a linear triangle.  Indeed, $a$
and $b$ lie in different vertex classes of the copy of
$K_{\alpha,\alpha}$ containing $xx'$, and hence $a\ne b$.  Thus each
core belongs to exactly $(n/2-1)(\alpha-1)^2$ linear triangles, and
the total number of linear triangles in $H$ is
\[
 |H|\left(\frac n2-1\right)(\alpha-1)^2
 =\frac{\alpha(\alpha-1)^2n^2(n-2)}{16}.
\]

For each core $xx'y$ and each choice of $a$ and $b$ above, place the
pair of hyperedges $\{xay,x'by\}$ in $\mathcal Q$.  Then
\[
 |\mathcal Q|\leq(\alpha-1)^2|H|=O_\alpha(n^2),
\]
and every linear triangle in $H$ contains a member of $\mathcal Q$.

Let $\mathcal K$ be any collection of linear triangles in $H$.
Recalling that $\Delta_2(\mathcal K)$ is the maximum number of members
of $\mathcal K$ containing a fixed pair of hyperedges, we obtain
\begin{equation}\label{eq:C33-balanced-obstruction}
 |\mathcal K|
 \leq |\mathcal Q|\Delta_2(\mathcal K)
 =O_{\alpha,\varepsilon}\bigl(n^2\Delta_2(\mathcal K)\bigr).
\end{equation}
In the natural application of
Corollary~\ref{cor:containers-balanced}, one would take
$s=\Theta(n^2)$ and seek parameters $D,K$ satisfying
\[
 |\mathcal K|\geq\frac{D^2s}{K}
 \qquad\text{and}\qquad
 \Delta_2(\mathcal K)\leq D.
\]
Equation~\eqref{eq:C33-balanced-obstruction} would then imply
$D/K=O_{\alpha,\varepsilon}(1)$.  Hence $K\log D/D$ cannot tend to
zero as $D\to\infty$, as required for the total container cost to be
$2^{o(n^2)}$.  This construction does not rule out every possible
form of balanced supersaturation, but it shows why the form needed for
Corollary~\ref{cor:containers-balanced} cannot be applied uniformly in
the case $(r,k)=(3,3)$ and motivates the different encoding argument
used in Section~\ref{sec:C33-enumeration}.  It would be interesting to
determine whether similar encodings could be applied for other enumeration
problems for which balanced supersaturation together with
Corollary~\ref{cor:containers-balanced} does not give the correct
exponent.

\subsection*{Acknowledgments}

The authors used ChatGPT 5.6 Sol to assist with literature searches and to proofread and polish portions of this manuscript. It also helped prepare an initial draft of the proof of Theorem~\ref{thm:C33-enumeration}, following a sketch of the proof suggested by the authors; the proof was subsequently rewritten by the authors. All mathematical ideas are due to the authors, who take full responsibility for the manuscript. 

We are grateful to Abhishek Dhawan, Oliver Janzer, Tao Jiang, Yongjin Lee and Sam Spiro for several helpful discussions on topics related to this paper.

{\setlength{\emergencystretch}{1em}%
\printbibliography}

\end{document}